\documentclass{amsart}

\usepackage[english, activeacute]{babel}
\usepackage{amsmath,amsthm,amsxtra}
\usepackage{fix-cm}
\usepackage{graphics}
\usepackage{graphicx}	
\usepackage{fancyhdr}
\usepackage{float}
\usepackage{anysize}
\usepackage{thmtools}
\usepackage{amssymb,mathrsfs}
\usepackage{arydshln}
\usepackage{enumerate}
\usepackage{mathtools}
\usepackage{epsfig}
\usepackage{amssymb}
\usepackage{float}
\usepackage{latexsym}
\usepackage{amsfonts}
\usepackage{hyperref}
\usepackage{float}
\usepackage{xfrac}
\usepackage[capitalize]{cleveref}

\usepackage{tikz}
\usetikzlibrary{automata,positioning,matrix}

\newtheorem{theorem}{Theorem}
\newtheorem*{theorem*}{Theorem}
\newtheorem{corollary}[theorem]{Corollary}

\newtheorem{lemma}[theorem]{Lemma}
\newtheorem{proposition}[theorem]{Proposition}

\newtheorem*{problem}{Problem}
\newtheorem*{example}{Example}
\theoremstyle{definition}
\newtheorem{definition}[theorem]{Definition}
\newtheorem{remark}[theorem]{Remark}

\newtheorem{claim}{Claim}
\newcommand{\N}{\mathbb{N}}
\newcommand{\Z}{\mathbb{Z}}

\def\rr{\mathcal{R}}
\def\F{\mathcal{F}}

\def\T{\mathbb{T}}
\def\rr{\mathcal{R}}
\def \QQ {{\bf Q}}
\def \RP {{\bf RP}}

\def \KK {{\bf K}}
\def \id {\textrm{id}}

\title[Directional dynamical cubes]{Directional dynamical cubes for minimal $\Z^{d}$-systems}

\author{Christopher Cabezas}
\address{Departamento de Ingenier\'ia Matem\'atica, Universidad de Chile, Av. Beauchef 851, Santiago, Chile}
\email{ccabezas@dim.uchile.cl}
\thanks{This work was funded by Proyecto/Grant PIA AFB-170001 and Math-AmSud DYSTIL}

\author{Sebasti\'an Donoso}
\address{Instituto de Ciencias de la Ingenier\'ia, Universidad de O'Higgings, Av. Lib. Bernardo O'Higgins 611, Rancagua, Chile}
\email{sebastian.donoso@uoh.cl}
\thanks{The second author is also supported by Fondecyt Iniciaci\'on Grant 11160061}

\author{Alejandro Maass}
\address{Departamento de Ingenier\'{\i}a Matem\'atica and Centro de Modelamiento Matem\'atico, Universidad de Chile \& UMI-CNRS 2807, Av. Beauchef 851, Santiago, Chile}
\email{amaass@dim.uchile.cl}

\begin{document}
\begin{abstract}
We introduce the notions of \emph{directional dynamical cubes} and \emph{directional regionally proximal relation} defined via these cubes for a minimal $\Z^d$-system $(X,T_1,\ldots,T_d)$. We study the structural properties of systems that satisfy the so called \emph{unique closing parallelepiped property} and we characterize them in several ways. In the distal case, we build the maximal factor of a $\Z^d$-system 
$(X,T_1,\ldots,T_d)$ that satisfies this property by taking the quotient with respect to the directional regionally proximal relation. Finally, we completely describe distal $\Z^d$-systems that enjoy the unique closing parallelepiped property and provide explicit examples. 
\end{abstract}
\maketitle 
	
\section{Introduction}
The study of {\em cubical structures} or {\em cubespaces} was initiated by Host and Kra in \cite{host2005nonconventional} as a fundamental tool to show the $L^2$ convergence of Furstenberg multiple averages for a single transformation. Later, Host, Kra and Maass in \cite{host2010nilsequences} introduced  cube structures in topological dynamics given by a single homeomorphism proving a structure theorem for transitive systems with the \emph{unique completion property} (this condition is referred in this work as the {\em unique closing parallelepiped property}). The main result is that this property characterizes inverse limits of minimal rotations on nilmanifolds. 

Soon after, the notion of cubespaces became a major theme of research in the development of {\em higher order Fourier analysis}, starting from Antol\'in-Camarena and Szegedys's axiomatic formulation of them \cite{Szegedy2010}. Further description of such theory was given by Candela in \cite{candela2017nil,candela2017nilcompact} and another point of view was developed by Gutman, Manners and Varj\'u  in \cite{gutman2016structurei,gutman2016structureii,gutman2016structureiii}. In all these works a key feature is the study of the unique closing parallelepiped property and the conclusion (roughly speaking) is that a compact space whose cubes have the unique closing parallelepiped property has a {\em nilstructure}. 

Interestingly, the study of {\em dynamical cubes}, \emph{i.e.}, those defined in \cite{host2005nonconventional} and \cite{host2010nilsequences}, has reached an independent interest, mainly because of his various applications that range from the construction of nilfactors, recognizing patterns in number theory, the study of recurrence in topological dynamics and even the study of pointwise convergence of averages in ergodic theory \cite{Donoso-Sun-2016,donoso_sun_2018-2Ptw,donoso-sun-2018-Pointw,gutman2016structureiii, huang2016nil}. Also, extensions to general group actions have been recently carried out in \cite{Glasner20181004}. 
  
In this work, we consider a slightly different (but related) notion that we call {\em directional dynamical cubes}. Such notion first appeared (without a name) in the work of Host \cite{host2009ergodic}, where he gave a proof of Tao's $L^2$ convergence of Furstenberg multiple averages for commuting transformations \cite{tao2008norm}. Motivated by those works, the study of directional cubes in topological dynamics  started in \cite{donoso2014dynamical}, where the authors introduced for a minimal $\mathbb{Z}^{2}$-system 
$(X,T_1,T_2)$ a space of cubes denoted by $\QQ_{T_1,T_2}(X)$ and proved a structure theorem for systems with the unique closing parallelepiped property (here $X$ is a compact metric space and $T_1,T_2:X\to X$ are commuting homeomorphisms). Additionally, the authors 
introduced the notion of $(T_1,T_2)$-regionally proximal relation $\mathcal{R}_{T_1,T_2}(X)$ associated to the cube structure $\QQ_{T_1,T_2}(X)$. This is a strong variant of the classical regionally proximal relation of order one $\RP^{[1]}(X,\left\langle T_1,T_2\right\rangle)$ for $\Z^{2}$-systems. In the distal case they proved that $\mathcal{R}_{T_1,T_2}(X)$ is an equivalence relation and that 
$(X/{\mathcal{R}_{T_1,T_2}(X)}, T_1,T_2)$ is the maximal factor with the unique closing parallelepiped property.

The purpose of this article is to extend the program of research started in \cite{donoso2014dynamical} to arbitrary $\Z^{d}$-systems, \emph{i.e.}, $d$ commuting homeomorphisms $T_{1},\ldots,T_{d}\colon X\to X$ of a compact metric space $X$ with $d\geq 2$. 

We define the \emph{space of directional cubes } $\QQ_{T_{1},\ldots,T_{d}}(X)$ as the closure in $X^{2^{d}}$ of the points $$(T_{1}^{n_{1}\varepsilon_{1}}\cdots T_{d}^{n_{d}\varepsilon_{d}}x)_{\varepsilon\in \{0,1\}^{d}},$$ where $x\in X$ and $\textbf{n}=(n_{1},\ldots,n_{d})\in \Z^{d}$. Here we point out that since the transformations $T_j$ could be different, the space of directional cubes $\QQ_{T_{1},\ldots,T_{d}}(X)$ has much less symmetries than its relative ones defined for a single homeomorphism. 

As in previous works we are interested in $\Z^d$-systems with the unique closing parallelepiped property, as a way of understanding the topological counterpart of the characteristic factors introduced by Host in \cite{host2009ergodic} for an arbitrary number of  commuting transformations. We also introduce the 
$(T_{1},\ldots,T_{d})$ -\emph{regionally proximal relation} $\mathcal{R}_{T_{1},\ldots,T_{d}}(X)$ of 
$(X,T_{1},\ldots,T_{d})$ associated to these cube structures in order to give a complete description of systems with the unique closing parallelepiped property and define the maximal factor with this property for any $\Z^{d}$-system. We defer the precise definitions of the relation to Section \ref{sec:RegRelation}. 

We show the following structure theorem for minimal distal systems.

\begin{theorem}
\label{StructThm}
Let $(X,T_{1},\ldots,T_{d})$ be a minimal distal $\Z^{d}$-system. The following statements are equivalent:
\begin{enumerate}[(1)]
\item $(X,T_{1},\ldots,T_{d})$ has the unique closing parallelepiped property, \emph{i.e.}, if $\textbf{x},\textbf{y}\in \QQ_{T_{1},\ldots,T_{d}}(X)$ have $2^{d}-1$ coordinates in common, then $\textbf{x}=\textbf{y}$.
\item $\mathcal{R}_{T_{1},\ldots,T_{d}}(X)=\Delta_{X}$.
\item The structure of $(X,T_{1},\ldots,T_{d})$ can be described as follows: (i) it is a factor of a minimal distal $\Z^{d}$-system $(Y,T_{1},\ldots,T_{d})$ which is a joining of $\Z^{d}$-systems $(Y_{1},T_{1},\ldots,T_{d}),\ldots,(Y_{d},T_{1},\ldots,T_{d})$, where for each $i\in \{1,\ldots,d\}$ the action of $T_i$ on $Y_i$ is the identity; (ii) for each $i,j\in \{1,\ldots,d\}$, $i< j$, there exists a $\Z^{d}$-system
$(Y_{i,j},T_{1},\ldots,T_{d})$ which is a common factor of $(Y_{i},T_{1},\ldots,T_{d})$ and $(Y_{j},T_{1},\ldots,T_{d})$ and where $T_i$ and $T_j$ act as the identity; and (iii) $Y$ is jointly relatively independent with respect to the systems $\left ( (Y_{i,j},T_{1},\ldots,T_{d}): \ i,j \in \{1,\ldots,d \}, \ i<j \right )$.
\end{enumerate}
\end{theorem}

We will see in Section \ref{sec:ProofThm} that $Y$, the $Y_j$'s and the $Y_{i,j}$'s are given explicitly by the cube structure. 

When $d=2$ we have that $(Y_{1,2},T_{1},T_{2})$ is the trivial system (because it is minimal and the actions are identities). Then, relative independence implies that $(Y,T_1,T_2)$ can be seen as the direct product of the systems $(Y_{1},T_1,T_2)$ and $(Y_{2},T_1,T_2)$, so Theorem \ref{StructThm} generalizes the structure theorem proved by Donoso and Sun for minimal distal systems in \cite{donoso2014dynamical} when $d=2$. 
Also, the structure in the previous theorem can be thought of as a topological version of \emph{pleasant extensions} developed in a measure theoretical setting by Austin in \cite{austin2010multiple}.

We also prove,

\begin{theorem}
Let $(X,T_1,\ldots,T_d)$ be a minimal distal $\Z^d$-system. Then, the relation $\rr_{T_{1},\ldots,T_{d}}(X)$ is an equivalence relation and $(X/\rr_{T_{1},\ldots,T_{d}}(X),T_1,\ldots,T_d)$ is the {\em maximal factor }of $(X,T_1,\ldots,T_d)$ with the unique closing parallelepiped property. This means that $(X/\rr_{T_{1},\ldots,T_{d}}(X),T_1,\ldots,T_d)$ has the unique closing parallelepiped property and that any factor of 
$(X,T_1,\ldots,T_d)$ that enjoys that property is necessarily a factor of $(X/\rr_{T_{1},\ldots,T_{d}}(X),T_1,\ldots,T_d)$.
\end{theorem}

Finally we develop interesting applications of the directional dynamical cubes to study recurrence properties of systems with the unique closing parallelepiped property. It turns out that this property can be characterized studying the sets of return times of a $\Z^d$-system. 

 In the previous results, we do not know if the assumption of distality is essential or not. The main feature we use about distal systems is that they enjoy the \emph{gluing property} (we start its discussion in Definition \ref{definition:gluing}) and most of proofs only use this property. We ask then
	
\begin{problem}
How much of our theory can be extended for general minimal $\Z^d$-systems? In particular, is the unique closing parallelepiped property preserved under factor maps? We show in this article that this is true for factor maps between minimal distal systems, but we do not know if this holds for general minimal $\Z^d$-systems.  
\end{problem}

\subsection{Organization of the paper}
In Section \ref{sec:preliminaries} we give some basic material used throughout the paper.
In Section \ref{sec:DirecCubes} we present directional dynamical cubes for a $\Z^{d}$-system $(X,T_{1},\ldots,T_{d})$, introduce the main terminology and give their general properties. Next, in Section  \ref{sec:RegRelation} we introduce the associated $(T_{1},\ldots,T_{d})$-regionally proximal relation and establish some general results. In Section \ref{sec:characterizing} we give a characterization of $\Z^d$-systems with the unique closing parallelepiped property in the distal case and in Section \ref{sec:Structure} we show the structure of a minimal distal system with the unique closing parallelepiped property. Section \ref{sec:ProofThm} is devoted to the proof of Theorem \ref{StructThm}. Then, in Section \ref{sec:ReturnTime} we study the sets of return times for minimal distal systems with the unique closing parallelepiped property and give a theorem that characterizes these systems using their return time sets. 
In Section \ref{sec:Examples} we provide a family of explicit examples of minimal distal $\Z^{d}$-systems with the unique closing parallelepiped property. 

\subsection*{Acknowledgement} We thank Wenbo Sun for many helpful discussions. We also thank the anonymous referee for valuable remarks.

\section{Preliminaries}\label{sec:preliminaries}
	
\subsection{Basic notions from topological dynamics} A \emph{topological dynamical system} is a pair $(X,G)$, where $X$ is a compact metric space and $G$ is a group of homeomorphisms of the space $X$ into itself. We also refer to $X$ as a $G$-system. We use $\rho_{X}(\cdot,\cdot)$ to denote the metric of $X$ and we let $\Delta_{X}=\{(x,x)\colon x \in X\}$ denote the diagonal of $X\times X$. 
	
If $d\geq 1$ is an integer and $T_{1},\ldots,T_{d}:X\to X$ are $d$ commuting homeomorphisms of $X$, we write $(X,T_{1},\ldots,T_{d})$ to denote the topological dynamical system $(X,\langle T_{1},\ldots,T_{d} \rangle)$, where $\langle T_{1},\ldots,T_{d} \rangle$ is the group spanned by $T_{1},\ldots,T_{d}$. 
Throughout this paper those systems are called $\Z^d$-systems.
	
A \emph{factor map} between the dynamical systems $(Y,G)$ and $(X,G)$ is an onto and continuous map $\pi: Y \to X$ such that $\pi \circ g = g \circ \pi$ for every $g\in G$. We say that $(Y,G)$ is an \emph{extension} of $(X,G)$ or that $(X,G)$ is a \emph{factor} of $(Y,G)$. When $\pi$ is bijective we say that $\pi$ is an isomorphism and that $(Y,G)$ and $(X,G)$ are isomorphic or conjugate.
	
Let $(X_{1},G), \ldots, (X_{k},G)$ be $k$ topological dynamical systems. A \emph{joining} between them is a closed subset ${Z\subseteq X_{1}\times \cdots \times X_{k}}$ which is invariant under the diagonal action $g\times \cdots \times g$ ($k$ times) for all $g\in G$ and projects onto each factor.
	
Suppose the systems $(X_{1},G),\ldots,(X_{k},G)$ are extensions of a common system $(W,G)$ and for $i\in \{1,\ldots,k\}$ denote by $\pi_{i}\colon X_{i}\to W$ the associated factor map. We say that a joining $Z \subseteq X_{1}\times \cdots \times X_{k}$ is \emph{relatively independent} with respect to $W$ if for every $(x_{1},\ldots,x_{k})\in Z$, for every $i \in \{ 1,\ldots,k\}$ and for every 
$x_{i}'\in X_{i}$ with $\pi_{i}(x_{i}')=\pi_{i}(x_{i})$ we have
$$(x_{1},\ldots,x_{i-1},x_{i}',x_{i+1},\ldots,x_{k})\in Z.$$
That is, we can freely change coordinates of points in the joining when the corresponding projection to the factor $W$ coincide.
	
Given a topological dynamical system $(X,G)$ and a point $x\in X$ we denote by ${\mathcal{O}_{G}(x)=\{gx\colon g \in G\}}$ the orbit of $x$. We say that $(X,G)$ is \emph{transitive} if there exists a point in $X$ whose orbit is dense. We say that $(X,G)$ is \emph{minimal} if the orbit of any point is dense in $X$. A system $(X,G)$ is \emph{pointwise almost periodic} if for any $x\in X$ the system $(\overline{\mathcal{O}_{G}(x)},G)$ is minimal.

A topological dynamical system $(X,G)$ is {\it distal} if for every $x,y \in X$ such that $x\neq y$ we have $$\inf_{g \in G} d(gx,gy)>0.$$
Distal systems have many interesting properties that we will use in the sequel (see \cite{auslander1988minimal}, chapters 5 and 7). We recall some of them:
\begin{theorem} 
\label{distal}
\quad
\begin{enumerate}
\item The Cartesian product of distal systems is distal.
\item Distality is preserved by taking factors and subsystems.
\item A distal system is minimal if and only if it is transitive, \emph{i.e.}, a distal system is pointwise almost periodic.
\item If $(X,G)$ is distal and $G'$ is a subgroup of $G$, then $(X,G')$ is distal.
\item Factor maps between distal systems are open maps.
\end{enumerate}
\end{theorem}

If $(X,G)$ is a topological dynamical system, we can see $G$ as a subset of $X^{X}$. We define the \emph{enveloping semigroup} $E(X,G)$ of $(X,G)$ as the closure (for the pointwise convergence) of $G$ in $X^{X}$. We have that $E(X,G)$ is compact and Hausdorff, usually non-metrizable. Algebraic properties of the enveloping semigroup have a precise translation into dynamical properties of the system. For example, a topological dynamical system is distal if and only if its enveloping semigroup is a group.

A \emph{left ideal} in a enveloping semigroup $E(X,G)$ is a non-empty subset $I\subseteq E(X,G)$ such that $E(X,G)I$ $\subseteq I$. A \emph{minimal left ideal} is one which does not properly contain a left ideal. An \emph{idempotent} in a enveloping semigroup $E(X,G)$ is an element $u \in E(X,G)$ such that $u^{2}=u$. We denote by $J(E(X,G))$ the set of idempotents in $E(X,G)$. We can introduce a \emph{quasi-order} $<$ on the set $J(E(X,G))$ by defining $v<u$ if and only if $vu=v$. If $v<u$ and $u<v$ we say that $u$ and $v$ are \emph{equivalent} and we write $u\sim v$. An idempotent $u\in J(E(X,G))$ is \emph{minimal} if whenever $v\in J(E(X,G))$ and $v<u$ then $u<v$ (or $u\sim v$). 

We recall some results about the enveloping semigroup that we use later.

\begin{theorem} 
\label{thm:basics enveloping}
Let $(Y,G)$ and $(X,G)$ be topological dynamical systems.
\begin{enumerate}
\item If $\pi:Y\to X$ is a factor map, then there exists a unique continuous semigroup homomorphism 
${\theta: E(Y,G)\to E(X,G)}$ such that $\pi(py)=\theta(p)\pi(y)$ for all $y\in Y$ and $p \in E(Y,G)$. 
If $u \in E(Y,G)$ is a minimal idempotent then $\theta(u) \in E(X,G)$ is a minimal idempotent too.
\item $x, x' \in X$ are proximal if and only if there exists $p\in E(X,G)$ such that $px=px'$.
\item $x \in X$ is an almost periodic point (\emph{i.e.}, its orbit is minimal) if and only if there exists a minimal idempotent $u \in E(X,G)$ such that $u x=x$.
\end{enumerate} 
\end{theorem}
	
\subsection{Cubes and faces} \label{sec:cubesandfaces}
For an integer $d\geq 1$ we write $[d]=\{1,\ldots,d\}$. We view $\{0,1\}^{d}$ in one of two ways, either as a sequence $\varepsilon=\varepsilon_{1}\ldots\varepsilon_{d}$ of 0's and 1's written without commas or parentheses, or as a subset of $[d]$. A subset $\varepsilon \subseteq [d]$ corresponds to the sequence 
$\varepsilon_{1}\ldots\varepsilon_{d}\in \{0,1\}^{d}$ such that for each $i\in [d]$, $i\in \varepsilon$ if and only if $\varepsilon_{i}=1$, \emph{i.e.}, seen as a vector it is the sum of the canonical vectors indexed by 
the set $\varepsilon$.

For any set $X$ we write $X^{[d]}=X^{2^{d}}$. A point $\textbf{x} \in X^{[d]}$ can be described in one of two equivalent ways, depending on the context and convenience:
$$\textbf{x}=(x_{\varepsilon}\colon \varepsilon \in \{0,1\}^{d})=(x_{\varepsilon}\colon \varepsilon \subseteq [d]).$$
For $x\in X$ we write $x^{[d]}=(x,\ldots,x)\in X^{[d]}$ and the diagonal of $X^{[d]}$ is denoted by $\Delta^{[d]}_X=\{x^{[d]}\colon x \in X\}$.
	
A point $\textbf{x}\in X^{[d]}$ can be decomposed as $\textbf{x}=(\textbf{x}',\textbf{x}'')$ with $\textbf{x}',\textbf{x}''\in X^{[d-1]}$, where $\textbf{x}'=(x_{\eta0}\colon \eta\in \{0,1\}^{d-1})$ and $\textbf{x}''=(x_{\eta1}\colon \eta \in \{0,1\}^{d-1})$. We also isolate the first coordinate writing $X_{*}^{[d]}=X^{2^{d}-1}$ and a point $\textbf{x}\in X^{[d]}$ as $\textbf{x}=(x_\emptyset,\textbf{x}_{*})$, where $x_\emptyset \in X$ and $\textbf{x}_{*}=(x_{\varepsilon}\colon \varepsilon\subseteq [d],\ \varepsilon \neq \emptyset)\in X_{*}^{[d]}$. 
	
For $r\leq d$, let $J\subseteq [d]$ with $|J|=d-r$ and $\xi \in \{0,1\}^{d-r}$. A \emph{face of dimension $r$} of $\textbf{x}$ is an element of the form $(x_{\varepsilon}\colon \varepsilon\in \{0,1\}^{d}, \varepsilon_{J}=\xi) \in X^{[r]}$, where $\varepsilon_{J}=(\varepsilon_{i}\colon i \in J)$.
	
Let $(X,T_{1},\ldots,T_{d})$ be a $\Z^d$-system. For $j \in [d]$, the \emph{j-th face transformation} $T_{j}^{[d]}\colon X^{[d]}\to X^{[d]}$ is defined for every $\textbf{x}\in X^{[d]}$ and every $\varepsilon\subseteq [d]$ by:
$$(T_{j}^{[d]}\textbf{x})_\varepsilon=\left\{\begin{array}{ll}
T_{j}x_{\varepsilon} & \text{if}\ j\in \varepsilon,\\
x_{\varepsilon} & \text{if}\ j\notin \varepsilon.
\end{array}\right.$$
The \emph{face group of dimension d} is the group $\mathcal{F}_{T_{1},\ldots,T_{d}}$ of transformations of $X^{[d]}$ generated by the face transformations. Let $G$ be the group spanned by $T_{1},\ldots,T_{d}$ and set $G_{[d]}^{\Delta}=\{g^{[d]}\colon g \in G\}$. We denote by $\mathcal{G}_{T_{1},\ldots,T_{d}}$ the subgroup of $G^{[d]}$ generated by $\mathcal{F}_{T_{1},\ldots, T_{d}}$ and $G_{[d]}^{\Delta}$.

\section{Directional dynamical cubes for $\Z^d$-systems}\label{sec:DirecCubes}
\label{sec:CubeDefinitions}	
In this section we present the notion of directional dynamical cubes for $\Z^{d}$-systems $(X,T_{1},\ldots,T_{d})$ together with its main properties. This is a generalization of the dynamical cubes introduced by  Donoso and  Sun in \cite{donoso2014dynamical} when $d=2$. 
	
\begin{definition}
\label{def:directioncubes}
Let $(X,T_{1},\ldots, T_{d})$ be a $\Z^{d}$-system. The set of \emph{directional dynamical cubes} associated to $(X,T_{1},\ldots, T_{d})$ is defined by,
$$
\QQ_{T_{1},\ldots,T_{d}}(X)=\overline{\left\{(T_{1}^{n_{1}\varepsilon_{1}}\cdots T_{d}^{n_{d}\varepsilon_{d}}x)_{\varepsilon\in \{0,1\}^{d}}\colon x\in X, \textbf{n}=(n_{1},\ldots,n_{d})\in \Z^{d}\right\}}\subseteq X^{[d]}.
$$
Adittionally, given $x_{0}\in X$ we consider the following restriction of theses cubes to $X_{*}^{[d]}$,
$$
\KK_{T_{1},\ldots,T_{d}}^{x_{0}}(X)=\overline{\left\{(T_{1}^{n_{1}\varepsilon_{1}}\ldots T_{d}^{n_{d}\varepsilon_{d}}x_{0})_{\varepsilon\in \{0,1\}^{d}\setminus \{\vec{0}\}}\colon \textbf{n}=(n_{1},\ldots,n_{d})\in \Z^{d}\right\}} \subseteq X_{*}^{[d]}.
$$
\end{definition}

Given $\{j_{1},\ldots,j_{k}\}\subseteq [d]$ and $x_{0}\in X$, since $(X,T_{j_1},\ldots,T_{j_k})$ is a $\Z^k$-system, then we will often consider the set of cubes $\QQ_{T_{j_{1}},\ldots,T_{j_{k}}}(X) \subseteq X^{[k]}$ and $\KK_{T_{j_{1}},\ldots,T_{j_k}}^{x_{0}}(X) \subseteq X_{*}^{[k]}$. For example, if $k=3$ and $j_1,j_2,j_3 \in [d]$ then $\QQ_{T_{j_1},T_{j_2},T_{j_3}}(X)$ is the closure of the set of points 
$$
(x,T_{j_1}^{n_1}x,T_{j_2}^{n_2}x,T_{j_1}^{n_1}T_{j_2}^{n_2}x,
T_{j_3}^{n_3}x,T_{j_1}^{n_1}T_{j_3}^{n_3}x,T_{j_2}^{n_2}T_{j_3}^{n_3}x,
T_{j_1}^{n_1}T_{j_2}^{n_2}T_{j_3}^{n_3}x),
$$
for $x\in X$ and $n_1,n_2,n_3 \in \Z$.

Remark that the generators of $\Z^d$ are given with a precise order. A different choice of this order produces a different set of dynamical cubes. This is the reason why we call them directional cubes. 

\subsection{Basic structural properties} 
 
\begin{proposition}
\label{prop1}
Let $(X,T_{1},\ldots, T_{d})$ be a $\Z^{d}$-system and take $k\in [d]$. We have,
\begin{enumerate}[(1)]
\item \label{prop1:1} $x^{[d]}\in \QQ_{T_{1},\ldots,T_{d}}(X)$ for every $x\in X$.
\item \label{prop1:2} For all $x,y\in X$ and for all $j\in [d]$, $(x,y)\in \QQ_{T_{j}}(X)$ if and only if $(y,x)\in \QQ_{T_{j}}(X)$.
\item \label{prop1:3} The set $\QQ_{T_{1},\ldots,T_{d}}(X)$ is invariant under $\mathcal{G}_{T_{1},\ldots,T_{d}}$ and 
$\KK_{T_{1},\ldots,T_{d}}^{x_{0}}(X)$ is invariant under $\F_{T_{1},\ldots,T_{d}}^{x_{0}}$ \footnote{For convenience,  we let $\F_{T_{1},\ldots,T_{d}}^{x_{0}}$ denote the restriction of $\F_{T_{1},\ldots,T_{d}}$ to $\KK_{T_{1},\ldots,T_{d}}^{x_{0}}(X)$}. Therefore  $(\QQ_{T_{1},\ldots,T_{d}}(X),$  $\mathcal{G}_{T_{1},\ldots,T_{d}})$ and $(\KK_{T_{1},\ldots,T_{d}}^{x_{0}}(X),\F_{T_{1},\ldots,T_{d}}^{x_{0}})$ are topological dynamical systems. 
\item \label{prop1:4} (Projection) Let $\textbf{x}\in \QQ_{T_{1},\ldots,T_{d}}(X)$, $\{j_{1},\ldots,j_{k}\}\subseteq [d]$ and $\xi \in \{0,1\}^{d-k}$. Then, 
$$(x_{\varepsilon}\colon \varepsilon \in \{0,1\}^{d},\ \varepsilon_{[d]\setminus\{j_{1},\ldots,j_{k}\}}=\xi)\in \QQ_{T_{j_{1}},\ldots,T_{j_{k}}}(X).$$
\item \label{prop1:5} (Duplication) Let $\textbf{x}=(x_{\eta}\colon \eta\in \{0,1\}^{k})\in \QQ_{T_{j_{1}},\ldots,T_{j_{k}}}(X)$ with $\{j_{1},\ldots,j_{k}\}\subseteq [d]$. If $\textbf{y}\in X^{[d]}$ is defined such that for every $\varepsilon \in \{0,1\}^{d}$, 
$$y_{\varepsilon}=x_{\eta} \Longleftrightarrow\ \forall \ell \in \{1,\ldots,k\},\ \varepsilon_{j_{\ell}}=\eta_{\ell},$$
\noindent then $\textbf{y}\in \QQ_{T_{1},\ldots,T_{d}}(X)$.
\end{enumerate}
\end{proposition}
	
\begin{proof} Properties \eqref{prop1:1} to \eqref{prop1:3} follow directly from definitions. We only prove properties \eqref{prop1:4} and \eqref{prop1:5}.
For property \eqref{prop1:4} let $\textbf{x}\in \QQ_{T_{1},\ldots,T_{d}}(X)$, $\{j_{1},\ldots,j_{k}\}\subseteq [d]$ and $\xi \in \{0,1\}^{d-k}$. By definition, there exist  $(x^{i})_{i\in \N}\subseteq X$ and $(\textbf{n}(i))_{i\in \N}\subseteq \Z^{d}$ such that
$$\forall \varepsilon \in \{0,1\}^{d},\quad x_{\varepsilon}=\lim\limits_{i\to\infty} T_{1}^{n_{1}(i)\varepsilon_{1}}\cdots T_{d}^{n_{d}(i)\varepsilon_{d}}x^{i}.$$
Now, consider $\eta\in \{0,1\}^{d}$ such that
$$\eta_{\ell}=\left\{\begin{array}{ll}
0 & \ell \in \{j_{1},\ldots,j_{k}\},\\
\xi_{\ell} & \ell \in [d]\setminus\{j_{1},\ldots,j_{k}\},
\end{array} \right.$$
and set $y^{i}=T_1^{n_1(i)\eta_1}\cdots T_d^{n_d(i)\eta_d}x^{i}$. Then, for 
$\varepsilon \in \{0,1\}^{d}$ such that $\varepsilon_{[d]\setminus\{j_{1},\ldots,j_{k}\}}=\xi$ we have
$$\begin{array}{ll}
x_{\varepsilon} & =\lim\limits_{i\to\infty} T_{1}^{n_{1}(i)\varepsilon_{1}}\cdots T_{d}^{n_{d}(i)\varepsilon_{d}}x^{i},\\
& = \lim\limits_{i \to \infty} T_{j_{1}}^{n_{j_{1}}(i)\varepsilon_{j_{1}}}\cdots T_{j_{k}}^{n_{j_{k}}(i)\varepsilon_{j_{k}}}y^{i}.
\end{array}$$
This implies, by definition, that $(x_{\varepsilon}\colon \varepsilon \in \{0,1\}^{d}, \varepsilon_{[d]\setminus\{j_{1},\ldots,j_{k}\}}=\xi)\in \QQ_{T_{j_{1}},\ldots,T_{j_{k}}}(X)$.

Now we prove property \eqref{prop1:5}. Let $\textbf{x}=(x_{\eta}\colon \eta\in \{0,1\}^{k})\in \QQ_{T_{j_{1}},\ldots,T_{j_{k}}}(X)$ with $\{j_{1},\ldots,j_{k}\}\subseteq [d]$. By definition, there exist $(x^{i})_{i\in \N}\subseteq X$ and $(\textbf{m}(i))_{i\in\N} \subseteq \mathbb{Z}^{k}$ such that
$$\forall \eta \in \{0,1\}^{k},\quad x_{\eta}=\lim\limits_{i\to \infty} T_{j_{1}}^{m_{1}(i)\eta_{1}}\cdots T_{j_{k}}^{m_{k}(i)\eta_{k}}x^{i}.$$
Define $\textbf{n}(i)\in \mathbb{Z}^{d}$ by 
$$n_{p}(i)=\begin{cases}
m_{\ell}(i) & \text{if}\ p=j_\ell \text{ for some } j_\ell \in \{j_{1},\ldots,j_{k}\},\\
0 & \text{if}\ p\in [d]\setminus \{j_{1},\ldots,j_{k}\}.
\end{cases}$$
Then, for $\varepsilon \in \{0,1\}^{d}$ and $\eta \in \{0,1\}^{k}$ such that $\varepsilon_{j_{\ell}}=\eta_{\ell}$ for all $\ell \in \{1,\ldots,k\}$ we have 
$$\begin{array}{ll}
y_{\varepsilon}=x_\eta 
& =  
\lim\limits_{i\to \infty} T_{j_{1}}^{m_{1}(i)\eta_{1}}\cdots T_{j_{k}}^{m_{k}(i)\eta_{k}}x^{i},\\
& = 
\lim\limits_{i\to \infty} T_{1}^{n_{1}(i)\varepsilon_{1}}\cdots T_{d}^{n_{d}(i)\varepsilon_{d}}x^{i},
\end{array}$$
which implies \eqref{prop1:5}.
\end{proof}

Proposition \ref{prop1} \eqref{prop1:3} tells us that $(\QQ_{T_{1},\ldots,T_{d}}(X),\mathcal{G}_{T_{1},\ldots,T_{d}})$ is a topological dynamical system. The following proposition states a little more. Its 
proof requires the technology of the enveloping semigroup. Basic properties of enveloping semigroups were considered in the preliminaries, for a complete exposition see \cite{auslander1988minimal}. 

\begin{proposition}
\label{CubeMin}
Let $(X,T_{1},\ldots, T_{d})$ be a minimal $\Z^{d}$-system. Then, $(\QQ_{T_{1},\ldots,T_{d}}(X),\mathcal{G}_{T_{1},\ldots,T_{d}})$ is a minimal system. Furthermore, if $(X,T_{1},\ldots,T_{d})$ is distal, then $(\QQ_{T_{1},\ldots,T_{d}}(X),\mathcal{G}_{T_{1},\ldots,T_{d}})$ is distal too.
\end{proposition}

\begin{proof}
Let $E(\QQ_{T_{1},\ldots,T_{d}}(X),\mathcal{G}_{T_{1},\ldots,T_{d}})$ be the enveloping semigroup of the system $(\QQ_{T_{1},\ldots,T_{d}}(X),\mathcal{G}_{T_{1},\ldots,T_{d}})$. For every ${\varepsilon \in \{0,1\}^{d}}$ let $\pi_{\varepsilon}\colon\QQ_{T_{1},\ldots,T_{d}}(X)\to X$ be the projection onto the $\varepsilon$-th coordinate and let $\pi_{\varepsilon}^{*}\colon$ $ E(\QQ_{T_{1},\ldots,T_{d}}(X),\mathcal{G}_{T_{1},\ldots,T_{d}})\to E(X,T_{1},\ldots,T_{d})$ be the corresponding semigroups homomorphism.

Let $u\in E(\QQ_{T_{1},\ldots,T_{d}}(X),G_{[d]}^{\Delta})$ denote a minimal idempotent for the system $(\QQ_{T_{1},\ldots,T_{d}}(X),G_{[d]}^{\Delta})$. We show that $u$ is also a minimal idempotent in $E(\QQ_{T_{1},\ldots,T_{d}}(X),\mathcal{G}_{T_{1},\ldots,T_{d}})$. It suffices to show that if $v\in E(\QQ_{T_{1},\ldots,T_{d}}(X),\mathcal{G}_{T_{1},\ldots,T_{d}})$ with $vu=v$, then $uv=u$. Projecting onto  coordinates we deduce that $\pi_{\varepsilon}^{*}(vu)$ $=\pi_{\varepsilon}^{*}(v)\pi_{\varepsilon}^{*}(u)=\pi_{\varepsilon}^{*}(v)$ for every $\varepsilon \in \{0,1\}^{d}$. Since the action of $G^{\Delta}_{[d]}$ is diagonal, then the projection of a minimal idempotent of $E(\QQ_{T_{1},\ldots,T_{d}}(X),G_{[d]}^{\Delta})$ is a minimal idempotent in $E(X,T_{1},\ldots,T_{d})$. Then, we have that $\pi_{\varepsilon}^{*}(u)\pi_{\varepsilon}^{*}(v)=\pi_{\varepsilon}^{*}(u)$ for every $\varepsilon \in \{0,1\}^{d}$. Since $\QQ_{T_{1},\ldots,T_{d}}(X)\subseteq X^{[d]}$, we can view the elements of $E(\QQ_{T_{1},\ldots,T_{d}}(X),\mathcal{G}_{T_{1},\ldots,T_{d}})$ as vectors of dimension $2^{d}$. Thus the projections to the coordinates determine an element of  $E(\QQ_{T_{1},\ldots,T_{d}}(X),\mathcal{G}_{T_{1},\ldots,T_{d}})$. We have deduced that $uv=u$, which implies that $u$ is a minimal idempotent in $E(\QQ_{T_{1},\ldots,T_{d}}(X),\mathcal{G}_{T_{1},\ldots,T_{d}})$.

Let $x\in X$. Since $(X,T_{1},\ldots, T_{d})$ is minimal, there exists a minimal idempotent $u\in E(X,T_{1},\ldots, T_{d})$ such that $ux=x$. Consider $u^{[d]}\in E(\QQ_{T_{1},\ldots,T_{d}}(X),G_{[d]}^{\Delta})$. We have that $u^{[d]}x^{[d]}=x^{[d]}$, so $x^{[d]}$ is an almost periodic point in $X^{[d]}$ (so in $\QQ_{T_{1},\ldots,T_{d}}(X)$) under the action of $G_{[d]}^{\Delta}$. We observe that the point $x^{[d]}$ is minimal under the action of $\mathcal{G}_{T_{1},\ldots,T_{d}}$ since, by the property proved before, $u^{[d]}$ is also a minimal idempotent in $E(\QQ_{T_{1},\ldots,T_{d}}(X), \mathcal{G}_{T_{1},\ldots,T_{d}})$. As $\overline{\mathcal{O}(x^{[d]},\mathcal{G}_{T_{1},\ldots,T_{d}})}=\QQ_{T_{1},\ldots,T_{d}}(X)$, we conclude that $(\QQ_{T_{1},\ldots,T_{d}}(X),$ $\mathcal{G}_{T_{1},\ldots,T_{d}})$ is a minimal system.

Now, if $(X,T_1,\ldots, T_d)$ is distal, then $(X^{[d]}, G^{[d]})$ is also distal by Theorem \ref{distal} (1),  which implies that $(X^{[d]},\mathcal{G}_{T_{1},\ldots,T_{d}})$ is distal by Theorem \ref{distal} (4). Additionally, since $\QQ_{T_{1},\ldots,T_{d}}(X)$ is invariant under $\mathcal{G}_{T_{1},\ldots,T_{d}}$ we also get that $(\QQ_{T_{1},\ldots,T_{d}}(X),\mathcal{G}_{T_{1},\ldots,T_{d}})$ is distal.
\end{proof}

The following proposition states that directional cubes pass through factors.

\begin{proposition}
\label{cube1}
Let $\pi\colon Y \to X$ be a factor map between the $\Z^{d}$-systems 
$(Y,T_{1},\ldots,T_{d})$ and $(X,T_{1},\ldots,$ $T_{d})$. Then, $$\pi^{[d]}(\QQ_{T_{1},\ldots,T_{d}}(Y))=\QQ_{T_{1},\ldots,T_{d}}(X),$$
where $\pi^{[d]}: Y^{[d]} \to X^{[d]}$ is defined from $\pi$ coordinatewise. 
\end{proposition}

\begin{proof}
It follows directly from definition that $\pi^{[d]}(\QQ_{T_{1},\ldots,T_{d}}(Y))\subseteq\QQ_{T_{1},\ldots,T_{d}}(X)$. Then, if $(X,T_{1},\ldots,T_{d})$ is minimal, by Proposition \ref{CubeMin}
the proof follows directly. If not, let $\textbf{x}\in \QQ_{T_{1},\ldots,T_{d}}(X)$ and take sequences ${(x^{i})_{i\in \N}\subseteq X}$ and $(\textbf{n}(i))_{i\in \N}\subseteq \Z^{d}$ such that
$$\forall \varepsilon \in \{0,1\}^{d},\quad x_{\varepsilon}=\lim\limits_{i\to\infty} T_{1}^{n_{1}(i)\varepsilon_{1}}\cdots T_{d}^{n_{d}(i)\varepsilon_{d}}x^{i}.$$
For each $i\in \N$ take $y^{i}\in \pi^{-1}(\{x^{i}\})$. By compactness, we can assume that 
$\displaystyle \lim_{i\to \infty} y^{i}= y$ and 
$$\lim\limits_{i\to \infty}T_{1}^{n_{1}(i)\varepsilon_{1}}\cdots T_{d}^{n_{d}(i)\varepsilon_{d}}y^{i}=y_{\varepsilon},$$
for all $\varepsilon\in \{0,1\}^{d}$. Now, by continuity of $\pi$, we have that $\displaystyle \lim_{i\to \infty}\pi(y^{i})= \pi(y)=x$ and 
$$\pi(y_{\varepsilon})=\pi\left(\lim\limits_{i\to \infty}T_{1}^{n_{1}(i)\varepsilon_{1}}\cdots T_{d}^{n_{d}(i)\varepsilon_{d}}y^{i}\right)=\lim\limits_{i\to \infty}T_{1}^{n_{1}(i)\varepsilon_{1}}\cdots T_{d}^{n_{d}(i)\varepsilon_{d}}\pi\left(y^{i}\right)=x_{\varepsilon},$$
for all $\varepsilon \in \{0,1\}^{d}$.
Then, ${\textbf{y}\in \QQ_{T_{1},\ldots,T_{d}}(Y)}$ and $\pi^{[d]}(\textbf{y})=\textbf{x}\in \QQ_{T_{1},\ldots,T_{d}}(X)$. This proves the lemma. 
\end{proof}
	
We finish with two crucial properties that appear in the different cube theories developed in topological dynamics \cite{host2010nilsequences, Szegedy2010, donoso2014dynamical,gutman2016structurei,gutman2016structureii,gutman2016structureiii,Glasner20181004}. We need to introduce some terminology. 

\begin{definition}
	Let ${\bf x}=(x_{\varepsilon}:\varepsilon \in \{0,1\}^d), \ {\bf y}=(y_{\varepsilon}:\varepsilon \in \{0,1\}^d)$ in $X^{[d]}$ and $j\in [d]$. We say that ${\bf x}$ and ${\bf y}$ coincide in the $j$-upper face if \[(x_{\varepsilon}: \varepsilon\in\{0,1\}^d, \varepsilon_j=1)=(y_{\varepsilon}: \varepsilon\in\{0,1\}^d, \varepsilon_j=1).\]
Similarly, we define when ${\bf x}$ and ${\bf y}$ coincide in their $j$-lower face. 
Also, we say that the $j$-upper face of ${\bf x}$ coincides with the $j$-lower face of ${\bf y}$ if
\[(x_{\varepsilon}: \varepsilon\in\{0,1\}^d, \varepsilon_j=1)=(y_{\varepsilon}: \varepsilon\in\{0,1\}^d, \varepsilon_j=0).\]
\end{definition}

\begin{definition}[Gluing operation] \label{definition:gluing}
Let ${\bf x}=(x_{\varepsilon}:\varepsilon \in \{0,1\}^d)$ and ${\bf y}=(y_{\varepsilon}:\varepsilon \in \{0,1\}^d)$ in $X^{[d]}$. Suppose that the $j$-upper face of ${\bf x}$ coincides with the $j$-lower face of ${\bf y}$. We define the gluing of ${\bf x}$ and ${\bf y}$ along the $j$-face as the point 
${\bf z}=( z_{\varepsilon}: \varepsilon\in \{0,1\}^d) \in X^{[d]}$ such that 
\[ z_{\varepsilon}=\begin{cases} x_{\varepsilon}  & \varepsilon_j=0,  \\ 
y_{\varepsilon} & \varepsilon_j=1.
\end{cases} \]
\end{definition}

\begin{definition}[Crucial properties] 
Let $(X,T_{1},\ldots,T_{d})$ be a $\Z^{d}$-system. 
\begin{enumerate}
\item{\emph{Unique closing parallelepiped property.}} We say that $(X,T_{1},\ldots,T_{d})$ has the \emph{unique closing parallelepiped property} if whenever 
$\textbf{x}, \textbf{y}\in \QQ_{T_{1},\ldots,T_{d}}(X)$ have $2^{d}-1$ coordinates in common then 
$\textbf{x}=\textbf{y}$.
\item{\emph{Gluing property.}}  
We say that $(X,T_{1},\ldots,T_{d})$ has the \emph{gluing property} if for every $j\in [d]$, if ${\bf x},{\bf y}$ $\in\QQ_{T_{1},\ldots,T_{d}}(X)$ are such that the $j$-upper face of ${\bf x}$ coincides with the $j$-lower face of ${\bf y}$, then the point given by the gluing of ${\bf x}$ and ${\bf y}$ along their $j$-face belongs to $\QQ_{T_{1},\ldots,T_{d}}(X)$. 
\end{enumerate}
\end{definition}

In the measure theoretical category the \emph{unique closing parallelepiped property} is exactly the property that characteristic factors for multiple averages for systems with commuting transformations satisfy \cite[Section 2.6]{host2009ergodic}. In the minimal distal case the gluing property always holds. 

\begin{lemma}[Gluing Lemma] 
\label{pegado1}
Let $(X,T_{1},\ldots,T_{d})$ be a minimal distal $\Z^{d}$-system. Then, $(X,T_{1},\ldots,T_{d})$ has the gluing property.
\end{lemma}
\begin{proof}
We write the proof for $j=d$. The general proof is the same modulo a small change of notation.
Let $\textbf{x}',\textbf{x}'',\textbf{y}''\in X^{[d-1]}$ be such that $\textbf{x}=(\textbf{x}',\textbf{x}''), \textbf{y}=(\textbf{x}'',\textbf{y}'')\in \QQ_{T_{1},\ldots,T_{d}}(X)$. By definition, 
the point given by the gluing of ${\bf x}$ and ${\bf y}$ along their $d$-face is ${\bf z}=(\textbf{x}',\textbf{y}'')$.

Pick any $a\in X$. Then, by Proposition \ref{prop1} \eqref{prop1:1}, $(a^{[d-1]},a^{[d-1]})\in \QQ_{T_{1},\ldots,T_{d}}(X)$. By Proposition \ref{CubeMin}, there exists a sequence $(g_{n})_{n\in \N}=((g_{n}',g_{n}''))_{n\in \N}\in \mathcal{G}_{T_{1},\ldots,T_{d}}$ such that $g_{n}(\textbf{x}',\textbf{x}'')=(g_{n}'\textbf{x}',g_{n}''\textbf{x}'')\to (a^{[d-1]},a^{[d-1]})$. We can assume, by compactness, that $g_{n}''\textbf{y}''\to \textbf{u}$ and thus $(g_{n}''\textbf{x}'',g_{n}''\textbf{y}'')\to (a^{[d-1]},\textbf{u})\in \QQ_{T_{1},\ldots,T_{d}}(X)$. Now, we have that $g_{n}(\textbf{x}',\textbf{y}'')=(g_{n}'\textbf{x}',g_{n}''\textbf{y}'')\to (a^{[d-1]},\textbf{u})$, and this point belongs to the closed orbit of $(\textbf{x}',\textbf{y}'')$ under $\mathcal{G}_{T_{1},\ldots,T_{d}}$. By distality (see Proposition \ref{CubeMin}), this orbit is minimal and thus 
it follows that $(\textbf{x}',\textbf{y}'')$ is in the closed orbit of $(a^{[d-1]},\textbf{u})$, which implies that ${\bf z}=(\textbf{x}',\textbf{y}'')\in \QQ_{T_{1},\ldots,T_{d}}(X)$.
\end{proof}

\begin{remark}\label{Rem:gluingsubset}
Remark that the same property holds for a subset of the transformations. That is, if $(X,T_1,\ldots,T_d)$ is a minimal distal $\Z^d$-system, then $(X,T_{i_1},T_{i_2},\ldots,T_{i_k})$, $1\leq k\leq d$, $1\leq i_1<i_2<\cdots <i_{k}\leq d$, also has the gluing property in its space of cubes $\QQ_{T_{i_1},T_{i_2},\ldots, T_{i_k}}(X)$. 
The proof of this is exactly the same as in Lemma \ref{pegado1}, but considering the system $\QQ_{T_{i_1},T_{i_2},\ldots, T_{i_k}}(X)$ endowed  with the action generated by $\mathcal{G}_{T_{i_1},\ldots,T_{i_k}}$ and $G^{\triangle}_{[k]}$ (instead of just $\mathcal{G}_{T_{i_1},\ldots,T_{i_k}}$) . We point out that this joint action is minimal, since it is the projection of   $(\QQ_{T_{1},T_{2},\ldots, T_{d}}(X), \mathcal{G}_{T_1,\ldots,T_d})$ onto a $k$ dimensional face (recall Proposition \ref{prop1} (4)). 
\end{remark}

\subsection{Symmetries of the euclidean cube and directional cubes}  \label{sec:symmetries} As was mentioned above, a different choice of the order of the transformations in a $\Z^d$-system produces a different set of dynamical cubes.
Thus, in order to simplify some proofs in the next sections we introduce the terminology proposed in \cite{host2005nonconventional} about Euclidean permutations of the $d$-dimensional cube $\{0,1\}^d$.

\subsubsection{Digit permutations} Let $d\geq 1$ be an integer. Given a permutation $\sigma\colon [d]\to [d]$  we consider the symmetry that it induces in the euclidean cube $\{0,1\}^d$, which is defined by setting $\sigma(\varepsilon)_i=\varepsilon_{\sigma(i)}$ for $i \in [d]$. Here we slightly abuse of the notation and continue calling this transformation $\sigma$.   
Let $\sigma_{\ast}\colon X^{[d]} \to X^{[d]}$ denote the transformation 
\[ \sigma_{\ast}((x_{\varepsilon}\colon \varepsilon \in \{0,1\}^{d}))=(x_{\sigma(\varepsilon)}\colon \varepsilon \in \{0,1\}^{d}) . \]
We refer to it as a \emph{digit permutation} of $X^{[d]}$.  

\begin{lemma} \label{lem:digitpermutation}
Let $(X,T_1,\ldots,T_d)$ be a $\Z^d$-system and $\sigma\colon[d]\to [d]$ be a permutation. 
Then, $$\sigma_{\ast}(\QQ_{T_{\sigma(1)},\ldots,T_{\sigma(d)}}(X))=\QQ_{T_{1},\ldots,T_{d}}(X)$$. 
\end{lemma}
\begin{proof}
First we prove that for $x\in X$ and $(n_1,\ldots,n_d)\in \Z^d$ the point  $\sigma_{\ast}((T_{\sigma(1)}^{n_{1}\varepsilon_{1}}\cdots T_{\sigma(d)}^{n_{d}\varepsilon_{d}}x : \varepsilon\in \{0,1\}^{d}))$ belongs to $\QQ_{T_{1},\ldots,T_{d}}(X)$. This will prove that $\sigma_{\ast}(\QQ_{T_{\sigma(1)},\ldots,T_{\sigma(d)}}(X)) \subseteq \QQ_{T_{1},\ldots,T_{d}}(X)$. The other inclusion follows analogously by considering the inverse of $\sigma$. 

Let $\eta \in\{0,1\}^d$. Since $\sigma(\eta)=\eta_{\sigma(1)}\cdots\eta_{\sigma(d)}$ we have that 
\begin{align*}\sigma_{\ast}((T_{\sigma(1)}^{n_{1}\varepsilon_{1}}\cdots T_{\sigma(d)}^{n_{d}\varepsilon_{d}}x : \varepsilon\in \{0,1\}^{d}))_{\eta} &=(T_{1}^{n_{1}\varepsilon_{1}}\cdots T_{d}^{n_{d}\varepsilon_{d}}x : \varepsilon\in \{0,1\}^{d})_{\sigma(\eta)},\\
&= T_{\sigma(1)}^{n_{1}\eta_{\sigma(1)}}\cdots T_{\sigma(d)}^{n_{d}\eta_{\sigma(d)}}x, \\
&=T_1^{m_1\eta_{1}}\cdots T_{d}^{m_{d}\eta_{d}}x,
\end{align*}
where $m_i=n_{\sigma^{-1}(i)}$  for all $i\in [d]$.  It is then clear that $\sigma_{\ast}((T_{\sigma(1)}^{n_{1}\varepsilon_{1}}\cdots T_{\sigma(d)}^{n_{d}\varepsilon_{d}}x : \varepsilon\in \{0,1\}^{d}))$ belongs to $\QQ_{T_{1},\ldots,T_{d}}(X)$. 
\end{proof}
We remark that $\QQ_{T_{1},\ldots,T_{d}}(X)$ and $\QQ_{T_{\sigma(1)},\ldots,T_{\sigma(d)}}(X)$ do not need to (and usually do not) coincide.  This lack of symmetry is one of the main differences with the case where all transformations are powers of a given transformation $T$, \emph{i.e.}, when $T_{i}=T^{i}$ for $i\in [d]$. In the latter case, all digit symmetries hold.  We think this is one of the reasons why we cannot expect the unique closing parallelepiped property to imply the strong algebraic consequences that one obtains in the single transformation case  \cite{host2005nonconventional,host2010nilsequences}, or more generally, when the whole group is acting in all directions \cite{Szegedy2010,gutman2016structureiii, Glasner20181004}.
Nevertheless, we still have a reduced number of symmetries that leave invariant $\QQ_{T_{1},\ldots,T_{d}}(X)$ that we describe in the next section. 
\subsubsection{Reflections}
For every $j\in [d]$ we consider the following Euclidean permutation:
$$\begin{array}{lcll}
\Phi_{j}\colon & \{0,1\}^{d} & \to & \{0,1\}^{d}\\
& \varepsilon & \mapsto & \Phi_{j}(\varepsilon)=\left\{\begin{array}{cl} \varepsilon_{k} & \text{if}\ k\neq j,\\ 1-\varepsilon_{k} & \text{if}\ k=j.\\
\end{array}\right.\\
\end{array}
$$
The transformation $\Phi_j$ is the one that takes the upper $j$-face and turn it into the lower $j$-face and vice versa. We naturally define the induced action $\Phi_{j\ast}\colon X^{[d]}\to  X^{[d]}$ by setting 
\[ (\Phi_{j\ast}({\bf x}))_{\varepsilon} = x_{\Phi_{j}(\varepsilon)},\ \text{for all } \varepsilon \in \{0,1\}^{d},  \]
where ${\bf x}=(x_{\varepsilon}: \varepsilon\in \{0,1\}^d)$. 
We have, 
\begin{lemma}[Face permutation invariance]
\label{lem:goodpermutations}
For every $ j\in  [d]$ the transformation $\Phi_{j\ast}$ leaves invariant $\QQ_{T_{1},\ldots,T_{d}}(X)$.
\end{lemma}
\begin{proof}
Since  $\Phi_{j\ast}$ is clearly continuous, it suffices to show that for every $x\in X$ and $(n_1,\ldots,n_d)\in \Z^d$ the image of $(T_1^{n_1\varepsilon_1}\cdots T_d^{n_d\varepsilon_d} x: \varepsilon \in \{0,1\}^d)$ under $\Phi_{j\ast}$ belongs to  $\QQ_{T_{1},\ldots,T_{d}}(X)$.
We have that 
\begin{align*} \Phi_{j\ast}((T_1^{n_1\varepsilon_1}\cdots T_d^{n_d\varepsilon_d} x: \varepsilon \in \{0,1\}^d))=& (T_1^{n_1\varepsilon_1}\cdots T_j^{n_j(1-\varepsilon_j)} \cdots T_d^{n_d\varepsilon_d} x: \varepsilon \in \{0,1\}^d) \\ 
=& (T_1^{n_1\varepsilon_1}\cdots T_j^{(-n_j)\varepsilon_j} \cdots T_d^{n_d\varepsilon_d} (T^{n_j} x): \varepsilon \in \{0,1\}^d)  \end{align*}
and this point clearly belongs to $\QQ_{T_{1},\ldots,T_{d}}(X)$ since  $T^{n_j} x\in X$ and $(n_1,\ldots,-n_j,\ldots,n_d)\in \Z^d$. 
\end{proof}

\begin{definition}[The insertion procedure] 
Let ${\bf x}=(x_{\varepsilon}: \varepsilon\in \{0,1\}^d)$ and ${\bf y}=(y_{\varepsilon}: \varepsilon\in \{0,1\}^d)$ be points in  $X^{[d]}$. We say that the point ${\bf z}=(z_{\varepsilon}: \varepsilon\in \{0,1\}^d)\in X^{[d]}$ is the result of inserting the $j$-upper face of ${\bf x}$ 
into the $j$-lower face of ${\bf y}$ if
\[z_{\varepsilon} =\begin{cases} y_{\varepsilon}  & \text{ if } \varepsilon_j=1,  \\ 
x_{\Phi_j(\varepsilon)} &  \text{ if } \varepsilon_j=0.
\end{cases}  \]	
Similarly one defines the insertion of the $j$-lower face of ${\bf x}$ 
into the $j$-upper face of ${\bf y}$ 
\end{definition}

From Lemma \ref{lem:goodpermutations} and the gluing property we get,

\begin{lemma} \label{lem:insert}
Let $(X,T_1,\ldots,T_d)$ be a $\Z^d$-system where the gluing property holds (in particular when the system is minimal and distal). If ${\bf x}, {\bf y}\in \QQ_{T_{1},\ldots,T_{d}}(X)$ coincide in their $j$-upper face, then the point obtained by inserting the $j$-lower face of ${\bf x}$ into the  $j$-upper face of ${\bf y}$ also belongs to $\QQ_{T_{1},\ldots,T_{d}}(X)$.
\end{lemma}
		
\section{The $(T_{1},\ldots,T_{d})$-regionally proximal relation} \label{sec:RegRelation}
Let $(X,T_{1},\ldots,T_{d})$ be a $\Z^{d}$-system. We define a relation in $X$ associated with the cube structure $\QQ_{T_{1},\ldots,T_{d}}(X)$ established in Section \ref{sec:DirecCubes}. This relation generalizes the definition given by Donoso and Sun for two commuting transformations in \cite{donoso2014dynamical}. The motivation is to find precisely the relations producing the maximal factors with the unique closing parallelepiped property. 
We start defining for $j\in [d]$ the following transformations:
$$\begin{array}{lcll}
& & & \\
\Psi_{j}^{0}\colon & \{0,1\}^{d-1} & \to & \{0,1\}^{d}\\
& \varepsilon & \mapsto & \Psi_{0}^{j}(\varepsilon)=\varepsilon_{1}\varepsilon_{2}\ldots\varepsilon_{j-1}0\varepsilon_{j}\ldots\varepsilon_{d-1},\\
& & &\\
\Psi_{j}^{1}\colon & \{0,1\}^{d-1} & \to & \{0,1\}^{d}\\
& \varepsilon & \mapsto & \Psi_{0}^{j}(\varepsilon)=\varepsilon_{1}\varepsilon_{2}\ldots\varepsilon_{j-1}1\varepsilon_{j}\ldots\varepsilon_{d-1}.\\
\end{array}$$
The maps $\Psi_{j}^{0}$ and $\Psi_{j}^{1}$ are used to replicate a face. Let $j\in [d]$ and $\textbf{x}\in X^{[d-1]}$. If we define $\textbf{y}\in X^{[d]}$ by $y_{\varepsilon}= x_{\eta}\ \text{if}\ \varepsilon = \Psi_{j}^{0}(\eta)\ \vee\ \varepsilon=\Psi_{j}^{1}(\eta)$, we have that both the $j$-lower face of $\textbf{y}$ and the $j$-upper face of $\textbf{y}$ are equal to $\textbf{x}$, \emph{i.e.}, 
$$(y_{\varepsilon}\colon \varepsilon \in \{0,1\}^d, \ \varepsilon_j=0)=(y_{\varepsilon}\colon \varepsilon \in \{0,1\}^d, \ \varepsilon_j=1)=\textbf{x} \in X^{[d-1]}.$$

\begin{definition} 
For $x,y\in X$, $\textbf{a}_{*}\in X_{*}^{[d-1]}$ and $j\in [d]$, define $\textbf{z}(x,y,\textbf{a}_{*},j)\in X^{[d]}$ with coordinates 
$$
(\textbf{z}(x,y,\textbf{a}_{*},j))_{\varepsilon}=\left\{\begin{array}{ll}x & \text{if}\ \varepsilon =\emptyset,\\ y & \text{if}\ \varepsilon=\{j\},\\
(\textbf{a}_{*})_{\eta} & \text{if}\ \varepsilon=\Psi_{j}^{0}(\eta)\ \vee\ \varepsilon= \Psi_{j}^{1}(\eta).\end{array}\right.
$$

\begin{definition}\label{def:Relation}
For each $j \in [d]$ we define the $T_{j}$-regionally proximal relation as
\[\mathcal{R}_{T_{j}}(X) = \left\{(x,y)\in X\times X\colon \exists \ \textbf{a}_{*}\in X_{*}^{[d-1]},\ \textbf{z}(x,y,\textbf{a}_{*},j)\in \QQ_{T_{1},\ldots,T_{d}}(X)\right\}. \]
\end{definition}

For $j\in [d]$ let $\sigma\colon [d]\to [d]$ be the permutation that interchanges $j$ and $1$ (and is the identity everywhere else). By Lemma \ref{lem:digitpermutation},  we have that $\sigma_{\ast}$ maps $\QQ_{T_j,\ldots,T_1,\ldots,T_d}(X)$ onto $\QQ_{T_{1},\ldots,T_{d}}(X)$. It then follows that $(x,y) \in \mathcal{R'}_{T_j}$ if and only if $(x,y)\in \mathcal{R}_{T_j}$, where $\mathcal{R'}_{T_j}$ is the same relation but in the cube $\QQ_{T_j,\ldots,T_1,\ldots,T_d}(X)$ (so the order of the transformations is changed, interchanging $T_1$ and $T_j$). The advantage of doing so is that (modulo changing the order of the transformations) we can assume that the points $(x,y)$ are in $R_{T_1}$, where $T_1$ is the first transformation. This point of view is only esthetical, but it will lighten some proofs notations in the next sections.  

Finally, we define the $(T_{1},\ldots,T_{d})$-regionally proximal relation as,
$$\rr_{T_{1},\ldots,T_{d}}(X)=\bigcap\limits_{j=1}^{d}\rr_{T_{j}}(X).$$
\end{definition}

As an example, for $d=3$ the relations are defined as follows:
$$\begin{array}{ll}
\rr_{T_{1}}(X) & = \{(x,y)\in X\times X\colon (x,y,a,a,b,b,c,c)\in \QQ_{T_{1},T_{2},T_{3}}(X)\ \text{for some}\ a,b,c\in X\},\\
\rr_{T_{2}}(X) & = \{(x,y)\in X \times X\colon (x,a,y,a,b,c,b,c)\in \QQ_{T_{1},T_{2},T_{3}}(X)\ \text{for some}\ a,b,c\in X\},\\
\rr_{T_{3}}(X) & = \{(x,y)\in X \times X\colon (x,a,b,c,y,a,b,c)\in \QQ_{T_{1},T_{2},T_{3}}(X)\ \text{for some}\ a,b,c\in X\}.
\end{array}$$

Graphically we represent these relations as follows:
\
	
\begin{figure}[H]
\centering
\begin{tikzpicture}[scale=1.2]
\node(b1) at (-2,1.35) [scale=1] {$(x,y) \in \rr_{T_{1}}(X):$};
		
\coordinate[scale=.5,label=left:$x$] (a1) at (0,0);
\coordinate[scale=.5,label=left:$b$] (a2) at (0,1.5);
\coordinate[scale=.5,label=right:$y$] (a3) at (1.5,0);
\coordinate[scale=.5,label=right:$b$] (a4) at (1.5,1.5);
\coordinate[scale=.5,label=right:$c$] (a5) at (2.13,2.13);
\coordinate[scale=.5,label=left:$c$] (a6) at (0.63,2.13);
\coordinate[scale=.5,label=right:$a$] (a7) at (2.13,0.63);
\coordinate[scale=.5,label=left:$a$] (a8) at (0.63,0.63);
		
\node(s1) at (3.5,1.35) [scale=1] {$\in \QQ_{T_{1},T_{2},T_{3}}(X)$};
		
\path[thin] (a1) edge (a2);
\path[thin] (a1) edge (a3);
\path[dashed] (a1) edge (a8);
\path[thin] (a2) edge (a4);
\path[thin] (a2) edge (a6);
\path[thin] (a3) edge (a4);
\path[thin] (a3) edge (a7);
\path[thin] (a4) edge (a5);
\path[thin] (a5) edge (a7);
\path[thin] (a5) edge (a6);
\path[dashed] (a6) edge (a8);
\path[dashed] (a7) edge (a8);
		
\node(b2) at (-2,-1.65) [scale=1] {$(x,y) \in \rr_{T_{2}}(X):$};
		
\coordinate[scale=.5,label=left:$x$] (c1) at (0,-3);
\coordinate[scale=.5,label=right:$a$] (c3) at (1.5,-3);
\coordinate[scale=.5,label=left:$y$] (c8) at (0.63,-2.37);
\coordinate[scale=.5,label=right:$a$] (c7) at (2.13,-2.37);
\coordinate[scale=.5,label=left:$b$] (c2) at (0,-1.5);
\coordinate[scale=.5,label=right:$c$] (c4) at (1.5,-1.5);
\coordinate[scale=.5,label=left:$b$] (c6) at (0.63,-0.87);
\coordinate[scale=.5,label=right:$c$] (c5) at (2.13,-0.87);
		
\node(s2) at (3.5,-1.65) [scale=1] {$\in \QQ_{T_{1},T_{2},T_{3}}(X)$};
		
\path[thin] (c1) edge (c2);
\path[thin] (c1) edge (c3);
\path[dashed] (c1) edge (c8);
\path[thin] (c2) edge (c4);
\path[thin] (c2) edge (c6);
\path[thin] (c3) edge (c4);
\path[thin] (c3) edge (c7);
\path[thin] (c4) edge (c5);
\path[thin] (c5) edge (c7);
\path[thin] (c5) edge (c6);
\path[dashed] (c6) edge (c8);
\path[dashed] (c7) edge (c8);
	
\node(b3) at (-2,-4.65) [scale=1] {$(x,y) \in \rr_{T_{3}}(X):$};
		
\coordinate[scale=.5,label=left:$x$] (d1) at (0,-6);
\coordinate[scale=.5,label=right:$a$] (d3) at (1.5,-6);
\coordinate[scale=.5,label=left:$b$] (d8) at (0.63,-5.37);
\coordinate[scale=.5,label=right:$c$] (d7) at (2.13,-5.37);
\coordinate[scale=.5,label=left:$y$] (d2) at (0,-4.5);
\coordinate[scale=.5,label=right:$a$] (d4) at (1.5,-4.5);
\coordinate[scale=.5,label=left:$b$] (d6) at (0.63,-3.87);
\coordinate[scale=.5,label=right:$c$] (d5) at (2.13,-3.87);
		
\node(32) at (3.5,-4.65) [scale=1] {$\in \QQ_{T_{1},T_{2},T_{3}}(X)$};
		
\path[thin] (d1) edge (d2);
\path[thin] (d1) edge (d3);
\path[dashed] (d1) edge (d8);
\path[thin] (d2) edge (d4);
\path[thin] (d2) edge (d6);
\path[thin] (d3) edge (d4);
\path[thin] (d3) edge (d7);
\path[thin] (d4) edge (d5);
\path[thin] (d5) edge (d7);
\path[thin] (d5) edge (d6);
\path[dashed] (d6) edge (d8);
\path[dashed] (d7) edge (d8);
		
\end{tikzpicture}
\caption{Representation of the relations $\rr_{T_{1}}(X)$, $\rr_{T_{2}}(X)$ and $\rr_{T_{3}}(X)$ in the case $d=3$.}
\label{fig:figuracycle}
\end{figure}
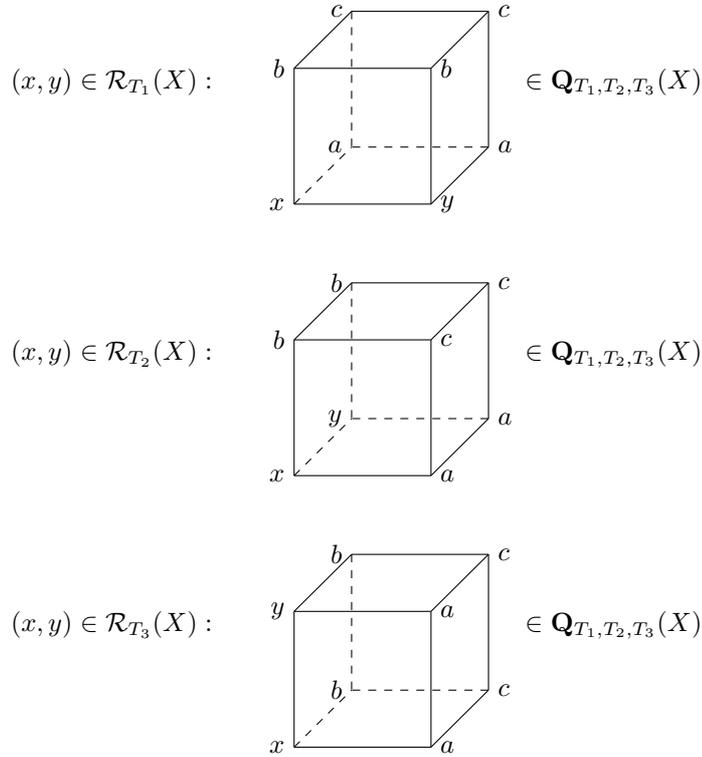

It is easy to see that these relations are reflexive, symmetric, closed and invariant under the action of $\langle T_1,\ldots,T_d\rangle$. In the next section we prove that these relations are transitive in the minimal distal case, but we do not know yet if they are transitive in the general minimal case. 

The next proposition shows that the relations $\rr_{T_{j}}(X)$ for $\Z^d$-systems with the unique closing parallelepiped property are trivial.	
	
\begin{proposition}
Let $(X,T_{1},\ldots,T_{d})$ be a $\Z^{d}$-system with the unique closing parallelepiped property. Then, for every $j\in [d]$ we have that $\rr_{T_{j}}(X)=\Delta_{X}$.
\label{prop4}
\end{proposition}
	
\begin{proof}
Take $(x,y)\in \rr_{T_{j}}(X)$ for a given $j\in [d]$. Then, by definition, there exists ${\textbf{a}_{*} \in X_{*}^{[d-1]}}$ such that $\textbf{z}(x,y,\textbf{a}_{*},j)\in \QQ_{T_{1},\ldots,T_{d}}(X)$. By Proposition \ref{prop1} \eqref{prop1:4}, we have that ${(x,\textbf{a}_{*})\in \QQ_{T_{1},\ldots,T_{j-1},T_{j+1},\ldots,T_{d}}(X)}$. Define $\textbf{w}\in X^{[d]}$ as
$$w_{\varepsilon} = \left\{\begin{array}{ll}x & \text{if}\ \varepsilon =\emptyset,\\ x & \text{if}\ \varepsilon=\{j\},\\
\left(\textbf{a}_{*}\right)_{\eta} & \text{if}\ \varepsilon=\Psi_{j}^{0}(\eta)\ \vee\ \varepsilon= \Psi_{j}^{1}(\eta).\end{array}\right.$$
By Proposition \ref{prop1} \eqref{prop1:5}, we have that $\textbf{w}\in \QQ_{T_{1},\ldots,T_{d}}(X)$ and therefore the unique closing parallelepiped property implies that $x=y$.
\end{proof}
	
The proof of Proposition \ref{prop4} in the case $d=3$ and $j=1$ is illustrated in the next figure. 
\begin{figure}[H]
\label{fig:proof of PropTrivialrelation}
\centering
\begin{tikzpicture}[scale=1.3]
\coordinate[scale=.5,label=left:$x$] (a1) at (0,0);
\coordinate[scale=.5,label=left:$b$] (a2) at (0,1.5);
\coordinate[scale=.5,label=right:$y$] (a3) at (1.5,0);
\coordinate[scale=.5,label=right:$b$] (a4) at (1.5,1.5);
\coordinate[scale=.5,label=right:$c$] (a5) at (2.13,2.13);
\coordinate[scale=.5,label=left:$c$] (a6) at (0.63,2.13);
\coordinate[scale=.5,label=right:$a$] (a7) at (2.13,0.63);
\coordinate[scale=.5,label=left:$a$] (a8) at (0.63,0.63);
		
\node(q1) at (0.75,-0.5) [scale=1] {(a)};
\node(q2) at (5.75,0) [scale=1] {(b)};
\node(q3) at (5.75,-3.2) [scale=1] {(c)};
		
\node(s1) at (3.5,1.35) [scale=1] {$\in \QQ_{T_{1},T_{2},T_{3}}(X) \implies$};
		
\coordinate[scale=.5,label=below:$x$] (b1) at (5,0.5);
\coordinate[scale=.5,label=above:$b$] (b2) at (5,2);
\coordinate[scale=.5,label=below:$a$] (b3) at (6.5,0.5);
\coordinate[scale=.5,label=above:$c$] (b4) at (6.5,2);
		
\node(s2) at (7.4,1.35) [scale=1] {$\in \QQ_{T_{2},T_{3}}(X)$};
\node(s3) at (4.4,-1.4) [scale=1] {$\implies$};
		
\coordinate[scale=.5,label=left:$x$] (c1) at (5,-2.7);
\coordinate[scale=.5,label=left:$b$] (c2) at (5,-1.2);
\coordinate[scale=.5,label=right:$x$] (c3) at (6.5,-2.7);
\coordinate[scale=.5,label=right:$b$] (c4) at (6.5,-1.2);
\coordinate[scale=.5,label=right:$c$] (c5) at (7.13,-0.57);
\coordinate[scale=.5,label=left:$c$] (c6) at (5.63,-0.57);
\coordinate[scale=.5,label=right:$a$] (c7) at (7.13,-2.07);
\coordinate[scale=.5,label=left:$a$] (c8) at (5.63,-2.07);
		
\node(s2) at (8.1,-1.4) [scale=1] {$\in \QQ_{T_{1},T_{2},T_{3}}(X)$};
				
\path[thin] (a1) edge (a2);
\path[thin] (a1) edge (a3);
\path[dashed] (a1) edge (a8);
\path[thin] (a2) edge (a4);
\path[thin] (a2) edge (a6);
\path[thin] (a3) edge (a4);
\path[thin] (a3) edge (a7);
\path[thin] (a4) edge (a5);
\path[thin] (a5) edge (a7);
\path[thin] (a5) edge (a6);
\path[dashed] (a6) edge (a8);
\path[dashed] (a7) edge (a8);
		
\path[thin] (c1) edge (c2);
\path[thin] (c1) edge (c3);
\path[dashed] (c1) edge (c8);
\path[thin] (c2) edge (c4);
\path[thin] (c2) edge (c6);
\path[thin] (c3) edge (c4);
\path[thin] (c3) edge (c7);
\path[thin] (c4) edge (c5);
\path[thin] (c5) edge (c7);
\path[thin] (c5) edge (c6);
\path[dashed] (c6) edge (c8);
\path[dashed] (c7) edge (c8);
		
\filldraw [gray,opacity=0.5,thick] (a1)--(a8)--(a6)--(a2)--(a1)--cycle;
		
\path[thin] (b1) edge (b2);
\path[thin] (b1) edge (b3);
\path[thin] (b2) edge (b4);
\path[thin] (b3) edge (b4);
\end{tikzpicture}
\caption{Illustration of the proof of Proposition \ref{prop4} in the case $d=3$ for $j=1$. We have the existence of a cube like in (a) because $(x,y)\in \rr_{T_{1}}(X)$.} 
\end{figure}
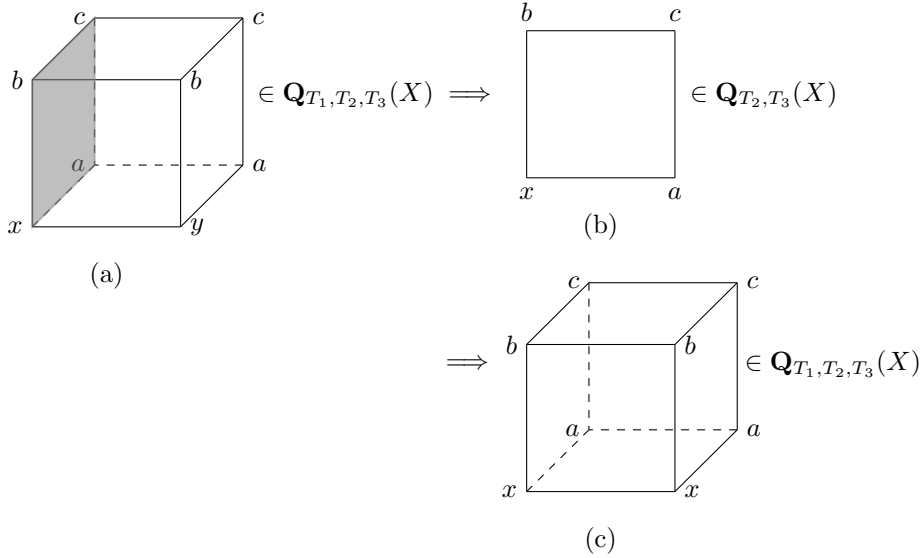

In the literature \cite{host2010nilsequences,shao2012regionally,donoso2014dynamical,gutman2016structurei,gutman2016structureii,gutman2016structureiii,Glasner20181004}, when defining these types of relations one also looks for the following properties: (i) if $\pi\colon Y \to X$ is a factor map between the minimal $\Z^{d}$-systems 
$(Y,T_{1},\ldots,$ $T_{d})$ and $(X,T_{1},\ldots,T_{d})$ we would like to have $\pi^{[d]}(\rr_{T_{j}}(Y))= \rr_{T_{j}}(X)$ for every $j\in [d]$ and $$\pi^{[d]}(\rr_{T_{1},\ldots,T_{d}}(Y)) = \rr_{T_{1},\ldots,T_{d}}(X);$$
and (ii) we also would like to have that $\rr_{T_{1},\ldots,T_{d}}(X)$ is an equivalence relation to study the factor system $({X}/{\rr_{T_{1},\ldots,T_{d}}(X)},T_1,\ldots,T_d)$. Nevertheless, for the moment we cannot prove these properties. In Section  \ref{sec:characterizing} we will prove that they hold for minimal distal systems.

\section{Characterizing the $(T_{1},\ldots,T_{d})$-regionally proximal relation for distal $\Z^d$-systems}
\label{sec:characterizing}
	
	We study the properties of the $(T_{1},\ldots,T_{d})$-regionally proximal relation for  minimal distal systems that will be used in the proof of Theorem \ref{StructThm}.
	
We start with the following theorem, which was proved in the case $d=2$ by Donoso and Sun in\cite{donoso2014dynamical}.
	
\begin{theorem} \label{thm:characterization}
Let $(X,T_{1},\ldots,T_{d})$ be a $\Z^{d}$-system where the gluing property holds. Consider $x,y\in X$. The following statements are equivalent:
\begin{enumerate}[(1)]
\item $(x,y)\in \mathcal{R}_{T_{1},\ldots,T_{d}}(X)$.
\item $(x,y_{*}^{[d]})\in \QQ_{T_{1},\ldots,T_{d}}(X)$. \label{xyyy}
\item There exists $\textbf{a}_{*}\in X_{*}^{[d]}$ such that $(x,\textbf{a}_{*}), (y,\textbf{a}_{*})\in \QQ_{T_{1},\ldots,T_{d}}(X)$.
\item For every $\textbf{a}_{*}\in X_{*}^{[d]}$, $(x,\textbf{a}_{*})\in \QQ_{T_{1},\ldots,T_{d}}(X)$ if and only if $(y,\textbf{a}_{*})\in \QQ_{T_{1},\ldots,T_{d}}(X)$. \label{replace} 
\item There exists $j \in [d]$ such that $(x,y)\in \mathcal{R}_{T_{j}}(X)$.
\end{enumerate}
\label{prop2}
\end{theorem}

\begin{proof} We start with the easy implications and then show the two main ones. 

\noindent $(1) \implies (5)$ follows from defintion.  
	
\noindent $(2) \implies (1)$. Let $j\in [d]$ and $\textbf{a}_{*}=y^{[d-1]}\in X^{[d-1]}$, then $\textbf{z}(x,y,\textbf{a}_{*},j)={(x,y_{*}^{[d]}) \in \QQ_{T_{1},\ldots,T_{d}}(X)}$. Hence $(x,y)\in \mathcal{R}_{T_j}(X)$  and since $j$ is arbitrary we get that $(x,y)\in \mathcal{R}_{T_1,\ldots,T_d}(X)$. 
		
\noindent $(4) \implies (2)$. Since $y^{[d]}=(y,y_{\ast}^{[d]}) \in \QQ_{T_{1},\ldots,T_{d}}(X)$, we can use
${\bf a}_{\ast}=y_{\ast}^{[d]}$ in property (4) to get $(x,y_{\ast}^{[d]})\in \QQ_{T_{1},\ldots,T_{d}}(X)$. 
		
\noindent $(2) \implies (3)$ It follows directly since $y^{[d]}=(y,y_{\ast}^{[d]})$ and $(x,y_{*}^{[d]})\in \QQ_{T_{1},\ldots,T_{d}}(X)$.
		
\noindent $(3) \implies (5)$ Let ${\bf a}_{\ast} \in X_{*}^{[d]}$ be such that $(x,\textbf{a}_{*}), (y,\textbf{a}_{*})\in \QQ_{T_{1},\ldots,T_{d}}(X)$. Then the $d$-th upper faces of $(x,\textbf{a}_{*})$ and $(y,\textbf{a}_{*})$ coincide. By Lemma \ref{lem:insert} the point obtained by inserting the lower $d$-th face of $(y,\textbf{a}_{*})$ into the $d$-th upper of $(x,\textbf{a}_{*})$ also belongs to $\QQ_{T_{1},\ldots,T_{d}}(X)$. But this point can be written as $(x,b_{\ast},y,b_{\ast})$ for some $b_{\ast}\in X^{[d-1]}$. This implies that $(x,y)\in \mathcal{R}_{T_d}(X)$. 
			
We now prove the main implications, namely $(5)\implies (2)$ and $(2)\implies (4)$. 
		
\noindent $(5) \implies (2)$. The idea of the proof is to use the gluing property to construct a sequence of points 
$\textbf{z}^{1},\ldots,{\textbf{z}^{d}\in \QQ_{T_{1},\ldots,T_{d}}(X)}$ so that we increase the number of times that $y$ appears as a coordinate of $\textbf{z}^i$ until getting $\textbf{z}^{d}=(x,y_{*}^{[d]})\in \QQ_{T_{1},\ldots,T_{d}}(X)$. We suggest the reader to look at the illustration of the proof in Figure 
\ref{fig:5implies2}
while reading it. 
		
By the discussion right after Definition \ref{def:Relation} we can assume for the moment that 
$(x,y)\in \mathcal{R}_{T_1}(X)$. So, there exists $\textbf{a}_{*}\in X^{[d-1]}$ for which $\textbf{z}(x,y,\textbf{a}_{*},1)\in \QQ_{T_{1},\ldots,T_{d}}(X)$. We call $\textbf{z}^{1}=\textbf{z}(x,y,\textbf{a}_{*},1)$.
		
Let ${\bf v}^1$ be the duplication of the $1$-th upper face of ${\bf z}^1$. Then 	$\textbf{z}^{1}$ and ${\bf v}^1$ coincide in their $2$-th upper face. By Lemma \ref{lem:insert} we can insert the $2$-th lower face of ${\bf v}^1$ into the $2$-th upper face of ${\bf z}^1$, obtaining a point ${\bf z}^2$ that belongs to $\QQ_{T_{1},\ldots,T_{d}}(X)$. The point ${\bf z}^2$ has the property that 
\[ z_{\varepsilon} = y \text{ if  }\varepsilon \neq \emptyset \text{ and } \varepsilon_j=0 \text{ for all } j \geq 3. \]
Otherwise saying, the lower face of codimension $d-2$ of ${\bf z}^2$  (\emph{i.e.}, the one associated to the face $\{0,1\}^2\times \{0\} \times \cdots \times \{0\}$) is $(x,y,y,y)$. 
	
Now suppose that for $1\leq k\leq d-1$ we have constructed ${\bf z}^k \in \QQ_{T_{1},\ldots,T_{d}}(X)$ such that its lower face of codimension $d-k$ (\emph{i.e.}, the face associated to $\{0,1\}^k\times \{0\}^{d-k}$ ) coincides with $(x,y_{\ast}^{[k]})$.  Let ${\bf v}^k$ be the duplication of the $k$-th upper face of ${\bf z}^k$. Then  ${\bf v}^k$ and ${\bf z}^k$ coincide in their $(k+1)$-th upper face. By Lemma \ref{lem:insert} we can insert the $(k+1)$-th lower face of ${\bf v}^k$ into the $(k+1)$-th upper face of ${\bf z}^k$, obtaining a point ${\bf z}^{k+1}$ that belongs to $\QQ_{T_{1},\ldots,T_{d}}(X)$. This point satisfies that its lower face of codimension $d-(k+1)$ (\emph{i.e.}, the one associated to the face $\{0,1\}^{k+1}\times \{0\}^{d-(k+1)}$) is $(x,y_{\ast}^{[k+1]})$.
With this procedure we build a sequence of points which ends up in ${\bf z}^{d}=(x,y_{\ast}^{[d]})$ as desired.
	
As was stated before, the assumption $(x,y)\in \rr_{T_1}(X)$ is irrelevant. Indeed, modulo changing the order of the transformations we get that $(x,y_{\ast}^{[d]}) \in \QQ_{T_{1},\ldots,T_{d}}(X)$. But, for every permutation $\sigma\colon[d]\to [d]$ we have that $\sigma_{\ast}(x,y_{\ast}^{[d]})= (x,y_{\ast}^{[d]})$. Hence, this point also belongs to $\QQ_{T_{\sigma(1)},\ldots,T_{\sigma(d)}}(X)$ by Lemma \ref{lem:digitpermutation}. This shows that the order considered for the transformations is not important. 
			
\begin{figure}[h]
\centering
			\begin{tikzpicture}[scale=1.1]
			\node(s1) at (-0.5,1) [scale=1] {\large$\textbf{z}^{1}=$};
			\coordinate[scale=.5,label=left:\large$x$] (a1) at (0,0);
			\coordinate[scale=.5,label=left:\large$b$] (a2) at (0,1.5);
			\coordinate[scale=.5,label=right:\large$y$] (a3) at (1.5,0);
			\coordinate[scale=.5,label=right:\large$b$] (a4) at (1.5,1.5);
			\coordinate[scale=.5,label=right:\large$c$] (a5) at (2.13,2.13);
			\coordinate[scale=.5,label=left:\large$c$] (a6) at (0.63,2.13);
			\coordinate[scale=.5,label=right:\large$a$] (a7) at (2.13,0.63);
			\coordinate[scale=.5,label=left:\large$a$] (a8) at (0.63,0.63);
			\filldraw [gray,opacity=0.5,thick] (a3)--(a4)--(a5)--(a7)--(a3)--cycle;
			\filldraw [red,opacity=0.5,thick] (a7)--(a8)--(a6)--(a5)--(a7)--cycle;
	
			\node(s2) at (3.3,1) [scale=1] {\large$\implies$};
			
			\node(s3) at (4.5,1) [scale=1] {\large$\textbf{v}^{1}=$};
			\coordinate[scale=.5,label=left:\large$y$] (b1) at (5,0);
			\coordinate[scale=.5,label=left:\large$b$] (b2) at (5,1.5);
			\coordinate[scale=.5,label=right:\large$y$] (b3) at (6.5,0);
			\coordinate[scale=.5,label=right:\large$b$] (b4) at (6.5,1.5);
			\coordinate[scale=.5,label=right:\large$c$] (b5) at (7.13,2.13);
			\coordinate[scale=.5,label=left:\large$c$] (b6) at (5.63,2.13);
			\coordinate[scale=.5,label=right:\large$a$] (b7) at (7.13,0.63);
			\coordinate[scale=.5,label=left:\large$a$] (b8) at (5.63,0.63);			
			\filldraw [gray,opacity=0.5,thick] (b3)--(b4)--(b5)--(b7)--(b3)--cycle;
			\filldraw [red,opacity=0.5,thick] (b7)--(b8)--(b6)--(b5)--(b7)--cycle;
			
			\node(s4) at (-1,-1.7) [scale=1] {\large$\implies \textbf{z}^{2}=$};
			\coordinate[scale=.5,label=left:\large$x$] (c1) at (0,-2.6);
			\coordinate[scale=.5,label=left:\large$b$] (c2) at (0,-1.1);
			\coordinate[scale=.5,label=right:\large$y$] (c3) at (1.5,-2.6);
			\coordinate[scale=.5,label=right:\large$b$] (c4) at (1.5,-1.1);
			\coordinate[scale=.5,label=right:\large$b$] (c5) at (2.13,-0.47);
			\coordinate[scale=.5,label=left:\large$b$] (c6) at (0.63,-0.47);
			\coordinate[scale=.5,label=right:\large$y$] (c7) at (2.13,-1.97);
			\coordinate[scale=.5,label=left:\large$y$] (c8) at (0.63,-1.97);
			
			\filldraw [gray,opacity=0.5,thick] (c7)--(c8)--(c6)--(c5)--(c7)--cycle;			
			\filldraw [blue,opacity=0.5,thick] (c2)--(c4)--(c5)--(c6)--(c2)--cycle;
			
			\node(s5) at (3.3,-1.7) [scale=1] {\large$\implies$}; 
			
			\node(s6) at (4.5,-1.7) [scale=1] {\large$\textbf{v}^{2}=$};
			\coordinate[scale=.5,label=left:\large$y$] (d1) at (5,-2.6);
			\coordinate[scale=.5,label=left:\large$b$] (d2) at (5,-1.1);
			\coordinate[scale=.5,label=right:\large$y$] (d3) at (6.5,-2.6);
			\coordinate[scale=.5,label=right:\large$b$] (d4) at (6.5,-1.1);
			\coordinate[scale=.5,label=right:\large$b$] (d5) at (7.13,-0.47);
			\coordinate[scale=.5,label=left:\large$b$] (d6) at (5.63,-0.47);
			\coordinate[scale=.5,label=right:\large$y$] (d7) at (7.13,-1.97);
			\coordinate[scale=.5,label=left:\large$y$] (d8) at (5.63,-1.97);	
			\filldraw [gray,opacity=0.5,thick] (d7)--(d8)--(d6)--(d5)--(d7)--cycle;
			\filldraw [blue,opacity=0.5,thick] (d2)--(d4)--(d5)--(d6)--(d2)--cycle;

			\node(s7) at (-1,-4.5) [scale=1] {\large$\implies \textbf{z}^{3}=$};
			\coordinate[scale=.5,label=left:\large$x$] (e1) at (0,-5.3);
			\coordinate[scale=.5,label=left:\large$y$] (e2) at (0,-3.8);
			\coordinate[scale=.5,label=right:\large$y$] (e3) at (1.5,-5.3);
			\coordinate[scale=.5,label=right:\large$y$] (e4) at (1.5,-3.8);
			\coordinate[scale=.5,label=right:\large$y$] (e5) at (2.13,-3.17);
			\coordinate[scale=.5,label=left:\large$y$] (e6) at (0.63,-3.17);
			\coordinate[scale=.5,label=right:\large$y$] (e7) at (2.13,-4.67);
			\coordinate[scale=.5,label=left:\large$y$] (e8) at (0.63,-4.67);
			
			\path[thin] (a1) edge (a2);
			\path[thin] (a1) edge (a3);
			\path[dashed] (a1) edge (a8);
			\path[thin] (a2) edge (a4);
			\path[thin] (a2) edge (a6);
			\path[thin] (a3) edge (a4);
			\path[thin] (a3) edge (a7);
			\path[thin] (a4) edge (a5);
			\path[thin] (a5) edge (a7);
			\path[thin] (a5) edge (a6);
			\path[dashed] (a6) edge (a8);
			\path[dashed] (a7) edge (a8);
			
			\path[thin] (b1) edge (b2);
			\path[thin] (b1) edge (b3);
			\path[dashed] (b1) edge (b8);
			\path[thin] (b2) edge (b4);
			\path[thin] (b2) edge (b6);
			\path[thin] (b3) edge (b4);
			\path[thin] (b3) edge (b7);
			\path[thin] (b4) edge (b5);
			\path[thin] (b5) edge (b7);
			\path[thin] (b5) edge (b6);
			\path[dashed] (b6) edge (b8);
			\path[dashed] (b7) edge (b8);
								
			\path[thin] (c1) edge (c2);
			\path[thin] (c1) edge (c3);
			\path[dashed] (c1) edge (c8);
			\path[thin] (c2) edge (c4);
			\path[thin] (c2) edge (c6);
			\path[thin] (c3) edge (c4);
			\path[thin] (c3) edge (c7);
			\path[thin] (c4) edge (c5);
			\path[thin] (c5) edge (c7);
			\path[thin] (c5) edge (c6);
			\path[dashed] (c6) edge (c8);
			\path[dashed] (c7) edge (c8);

			\path[thin] (d1) edge (d2);
			\path[thin] (d1) edge (d3);
			\path[dashed] (d1) edge (d8);
			\path[thin] (d2) edge (d4);
			\path[thin] (d2) edge (d6);
			\path[thin] (d3) edge (d4);
			\path[thin] (d3) edge (d7);
			\path[thin] (d4) edge (d5);
			\path[thin] (d5) edge (d7);
			\path[thin] (d5) edge (d6);
			\path[dashed] (d6) edge (d8);
			\path[dashed] (d7) edge (d8);

			\path[thin] (e1) edge (e2);
			\path[thin] (e1) edge (e3);
			\path[dashed] (e1) edge (e8);
			\path[thin] (e2) edge (e4);
			\path[thin] (e2) edge (e6);
			\path[thin] (e3) edge (e4);
			\path[thin] (e3) edge (e7);
			\path[thin] (e4) edge (e5);
			\path[thin] (e5) edge (e7);
			\path[thin] (e5) edge (e6);
			\path[dashed] (e6) edge (e8);
			\path[dashed] (e7) edge (e8);
\end{tikzpicture}
\caption{Illustration of the proof of (5) $\implies$ (2) of Theorem \ref{prop2} for the case $d=3$ and $j=1$. Always the grey faces are duplicated and the color faces serve to pass to the next step using Lemma \ref{lem:insert}.}
\label{fig:5implies2}
\end{figure}
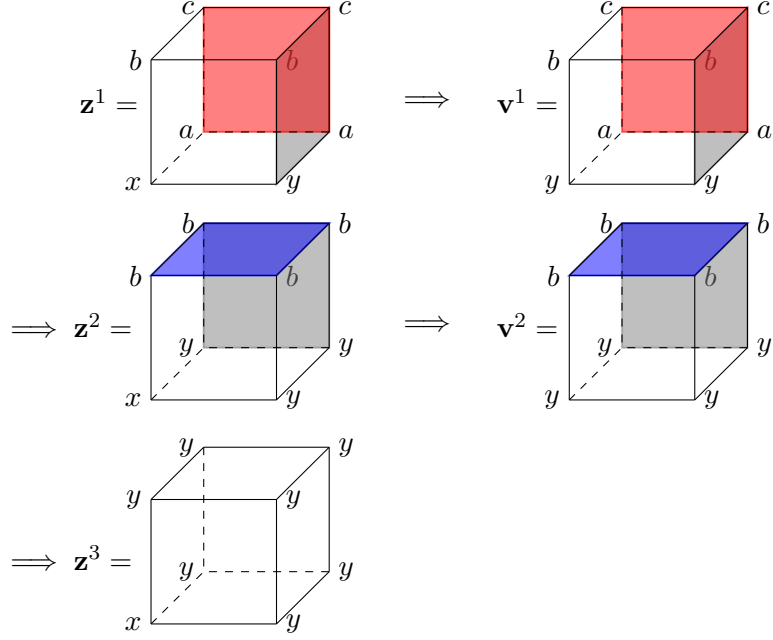
				
We remark that a consequence of this implication is that  
$(x,y_{\ast}^{[d]})\in \QQ_{T_{1},\ldots,T_{d}}(X)$ is equivalent to $(y,x_{\ast}^{[d]})\in \QQ_{T_{1},\ldots,T_{d}}(X)$.

\noindent $(2) \implies (4)$. Here we also advice the reader to look at Figure \ref{Fig:2imp4} while reading the proof. Assume that $(x,y_{\ast}^{[d]} )\in \QQ_{T_{1},\ldots,T_{d}}(X)$ and let $\textbf{a}_{*}\in X_{\ast}^{[d]}$ be such that $\textbf{z}=(y,\textbf{a}_{*})\in \QQ_{T_{1},\ldots,T_{d}}(X)$. We have to show that ${(x,\textbf{a}_{*})\in \QQ_{T_{1},\ldots,T_{d}}(X)}$. The other implication will follow by the symmetry implied by previous implication remarked above.  
		
Consider the $d$ points  ${\textbf{w}}^{1},\ldots,{\textbf{w}^{d}}\in \QQ_{T_{1},\ldots,T_{d}}(X)$ such that for every $1\leq k\leq d$ the point ${\bf w}^k$ is the replication of the lower face of dimension $k$ of ${\bf z}$, (the one associated to $\{0,1\}^k\times \{0\}^{d-k}$ or to $[k]$ in the subset notation). That means that for every $\varepsilon=\varepsilon_1\cdots\varepsilon_k\varepsilon_{k+1}\cdots\varepsilon_d \in \{0,1\}^d$ 
\[ w^k_{\varepsilon}=z_{\varepsilon_1\cdots\varepsilon_k 0\cdots0}.\]
Such points can be obtained using Proposition \ref{prop1}(5). We remark that ${\bf w}^d={\bf z}$. 
		
Next we construct $d+1$ points $\textbf{v}^{1},\ldots,\textbf{v}^{d+1}\in \QQ_{T_{1},\ldots,T_{d}}(X)$, starting with $\textbf{v}^{1}=(x,y_{\ast}^{[d]})$ and transforming it step by step into $\textbf{v}^{d+1}=(x,\textbf{a}_{*})$.   

We have that the 1-th upper face of ${\bf v}^1$ and the  1-th lower face of ${\bf w}^1$ coincide (they are equal to $y^{[d-1]}$). Hence by Lemma \ref{pegado1}, the point ${\bf v}^2$ obtained by gluing 
the 1-th lower face of ${\bf v}^1$ and the 1-th upper face of ${\bf w}^1$ belongs to $\QQ_{T_{1},\ldots,T_{d}}(X)$. We claim that the point ${\bf v}^2$ is equal to ${\bf w}^1$ everywhere but in the 
$\vec{0}$ coordinate, where it is equal to $x$. To see this, note that in the lower 1-th face, \emph{i.e.}, 
if $\varepsilon_1=0$ and $\varepsilon\neq \vec{0}$ we have
\[ v^2_{\varepsilon}=v^1_{\varepsilon} =y=w^1_{\varepsilon},\]
while if $\varepsilon_1=1$ we have
\[ v^2_{\varepsilon}=w^1_{\varepsilon},\]
proving the claim. 

Now assume that we have constructed ${\bf v}^k$ such that it coincides with ${\bf w}^{k-1}$ up to the 
$\vec{0}$ coordinate, where  it is equal to $x$, \emph{i.e.},   
\[v^k_{\vec{0}}=x\quad  \text{ and } \quad v^k_{\varepsilon}=w^{k-1}_{\varepsilon} \text{ for all }\varepsilon \in \{0,1\}^{d} \setminus\{\vec{0}\}.\]

We now construct ${\bf v}^{k+1}$ satisfying the corresponding properties. We have that the $k$-th lower face of ${\bf w}^{k}$ coincides with the $k$-th upper face of ${\bf v}^{k}$.
To see this, note that
\[ w^{k}_{\varepsilon_1\cdots \varepsilon_{k-1}0\varepsilon_{k+1}\cdots\varepsilon_d} =z_{\varepsilon_1\cdots \varepsilon_{k-1}0\cdots0} \]
and
\[v^{k}_{\varepsilon_1\cdots \varepsilon_{k-1}1\varepsilon_{k+1}\cdots\varepsilon_d} = { w}^{k-1}_{\varepsilon_1\cdots \varepsilon_{k-1}1\varepsilon_{k+1}\cdots\varepsilon_d}= z_{\varepsilon_1\cdots \varepsilon_{k-1}0\cdots0}.\]
Hence the gluing property ensures the existence of a point ${\bf v}^{k+1}\in \QQ_{T_{1},\ldots,T_{d}}(X)$ whose $k$-th lower face is equal to the one of ${\bf v}^k$ and its $k$-th upper face is equal to the one of ${\bf w}^k$. We now check that this point is equal to ${\bf w}^{k}$ everywhere but in the $\vec{0}$ coordinate. 

In the $k$-th lower  face, \emph{i.e.},  if $\varepsilon_k=0$ and $\varepsilon \neq {\vec{0}}$ we have
\[v^{k+1}_{\varepsilon} = v^k_{\varepsilon}=w^{k-1}_{\varepsilon_1\cdots\varepsilon_k0\cdots0}=z_{\varepsilon_1\cdots\varepsilon_{k-1}0\cdots0}=z_{\varepsilon_1\cdots\varepsilon_{k-1}\varepsilon_k0\cdots0}= w^k_{\varepsilon} , \]
while in the $k$-th upper face, \emph{i.e.}, if $\varepsilon_k=1$, we have by definition that
\[v^{k+1}_{\varepsilon} = w^k_{\varepsilon}.\]

We end up with ${\bf v}^{d+1}$ which is equal to ${\bf w}^{d}={\bf z}$ everywhere but in the $\vec{0}$ coordinate, where it is equal to $x$. That is, ${\bf v}^{d+1}=(x,{\bf a}_{\ast})\in \QQ_{T_{1},\ldots,T_{d}}(X)$, finishing the proof.
\end{proof}

\begin{figure}[H]
\centering
			\begin{tikzpicture}[scale=1.1]
			\node(s1) at (-0.5,0.7) [scale=1] {\large$\textbf{z}=$};
			\coordinate[scale=.5,label=left:\large$y$] (a1) at (0,0);
			\coordinate[scale=.5,label=right:\large$a$] (a3) at (1.5,0);
			\coordinate[scale=.5,label=left:\large$b$] (a8) at (0.63,0.63);
			\coordinate[scale=.5,label=right:\large$c$] (a7) at (2.13,0.63);
			\coordinate[scale=.5,label=left:\large$d$] (a2) at (0,1.5);
			\coordinate[scale=.5,label=right:\large$e$] (a4) at (1.5,1.5);
			\coordinate[scale=.5,label=left:\large$f$] (a6) at (0.63,2.13);
			\coordinate[scale=.5,label=right:\large$g$] (a5) at (2.13,2.13);
			
			\filldraw [blue,opacity=0.5,thick] (a1)--(a3)--(a7)--(a8)--(a1)--cycle;
			
			\path[thin] (a1) edge (a2);
			\path[thin] (a1) edge (a3);
			\path[dashed] (a1) edge (a8);
			\path[thin] (a2) edge (a4);
			\path[thin] (a2) edge (a6);
			\path[thin] (a3) edge (a4);
			\path[thin] (a3) edge (a7);
			\path[thin] (a4) edge (a5);
			\path[thin] (a5) edge (a7);
			\path[thin] (a5) edge (a6);
			\path[dashed] (a6) edge (a8);
			\path[dashed] (a7) edge (a8);
			
			\node(s4) at (4.6,0.7) [scale=1] {$\textbf{v}^{1}=$};
			
			\coordinate[scale=.5,label=left:\large$x$] (c1) at (5,0);
			\coordinate[scale=.5,label=left:\large$y$] (c2) at (5,1.5);
			\coordinate[scale=.5,label=right:\large$y$] (c3) at (6.5,0);
			\coordinate[scale=.5,label=right:\large$y$] (c4) at (6.5,1.5);
			\coordinate[scale=.5,label=right:\large$y$] (c5) at (7.13,2.13);
			\coordinate[scale=.5,label=left:\large$y$] (c6) at (5.63,2.13);
			\coordinate[scale=.5,label=right:\large$y$] (c7) at (7.13,0.63);
			\coordinate[scale=.5,label=left:\large$y$] (c8) at (5.63,0.63);
			
			\path[thin] (c1) edge (c2);
			\path[thin] (c1) edge (c3);
			\path[dashed] (c1) edge (c8);
			\path[thin] (c2) edge (c4);
			\path[thin] (c2) edge (c6);
			\path[thin] (c3) edge (c4);
			\path[thin] (c3) edge (c7);
			\path[thin] (c4) edge (c5);
			\path[thin] (c5) edge (c7);
			\path[thin] (c5) edge (c6);
			\path[dashed] (c6) edge (c8);
			\path[dashed] (c7) edge (c8);
			
			\filldraw [gray,opacity=0.5,thick] (c3)--(c4)--(c5)--(c7)--(c3)--cycle;
			
			\coordinate[scale=.5,label=left:\large$y$] (b1) at (0,-2.6);
			\coordinate[scale=.5,label=right:\large$a$] (b3) at (1.5,-2.6);
			\coordinate[scale=.5,label=left:\large$y$] (b8) at (0.63,-1.97);
			\coordinate[scale=.5,label=right:\large$a$] (b7) at (2.13,-1.97);
			\coordinate[scale=.5,label=left:\large$y$] (b2) at (0,-1.1);
			\coordinate[scale=.5,label=right:\large$a$] (b4) at (1.5,-1.1);
			\coordinate[scale=.5,label=left:\large$y$] (b6) at (0.63,-0.47);
			\coordinate[scale=.5,label=right:\large$a$] (b5) at (2.13,-0.47);
			
			\node(s2) at (-0.5,-1.8) [scale=1] {\large$\textbf{w}^{1}=$};
			
			\path[thin] (b1) edge (b2);
			\path[thin] (b1) edge (b3);
			\path[dashed] (b1) edge (b8);
			\path[thin] (b2) edge (b4);
			\path[thin] (b2) edge (b6);
			\path[thin] (b3) edge (b4);
			\path[thin] (b3) edge (b7);
			\path[thin] (b4) edge (b5);
			\path[thin] (b5) edge (b7);
			\path[thin] (b5) edge (b6);
			\path[dashed] (b6) edge (b8);
			\path[dashed] (b7) edge (b8);
			
			\filldraw [gray,opacity=0.5,thick] (b1)--(b8)--(b6)--(b2)--(b1)--cycle;
			
			\coordinate[scale=.5,label=left:\large$x$] (d1) at (5,-2.6);
			\coordinate[scale=.5,label=right:\large$a$] (d3) at (6.5,-2.6);
			\coordinate[scale=.5,label=left:\large$y$] (d8) at (5.63,-1.97);
			\coordinate[scale=.5,label=right:\large$a$] (d7) at (7.13,-1.97);
			\coordinate[scale=.5,label=left:\large$y$] (d2) at (5,-1.1);
			\coordinate[scale=.5,label=right:\large$a$] (d4) at (6.5,-1.1);
			\coordinate[scale=.5,label=left:\large$y$] (d6) at (5.63,-0.47);
			\coordinate[scale=.5,label=right:\large$a$] (d5) at (7.13,-0.47);
			
			\node(s4) at (4.6,-1.8) [scale=1] {$\textbf{v}^{2}=$};
			
			\path[thin] (d1) edge (d2);
			\path[thin] (d1) edge (d3);
			\path[dashed] (d1) edge (d8);
			\path[thin] (d2) edge (d4);
			\path[thin] (d2) edge (d6);
			\path[thin] (d3) edge (d4);
			\path[thin] (d3) edge (d7);
			\path[thin] (d4) edge (d5);
			\path[thin] (d5) edge (d7);
			\path[thin] (d5) edge (d6);
			\path[dashed] (d6) edge (d8);
			\path[dashed] (d7) edge (d8);
			
			\filldraw [red,opacity=0.5,thick] (d8)--(d7)--(d5)--(d6)--(d8)--cycle;
		
			\coordinate[scale=.5,label=left:\large$y$] (e1) at (0,-5.2);
			\coordinate[scale=.5,label=right:\large$a$] (e3) at (1.5,-5.2);
			\coordinate[scale=.5,label=left:\large$b$] (e8) at (0.63,-4.57);
			\coordinate[scale=.5,label=right:\large$c$] (e7) at (2.13,-4.57);
			\coordinate[scale=.5,label=above:\large$y$] (e2) at (0,-3.7);
			\coordinate[scale=.5,label=right:\large$a$] (e4) at (1.5,-3.7);
			\coordinate[scale=.5,label=above:\large$b$] (e6) at (0.63,-3.07);
			\coordinate[scale=.5,label=above:\large$c$] (e5) at (2.13,-3.07);

			\filldraw [red,opacity=0.5,thick] (e1)--(e3)--(e4)--(e2)--(e1)--cycle;
			
			\node(s9) at (-0.5,-4.4) [scale=1] {\large$\textbf{w}^{2}=$};
			
			\node(s10) at (4.6,-4.4) [scale=1] {\large$\textbf{v}^{3}=$};
			
			\coordinate[scale=.5,label=left:\large$x$] (f1) at (5,-5.2);
			\coordinate[scale=.5,label=right:\large$a$] (f3) at (6.5,-5.2);
			\coordinate[scale=.5,label=left:\large$b$] (f8) at (5.63,-4.57);
			\coordinate[scale=.5,label=right:\large$c$] (f7) at (7.13,-4.57);
			\coordinate[scale=.5,label=left:\large$y$] (f2) at (5,-3.7);
			\coordinate[scale=.5,label=right:\large$a$] (f4) at (6.5,-3.7);
			\coordinate[scale=.5,label=left:\large$b$] (f6) at (5.63,-3.07);
			\coordinate[scale=.5,label=right:\large$c$] (f5) at (7.13,-3.07);
			
			\filldraw [blue,opacity=0.5,thick] (f2)--(f4)--(f5)--(f6)--(f2)--cycle;
			
			\node(s12) at (1.5,-7) [scale=1] {\large$\textbf{v}^{4}=$};
			
			\coordinate[scale=.5,label=left:\large$x$] (g1) at (2,-7.8);
			\coordinate[scale=.5,label=right:\large$a$] (g3) at (3.5,-7.8);
			\coordinate[scale=.5,label=left:\large$b$] (g8) at (2.63,-7.17);
			\coordinate[scale=.5,label=right:\large$c$] (g7) at (4.13,-7.17);
			\coordinate[scale=.5,label=left:\large$d$] (g2) at (2,-6.3);
			\coordinate[scale=.5,label=right:\large$e$] (g4) at (3.5,-6.3);
			\coordinate[scale=.5,label=left:\large$f$] (g6) at (2.63,-5.67);
			\coordinate[scale=.5,label=right:\large$g$] (g5) at (4.13,-5.67);
			
			\path[thin] (e1) edge (e2);
			\path[thin] (e1) edge (e3);
			\path[dashed] (e1) edge (e8);
			\path[thin] (e2) edge (e4);
			\path[thin] (e2) edge (e6);
			\path[thin] (e3) edge (e4);
			\path[thin] (e3) edge (e7);
			\path[thin] (e4) edge (e5);
			\path[thin] (e5) edge (e7);
			\path[thin] (e5) edge (e6);
			\path[dashed] (e6) edge (e8);
			\path[dashed] (e7) edge (e8);
			
			\path[thin] (f1) edge (f2);
			\path[thin] (f1) edge (f3);
			\path[dashed] (f1) edge (f8);
			\path[thin] (f2) edge (f4);
			\path[thin] (f2) edge (f6);
			\path[thin] (f3) edge (f4);
			\path[thin] (f3) edge (f7);
			\path[thin] (f4) edge (f5);
			\path[thin] (f5) edge (f7);
			\path[thin] (f5) edge (f6);
			\path[dashed] (f6) edge (f8);
			\path[dashed] (f7) edge (f8);
			
			\path[thin] (g1) edge (g2);
			\path[thin] (g1) edge (g3);
			\path[dashed] (g1) edge (g8);
			\path[thin] (g2) edge (g4);
			\path[thin] (g2) edge (g6);
			\path[thin] (g3) edge (g4);
			\path[thin] (g3) edge (g7);
			\path[thin] (g4) edge (g5);
			\path[thin] (g5) edge (g7);
			\path[thin] (g5) edge (g6);
			\path[dashed] (g6) edge (g8);
			\path[dashed] (g7) edge (g8);
			
			\end{tikzpicture}
			\caption{Illustration of the proof of (2) $\implies$ (4) in the case $d=3$.} \label{Fig:2imp4}
		\end{figure}
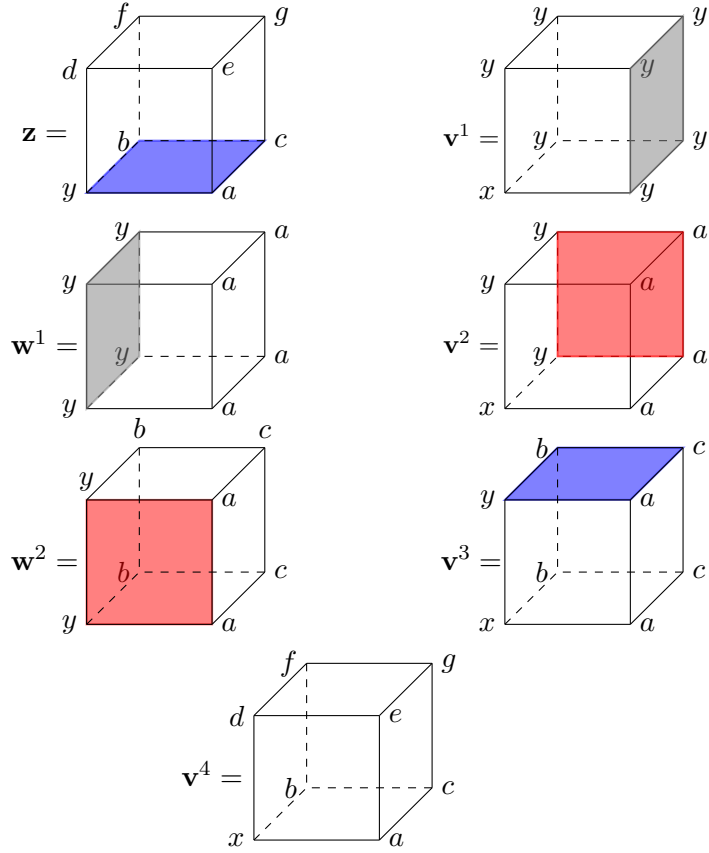
			
We can prove that $\mathcal{R}_{T_{1},\ldots,T_{d}}(X)$ is an equivalence relation in the minimal distal case.
	
\begin{theorem}
\label{EquivRel}
Let $(X,T_{1},\ldots,T_{d})$ be a minimal distal $\Z^d$-system. Then, $\mathcal{R}_{T_{1},\ldots,T_{d}}(X)$ is a closed and invariant equivalence relation on $X$.
\end{theorem}
	
\begin{proof}
It suffices to prove the transitivity of $\mathcal{R}_{T_{1},\ldots,T_{d}}(X)$. Let $(x,y), (y,z)\in \mathcal{R}_{T_{1},\ldots,T_{d}}(X)$. By Theorem \ref{prop2}\eqref{xyyy} we have that $(y,z_{*}^{[d]})\in \QQ_{T_{1},\ldots,T_{d}}(X)$. Now, by Theorem \ref{prop2}\eqref{replace} ${(x,z_{*}^{[d]})\in \QQ_{T_{1},\ldots,T_{d}}(X)}$ and thus ${(x,z)\in \mathcal{R}_{T_{1},\ldots,T_{d}}(X)}$.
\end{proof}
	
We also have the following property which allows us to lift the relation $\mathcal{R}_{T_{1},\ldots,T_{d}}(X)$ by a factor map and to prove our main theorem.
	
\begin{theorem}
Let $\pi\colon Y\to X$ be a factor map between the $\Z^d$-systems $(Y,T_{1},\ldots,T_{d})$ and $(X,T_{1},\ldots,T_{d})$.  If $(X,T_{1},\ldots,T_{d})$ is minimal and distal, then $\pi\times \pi(\mathcal{R}_{T_{1},\ldots,T_{d}}(Y))=\mathcal{R}_{T_{1},\ldots,T_{d}}(X)$.
\label{teo2}
\end{theorem}
	
\begin{proof}
The proof is similar to that of Theorem 6.4 in \cite{shao2012regionally}. It is a direct consequence of the definition that ${\pi\times\pi(\rr_{T_{1},\ldots,T_{d}}(Y))\subseteq\rr_{T_{1},\ldots,T_{d}}(X)}$ and so we are left to prove the reverse inclusion. Let $(x,z)\in \mathcal{R}_{T_{1},\ldots,T_{d}}(X)$. By Theorem \ref{thm:characterization} we have that $(x,\ldots,x,z)\in \QQ_{T_{1},\ldots,T_{d}}(X)$. Replacing $Y$ by a minimal subsystem if needed, we may assume that $(Y,T_1,\ldots,T_d)$ is minimal.
				
By Proposition \ref{cube1} there exists $\textbf{y}^{0}=(y_{1},\textbf{a}_{*}) \in\QQ_{T_{1},\ldots,T_{d}}(Y)$ such that $\pi^{[d]}(\textbf{y}^0)=(x,\ldots,x,z)$. Let  $\textbf{y}^0_{\textbf{I}}=(y^0_{\varepsilon}\colon \varepsilon_d=0)$ and ${\textbf{y}^0_{\textbf{II}}=(y^0_{\varepsilon}\colon \varepsilon_d=1)}$ the $d$-th lower and upper faces of ${\bf y}^0$ respectively. We have that $\textbf{y}^0_{\textbf{I}}\in \QQ_{T_{1},\ldots,T_{d-1}}(Y)$. Let $\mathcal{G}_{T_1,\ldots,T_{d-1}}'$ denote the projection of the action of $\mathcal{G}_{T_1,\ldots,T_{d}}$ onto $\QQ_{T_{1},\ldots,T_{d-1}}(Y)$. That is,  $\mathcal{G}_{T_1,\ldots,T_{d-1}}'$ is generated by  $\mathcal{G}_{T_1,\ldots,T_{d-1}}$ and the diagonal action $G^{\triangle}_{[d-1]}$. By minimality of $(\QQ_{T_{1},\ldots,T_{d}}(Y),\mathcal{G}_{T_1,\ldots,T_{d}})$ we have that  $(\QQ_{T_{1},\ldots,T_{d-1}}(Y),\mathcal{G}_{T_1,\ldots,T_{d-1}}')$ is minimal so there exists a sequence $(\mathcal{G}^0_i)_{i\geq 1}$ in $\mathcal{G}_{T_1,\ldots,T_{d-1}}'$ such that $\mathcal{G}^0_i\textbf{y}^0_{\textbf{I}}\to y_{1}^{[d-1]}$. By compactness we can assume that $\mathcal{G}^0_i\textbf{y}^0_{\textbf{II}}\to \textbf{z}_{\textbf{II}}^0$ for some $\textbf{z}_{\textbf{II}}^0 \in \QQ_{T_{1},\ldots,T_{d-1}}(Y)$. Let $\tilde{\mathcal{G}}^0_{i}=(\mathcal{G}^0_i,\mathcal{G}^0_i)$, \emph{i.e.}, the transformation that consists in applying  $\mathcal{G}^0_i$ to both the upper and lower $d$-th faces of a point in $\QQ_{T_{1},\ldots,T_{d}}(Y)$. Thus, we have that $\tilde{\mathcal{G}}^0_i\textbf{y}^0=(\mathcal{G}^0_i\textbf{y}^0_{\textbf{I}},\mathcal{G}^0_i\textbf{y}^0_{\textbf{II}})$ converges to a point $\textbf{y}^1$ in $\QQ_{T_{1},\ldots,T_{d}}(Y)$ whose $d$-th lower face is equal to $y_{1}^{[d-1]}$ and whose $d$-th upper face is equal to $\textbf{z}_{\textbf{II}}^0$. Notice that 

\[\pi^{[d]}(\textbf{y}^{1})=\textbf{x}^{1}=(x,\ldots,x,z_1),\]
for some $z_1 \in X$. 

\textbf{Claim:} The point  $(x,z_1)$ belongs to  $\overline{\mathcal{O}((x,z),G_{[2]}^{\Delta})}.$

To prove the claim, remark that the projection of $\tilde{\mathcal{G}}^0_i\textbf{y}^0$ onto the coordinates $011 \ldots 1$ and $1\ldots 1$ is of the form $(gx,gz)$ for some $g\in G$. By construction these coordinates converge to $(x,z_1)$, proving the claim.

Now assume we have constructed points $\textbf{y}^{{j}}\in \QQ_{T_{1},\ldots,T_{d}}(Y)$ for $1\leq j \leq d-1$ with $\pi^{[d]}(\textbf{y}^{j})=\textbf{x}^{j}$ such that $y_{\varepsilon}^{j}=y_{1}$ if there exists some $k$ with $d-j+1\leq k \leq d$ and $\varepsilon_k=0$ (\emph{i.e.}, $\textbf{y}^{{j}}$ coincides with $y_1^{[d]}$ in all $k$-th lower faces for $d-j+1\leq k\leq d$),  $z_j= x_{[d]}^{j}$, $x_{\varepsilon}^{j}=x$ for all $\varepsilon \neq [d]$ and $(x,z_{j})\in \overline{\mathcal{O}((x,z_{j-1}),G_{[2]}^{\Delta})}$.

Let $\textbf{y}_{\textbf{I}}^{j}=(y_{\varepsilon}^{j}\colon \varepsilon \in \{0,1\}^{d}, \varepsilon_{d-j}=0)$ and $\textbf{y}_{\textbf{II}}^{j}=(y_{\varepsilon}^{j}\colon \varepsilon \in \{0,1\}^{d}, \varepsilon_{d-j}=1)$ be the $(d-j)$-th lower and upper faces of $\textbf{y}^{{j}}$ respectively. Let $\mathcal{G}_{T_{1},\ldots,T_{d-j-1},T_{d-j+1},\ldots,T_{d}}'$  (excluding $T_{d-j}$) be the group generated by $\mathcal{G}_{T_{1},\ldots,T_{d-j-1},T_{d-j+1},\ldots,T_{d}}$  and the diagonal $G^{\triangle}_{[d-1]}$. By a minimality argument (identical to the one written for $j=0$), there exists a  sequence
$(\mathcal{G}^j_i)_{i\geq 1}$ in $\mathcal{G}_{T_{1},\ldots,T_{d-j-1},T_{d-j+1},\ldots,T_{d-1}}'$, such that $\mathcal{G}^j_i\textbf{y}^j_{\textbf{I}}\to y_{1}^{[d-1]}$ as $i$ goes to infinity. 
By compactness we can assume that $\mathcal{G}^j_i\textbf{y}^j_{\textbf{II}}\to \textbf{z}^j_{\textbf{II}}$ for some $\textbf{z}^j_{\textbf{II}}\in \QQ_{T_{1},\ldots,T_{d-j-1},T_{d-j+1},\ldots,T_{d}}(Y)$.
Let $\tilde{\mathcal{G}}^j_{i}$ be the transformation which consists in applying $\mathcal{G}^j_i$ to both the upper and lower $(d-j)$-th faces of a point in $\QQ_{T_{1},\ldots,T_{d}}(Y)$. We have that $\tilde{\mathcal{G}}^j_{i} \textbf{y}^j$ converges to a point $\textbf{y}^{j+1} \in \QQ_{T_{1},\ldots,T_{d}}(Y)$ whose $(d-j)$-th lower face is equal to $y_{1}^{[d-1]}$ and whose $(d-j)$-th upper face is equal to $\textbf{z}^j_{\textbf{II}}$. 

Let us analyze the properties of $\textbf{y}^{j+1}$. We show that $\textbf{y}^{{j}}$ coincides with $y_1^{[d]}$ in all $k$-lower faces for $d-(j+1)+1=d-j\leq k\leq d$. By construction of $\textbf{y}^{{j}}$, this is true for the $(d-j)$-th lower face. For $d-j<k\leq d$, let $\varepsilon\in \{0,1\}^d$ with $\varepsilon_k=0$ and notice that ${y}^{j+1}_{\varepsilon}$ equals to ${y}^{j+1}_{\varepsilon'}$, for all $\varepsilon'$ differing from  $\varepsilon$ at most in the $d-j$ coordinate, where it is equal to 0 (\emph{i.e}, $\varepsilon'_{d-j}=0$). Since ${y}^{j+1}_{\varepsilon}$ equals to a value in the $(d-j)$-th lower face, we have that it is equal to $y_1$. 
Set $\textbf{x}^{j+1}=\pi^{[d]}(\textbf{y}^{j+1})$ and notice  that since $\textbf{x}^{j}$ is equal to $(x,\ldots,x,z_j)$ we have that $\textbf{x}^{j+1}$ is equal to $(x,\ldots,x,z_{j+1})$ for some $z_{j+1}\in X$. 

\textbf{Claim:} The point  $(x,z_{j+1})$ belongs to  $\overline{\mathcal{O}((x,z_{j}),G_{[2]}^{\Delta})}.$

To see this, remark that the projection of $\tilde{\mathcal{G}}^j_{i} \textbf{y}^j$ onto the coordinates 
$1\ldots0\ldots1$ (where the 0 is in the $d-j$ position) and $1\ldots1$ has the form $(gx,gz_{j})$ and these coordinates converge to $(x,z_{j+1})$. This allows us to conclude. 

Repeating this process we construct the points  $\textbf{y}^{1},\ldots,\textbf{y}^{d} \in \QQ_{T_{1},\ldots,T_{d}}(Y)$ and $\textbf{x}^{1},\ldots,\textbf{x}^{d} \in \QQ_{T_{1},\ldots,T_{d}}(X)$ such that $\pi^{[d]}(\textbf{y}^{j})=\textbf{x}^{j}$ for all $j\in [d]$. The point $\textbf{y}^{d}$ satisfies that it is equal to $y_{1}$ in all $k$-lower faces, $k\in[d]$. That is, $y_{\varepsilon}^{d}=y_{1}$ if $\varepsilon_k=0$ for some $k\in[d]$. That means that $\textbf{y}^{d}$ is equal to $(y_1,\ldots,y_1,y_2)$ for some $y_{2}\in Y$. By Theorem \ref{prop2}, $(y_{1},y_{2})\in \mathcal{R}_{T_{1},\ldots,T_{d}}(Y)$. 

Also, by construction we have that $\pi(y_{2})=z_{d}$ and from the condition $(x,z_{j})\in \overline{\mathcal{O}((x,z_{j-1}),G_{[2]}^{\Delta})}$ for $j\in [d]$, we get that $(x,z_d)\in \overline{\mathcal{O}((x,z),G_{[2]}^{\Delta})}$. Now, by distality, we have that $(x,z)\in \overline{\mathcal{O}((x,z_d),G_{[2]}^{\Delta})}$, and thus there exists a sequence $(g_{i})_{i\geq 1}$ in $G$ such that $(g^{i}x,g^{i}z_d)\to (x,z)$. By compactness we may assume that ${(g^{i}y_{1},g^{i}y_{2})\to (y'_{1},y'_{2})}$ for some $y_1',y_2'\in Y$. Then, as $\mathcal{R}_{T_{1},\ldots,T_{d}}(Y)$ is closed and invariant, the point $(y'_{1},y'_{2})$ belongs to $\mathcal{R}_{T_{1},\ldots,T_{d}}(Y)$. It is immediate that ${\pi\times \pi(y'_{1},y'_{2})=(x,z)}$, finishing the proof.

\end{proof}

\section{Structure of systems with the unique closing parallelepiped property}	
\label{sec:Structure}
In this section we describe the structure of a minimal $\Z^d$-system with the unique closing parallelepiped property (for a given choice of generators). To do so, we need first to introduce some terminology. 

\subsection{The $\QQ_{H}(X)$ relations and classes $\mathsf{Z}_{0}^{H}$}

Let $(X,T_{1},\ldots,T_{d})$ be a $\Z^{d}$-system and let $H$ be a subgroup of $\langle T_1,\ldots,T_d\rangle$. We define 
\[\QQ_{H}(X)=\overline{\left\{(x,hx)\colon x\in X, h \in H \right\}}.\]
With this notation, slightly abusing of the notation, we have that $\QQ_{T_{j}}(X)=\overline{\left\{(x,T_{j}^{n}x)\colon x\in X, n\in \Z\right\}}$ defined in Section \ref{sec:cubesandfaces} is nothing but $\QQ_{\langle T_{j}\rangle}(X)$.

It is easy to see that $\QQ_{H}(X)$ is a closed, $H$-invariant, reflexive and symmetric relation. In \cite{glasner1994topological} Glasner proved that for a $G_{\delta}$-dense subset $X_0$ of $X$, $(X_0\times X_0) \cap  \QQ_{H}(X)$ is transitive. 
But in general this relation is not an equivalence relation (for an example see \cite{tu2013dynamical}). Nevertheless, transitivity of $\QQ_{H}(X)$ is satisfied when the system is distal, because it is a particular case of the gluing property: if $(x,y),(y,z)\in \QQ_{H}(X)$ then $(x,z) \in \QQ_{H}(X)$. The proof of the gluing property in this context is almost a verbatim copy of the one of Lemma \ref{pegado1} so we omit it.
\begin{proposition}
	\label{cube3} 
	Let $(X,T_{1},\ldots,T_{d})$ be a minimal distal $\Z^{d}$-system.  
	Then, for every subgroup $H$ of $\langle T_1,\ldots,T_d\rangle$ we have that $\QQ_{H}(X)$ is a closed and invariant equivalence relation of $X$.
\end{proposition}

We now define the following classes of $\Z^d$-systems that will be useful to characterize systems with the unique closing parallelepiped property. For $H$ as before, let
$$\mathsf{Z}_{0}^{H} = \{(X,T_{1},\ldots,T_{d})\colon H\ \text{acts as the identity in } X\}.$$
This definition and notation is analogous to the one used by Austin in \cite{austin2010multiple} when studying the $L^2$ convergence of multiple ergodic averages in the measure theoretical setting. Using a Zorn's Lemma argument it can be proved  that every $\Z^d$-system  possesses a maximal $\mathsf{Z}_{0}^{H}$-factor \cite[Chapter 9]{auslander1988minimal}. That is, any factor of a given $\Z^d$-system in the class $\mathsf{Z}_{0}^{H}$ factorizes through its maximal $\mathsf{Z}_{0}^{H}$-factor. In addition, this factor can be characterized precisely. To this purpose, 
for a minimal $\Z^{d}$-system $(X,T_{1},\ldots,T_{d})$ let $\sigma_{H}(X)$ denote the smallest closed and invariant equivalence relation containing $\QQ_{H}(X)$.

\begin{lemma}
	\label{lemma3}
	Let $(X,T_{1},\ldots,T_{d})$ be a minimal $\Z^{d}$-system. Then, $({X}/{\sigma_{H}(X)},T_1,\ldots, T_d)$ is the maximal $\mathsf{Z}_{0}^{H}$-factor of $(X,T_{1},\ldots,T_{d})$.
\end{lemma}
\begin{proof}
	Let $\pi:X\to {X}/{\sigma_{H}(X)}$ be the natural factor map. First we prove that $({X}/{\sigma_{H}(X)},T_{1},\ldots,T_{d})\in \mathsf{Z}_{0}^{H}$. Let $h\in H$, if $[\cdot ]$ denotes the equivalence class of a point for the relation $\sigma_{H}(X)$, then $h[x]=[hx]$. Since $(x,hx)\in \sigma_{H}(X)$ we get that $h[x]=[x]$. So $H$ acts trivially on ${X}/{\sigma_{H}(X)}$.
	Now, let $(Z,T_1,\ldots,T_d)$ be a factor of $(X,T_1,\ldots,T_d)$ that belongs to $\mathsf{Z}_{0}^{H}$ and let $\pi'\colon X\to Z$ be the corresponding factor map. Let $R_{\pi'}=\{(x,y)\in X\colon \pi'(x)=\pi'(y)\}$. We have to prove that $\sigma_{H}(X)\subseteq R_{\pi'}$. In fact, since $R_{\pi'}$ is a closed and invariant equivalence relation, it suffices to prove that $\QQ_{H}(X)\subseteq R_{\pi'}$.  
	Since $(Z,T_1,\ldots,T_d) \in \mathsf{Z}_{0}^{H}$ we have that for every $h\in H$, $\pi'\circ h=h\circ \pi'=\pi'$. Hence, for any $x \in X$, $\{(x,hx)\colon h\in H \}\subseteq R_{\pi'}$, and thus ${\QQ_{H}(X)\subseteq R_{\pi'}}$ as desired.
\end{proof}

In the distal case, by Proposition \ref{cube3}, $\QQ_{H}(X)$ is an equivalence relation and thus $\QQ_{H}(X)=\sigma_{H}(X)$. In this case Lemma \ref{lemma3} can be stated as,
\begin{corollary}
	\label{equivdistal}
	Let $(X,T_{1},\ldots,T_{d})$ be a minimal distal $\Z^{d}$-system. Then, $({X}/{\QQ_{H}(X)}, T_{1},\ldots,T_{d})$ is the maximal $\mathsf{Z}_{0}^{H}$-factor of $(X,T_{1},\ldots,T_{d})$.
\end{corollary}

This result generalizes to iterated quotients. 
\begin{proposition}
	Let $(X,T_1,\ldots,T_d)$ be a minimal distal $\Z^d$-system and let $H$, $H_1$ and $H_2$ be subgroups of 
	$\langle T_1,\ldots,T_d\rangle$ such that $H=H_1H_2$. 
	Let $Y=X/\QQ_{H_1}(X)$. 
	Then, $(X/\QQ_{H}(X),T_1,\ldots,T_d)$ is isomorphic to $(Y/\QQ_{H_2}(Y),T_1,\ldots,T_d)$.
\end{proposition}
\begin{proof} Let $\pi_1\colon X\to Y (=X/\QQ_{H_1}(X))$ and  $\pi\colon X\to Y/\QQ_{H_2}(Y)$ be the natural factor maps and $R_{\pi}=\{(x,x'): \pi(x)=\pi(x') \}$. 
	We have that $Y/\QQ_{H_2}(Y)$ is a factor of $X$ where the action of both $H_1$ and $H_2$ are trivial. Hence, the action of $H$ is trivial as well. It follows that $\QQ_{H}(X)\subseteq R_{\pi}$. For the converse inclusion, assume that $(x,x')\in R_{\pi}$. Then $(\pi_1(x),\pi_1(x'))\in \QQ_{H_2}(Y)$.  
	
	{\bf Claim:} there exists $x'' \in X$ such that $(x,x'')\in \QQ_{H_2}(X)$ and $\pi_1(x'')=\pi_1(x')$. 
	
	To prove the claim, consider $\epsilon>0$ and let $\epsilon>\delta>0$  so that $\pi(B(x,\epsilon))$ contains $B(\pi_1(x),\delta)$. This can be done by the openness of $\pi_1$ (recall the systems are distal). Since $(\pi_1(x),\pi_1(x'))\in \QQ_{H_2}(Y)$, there exists $\tilde{x}\in X$ and $h_2=h_2(\delta) \in H_2$ such that $\pi_1(\tilde{x})$ is $\delta$ close to $\pi_1(x)$ and $h_2\pi_1(\tilde{x})=\pi_1(h_2\tilde{x})$ is $\delta$ close to $\pi_1(x')$. We can find a point $\tilde{x}'$ that is $\epsilon$ close to $x$ such that $\pi_1(\tilde{x}')=\pi_1(\tilde{x})$ and so $\pi_1(h_2\tilde{x}')$ is $\delta$ close to $\pi_1(x')$. A compactness argument allows us to find $x''$ as a limit point of $h_2\tilde{x}'$,  satisfying then the claim statement. 
	
	Using the claim, notice that $(x,x'')\in \QQ_{H_2}(X)\subseteq \QQ_{H}(X)$ and $(x'',x')\in \QQ_{H_1}(X)\subseteq \QQ_{H}(X)$. Since $\QQ_{H}(X)$ is an equivalence relation we conclude that $(x,x')\in \QQ_{H}(X)$, finishing the proof.
	
\end{proof}

To ease notations, we let $X/\QQ_{H_1}/\QQ_{H_2}$ denote the quotient  $Y/\QQ_{H_2}(Y)$, where $Y=X/\QQ_{H_1}(X)$. We naturally extend this notation to iterated quotients $X/\QQ_{H_1}/\QQ_{H_2}/\cdots /\QQ_{H_n}$ that are defined in the obvious way. 

Notice that if $\{j_{1},\ldots,j_{k}\}\subseteq [d]$ we have that
$$\QQ_{\left\langle T_{j_{1}},\ldots,T_{j_{k}}\right\rangle}(X)=\overline{\{(x,hx)\colon x\in X,\ h\in \left\langle T_{j_{1}},\ldots,T_{j_{k}}\right\rangle\}}.$$
We warn the reader that the sets $\QQ_{\left\langle T_{j_{1}},\ldots,T_{j_{k}}\right\rangle}(X)$ and $\QQ_{T_{j_{1}},\ldots, T_{j_{k}}}(X)$ are different. The first one is a subset of $X^2$ while the second one is a subset of $X^{[k]}$.  

Following the notation in \cite{austin2010multiple} we denote by $\mathsf{Z}_{0}^{e_{j}}$  the class $\mathsf{Z}_{0}^{\left\langle T_{j}\right\rangle}$ and by $\mathsf{Z}_{0}^{e_{j_{1}} \wedge e_{j_{2}}}$ the intersection of the classes $\mathsf{Z}_{0}^{e_{j_{1}}}$ and $\mathsf{Z}_{0}^{e_{j_{2}}}$, which corresponds to the class $\mathsf{Z}_{0}^{\left\langle T_{j_{1}},T_{j_{2}}\right\rangle}$.

\subsection{Dynamical structure of systems with the unique closing parallelepiped property}

We describe first a soft structure theorem for minimal systems with the unique closing parallelepiped property.

\begin{proposition} \label{prop:softstructure}
	Let $(X,T_{1},\ldots,T_{d})$ be a minimal (not necessarily distal) $\Z^{d}$-system with the unique closing parallelepiped property. Then, 
	\begin{enumerate}
		\item for each $j\in [d]$ there exists a $\Z^d$-system $(Y_j,T_{1},\ldots,T_{d})$ such that $T_j$ is the identity transformation on $Y_j$;
		\item there exists a joining $(Y,T_1,\ldots,T_d)$ of the systems $(Y_1,T_1,\ldots,T_d),\ldots,(Y_d,T_1,\ldots,T_d)$ which is an extension of $(X,T_1,\ldots,T_d)$. 
	\end{enumerate}
Otherwise saying, $(X,T_1,\ldots,T_d)$ 
has an extension which is a joining of systems, where on each one a different transformation acts as the identity.
\end{proposition} 
\begin{figure}[H]
	\begin{center}
		\begin{tikzpicture}
		\matrix (m) [matrix of math nodes,row sep=3em,column sep=4em,minimum width=2em,ampersand replacement=\&]
		{
			\& (X,T_1,\ldots,T_d) \& (Y,T_1,\ldots,T_d) \&   \\
			(Y_1,,T_1,\ldots,T_d) \& (Y_2,T_1,\ldots,T_d) \& \cdots \& (Y_d,T_1,\ldots,T_d)  \\};
		\path[-stealth]
		(m-1-3) edge node[above] {$ $} (m-1-2)
		(m-1-3) edge node[above] {$ $} (m-2-1)
		(m-1-3) edge node[above] {$ $} (m-2-2)
		(m-1-3) edge node[above] {$ $} (m-2-3)
		(m-1-3) edge node[above] {$ $} (m-2-4)
		;     
		\end{tikzpicture}
		\caption{Structure of a minimal system $(X,T_1,\ldots,T_d)$ with the unique closing parallelepiped property}
	\end{center}
\end{figure}
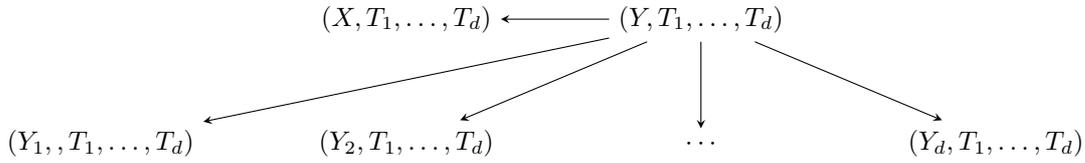

\begin{proof}
	Let $(X,T_1,\ldots,T_d)$ be a  minimal $\Z^{d}$-system and let $x_{0}\in X$.  Consider the $\Z^d$-system $(\KK_{T_{1},\ldots,T_{d}}^{x_{0}},$ $T_1^{[d]},\ldots,T_d^{[d]})$ (recall the face transformations defined in  Section \ref{sec:cubesandfaces}). Such system is similar to the one considered in \cite{donoso2014dynamical} when $d=2$.
	Let $Y\subseteq \KK_{T_{1},\ldots,T_{d}}^{x_{0}}$ be a minimal subsystem (note that in the distal case we have that $Y=\KK_{T_{1},\ldots,T_{d}}^{x_{0}}$). Here we consider the natural factor map from $(Y,T_1^{[d]},\ldots,T_d^{[d]})$ to $(X,T_1,\ldots,T_d)$ given by the projection onto the last coordinate of points in $Y$.  We analyze how $Y$ looks like when $(X,T_1,\ldots,T_d)$ has the unique closing parallelepiped property and show that $Y$ can be ``naturally decomposed'' as a joining of some factors where on each one the action of one transformation is the identity. That is, $Y$ is a joining of lower dimensional actions.
	
We proceed as follows. Points in $Y$ have $2^d-1$ coordinates and since $(X,T_1,\ldots,T_d)$ has the unique closing parallelepiped property, the last one is determined by the others. Hence, we can think of points in $Y$ having only $2^d-2$ coordinates. 
	
	\begin{figure}[H]\label{fig:constY}
		\begin{tikzpicture}[scale=1.3]
		\node(s) at (0,2) [scale=1] {\large$\begin{array}{cccccc:c}\text{ A point in } Y: & y_{10\ldots 0} & \cdots  & y_{0\ldots 01} & \cdots  & y_{1\ldots 1 0} & y_{1\cdots 1}\\ T_1^{[d]}:& T_{1} & \cdots & \id & \cdots & T_{1} & T_{1}\\  & \vdots & \ldots & \vdots & \cdots & \vdots & \vdots \\
			T_{d-1}^{[d]}  & \id & \cdots & \id & \cdots & T_{d-1} & T_{d-1}\\ T_{d}^{[d]}& \id & \cdots & T_{d} & \cdots & \id & T_{d} \end{array}$};
		\node(s1) at (0.51,1) [scale=.1] {};
		\end{tikzpicture}
		\caption{How a point in $Y$ and the face transformations look like.}
\end{figure}
In Figure \ref{fig:constY},	the dashed vertical line separates the last coordinate of a point in $Y$ that is determined by the others. From now on we omit this coordinate. Now we describe how to build the factors $Y_1,\ldots,Y_d$ whose joining is $Y$.
	
	For $j\in [d]$, let $Y_j$ be the projection of $Y$ onto the coordinates where $T_j^{[d]}$ is trivial, \emph{i.e.}, onto the coordinates $\varepsilon \in \{0,1\}^{d}\setminus \{\vec{0}\}$ such that $\varepsilon_j=0$.  We leave the reader to verify that $Y_j$ is a subset of $\KK_{T_{1},\ldots,T_{j-1},T_{j+1},\ldots,T_{d}}^{x_{0}}$ (the face orbit where we forget the transformation $T_j$).
	
	Because we are forgetting the  last coordinate,  for all coordinates of a point in $Y$ at least one of the transformations acts as the identity. It follows then that the projections of a point ${\bf y}  \in Y$ into the factors $Y_j$, $j\in [d]$, determine the point ${\bf y}$ itself, {\em i.e.,} $Y$ is a joining of the systems $Y_j$, $j\in[d]$. This finishes the proof.  
\end{proof}

Remark that by definition, for $j\in[d]$ the system  $(Y_j,T_1,\ldots,T_d)$  belongs to the class $\mathsf{Z}_{0}^{e_{j}}$. What next result shows is that in the distal  case $(Y_j,T_1,\ldots,T_d)$ is actually the maximal factor of $(Y,T_1,\ldots,T_d)$ with respect to this class.  	

\begin{proposition}
	\label{cube2}
	Let $(X,T_{1},\ldots,T_{d})$ be a minimal distal $\Z^{d}$-system and $x_{0}\in X$. Then, for any $j\in [d]$,  $(\KK_{T_{1},\ldots,T_{j-1},T_{j+1},\ldots,T_{d}}^{x_{0}},T_1^{[d]},\ldots,T_d^{[d]})$ is isomorphic to \large$(\KK_{T_{1},\ldots,T_{d}}^{x_{0}}/{\QQ_{T_{j}^{[d]}}(\KK_{T_{1},\ldots,T_{d}}^{x_{0}})},T_1^{[d]},\ldots,T_d^{[d]})$\normalsize.
	In particular, $(\KK_{T_{1},\ldots,T_{j-1},T_{j+1},\ldots,T_{d}}^{x_{0}},T_1^{[d]},\ldots,T_d^{[d]})$ is the maximal $\mathsf{Z}_{0}^{e_{j}}$-factor of $(\KK_{T_{1},\ldots,T_{d}}^{x_{0}},T_1^{[d]},\ldots,T_d^{[d]})$.
\end{proposition}

\begin{proof}
	We follow the notation introduced in Proposition \ref{prop:softstructure}. We start recalling that in the distal case, by minimality we have that $Y=\KK_{T_{1},\ldots,T_{d}}^{x_{0}}$ and $Y_j=\KK_{T_{1},\ldots,T_{j-1},T_{j+1},\ldots,T_{d}}^{x_{0}}$ for every $j\in [d]$. 
	Let $j\in [d]$. We know by Corollary \ref{equivdistal} that $({\KK_{T_{1},\ldots,T_{d}}^{x_{0}}}/{\QQ_{T_{j}^{[d]}}(\KK_{T_{1},\ldots,T_{d}}^{x_{0}}),T_1^{[d]},\ldots,T_d^{[d]})}$ is the maximal $\mathsf{Z}_{0}^{e_{j}}$-factor of $(\KK_{T_{1},\ldots,T_{d}}^{x_{0}},$ $T_1^{[d]},\ldots,T_d^{[d]})$  and since  $Y_j\in \mathsf{Z}_{0}^{e_{j}}$ we only need to show that if  $\textbf{x},\textbf{y}\in \KK_{T_{1},\ldots,T_{d}}^{x_{0}}$ are such that $$(x_{\varepsilon}\colon \varepsilon \in \{0,1\}^{d}\setminus\{\vec{0}\},\, \varepsilon_j=0)={(y_{\varepsilon}\colon \varepsilon \in \{0,1\}^{d}\setminus\{\vec{0}\},\, \varepsilon_j=0)}$$
(\emph{i.e.}, they have the same projection on $Y_j$), then $(\textbf{x},\textbf{y})\in {\QQ_{T_{j}^{[d]}}(\KK_{T_{1},\ldots,T_{d}}^{x_{0}})}$. 

Let $\delta >0$ and let $\pi\colon \KK_{T_{1},\ldots,T_{d}}^{x_{0}}\to \KK_{T_{1},\ldots,T_{j-1},T_{j+1},\ldots,T_{d}}^{x_{0}}$ be the natural projection. By the openness of $\pi$ (recall the systems are distal), we can find $0<\delta'<\delta$ such that
	\begin{equation}\label{ProjectionFacej}
	B_{\rho_{X_{*}^{[d-1]}}}((x_{\varepsilon}\colon \varepsilon \in \{0,1\}^{d},\, \varepsilon_j=0),\delta')\subseteq \pi\left (B_{\rho_{X_{*}^{[d]}}}(\textbf{x},\delta)\right ) \cap   \pi\left(B_{\rho_{X_{*}^{[d]}}}(\textbf{y},\delta)\right) .
	\end{equation}
	Here we recall that if $\rho$ is a metric on $X$, 	$\rho_{X_{\ast}}$ is any metric that generates the product topology on $X^{[d]}_{\ast}=X^{2^d-1}$ (for instance the sum of the metrics in each coordinate).

	By definition of $\KK_{T_{1},\ldots,T_{d}}^{x_{0}}$, there exists $\textbf{n}\in \Z^{d}$ such that
	$$\rho_{X_{*}^{[d]}}((T_{1}^{n_{1}\varepsilon_{1}}\ldots T_{d}^{n_{d}\varepsilon_{d}}x_{0})_{\varepsilon \in \{0,1\}^{d}\setminus\{\vec{0}\}},\textbf{x})<\delta'.$$
	
	Call $\textbf{z}=(T_{1}^{n_{1}\varepsilon_{1}}\ldots T_{d}^{n_{d}\varepsilon_{d}}x_{0})_{\varepsilon \in \{0,1\}^{d}\setminus\{\vec{0}\}}$. Then, we have that
	$$\rho_{X_{*}^{[d-1]}}((z_{\varepsilon}\colon \varepsilon \in \{0,1\}^{d}\setminus\{\vec{0}\},\ \varepsilon_j=0),(x_{\varepsilon}\colon \varepsilon \in \{0,1\}^{d}\setminus\{\vec{0}\},\, \varepsilon_j=0))<\delta'.$$
	
	Thus, by \eqref{ProjectionFacej}, there exists $\textbf{z}^{1}\in \KK_{T_{1},\ldots,T_{d}}^{x_{0}}$ such that
	\[(z_{\varepsilon}^{1}\colon \varepsilon \in \{0,1\}^{d}\setminus\{\vec{0}\},\ \varepsilon_j=0)=(z_{\varepsilon}\colon \varepsilon \in \{0,1\}^{d}\setminus\{\vec{0}\},\, \varepsilon_j=0)\] and $\rho_{X_{*}^{[d]}}(\textbf{z}^{1},\textbf{y})<\delta.$
	
	Let $0<\delta''<\delta'$ be such that $\delta'+\delta''<\delta$ and for any $\textbf{u},\textbf{v}\in \KK_{T_{1},\ldots,T_{d}}^{x_{0}}$ we have
	$$\rho_{X_{*}^{[d]}}(\textbf{u},\textbf{v})<\delta'' \implies \rho_{X_{*}^{[d]}}((T_{j}^{[d]})^{n_{j}}\textbf{u},(T_{j}^{[d]})^{n_{j}}\textbf{v})<\delta'.$$
	
	By definition of $\KK_{T_{1},\ldots,T_{d}}^{x_{0}}$, there exists $\textbf{n}'\in \Z^{d}$ such that
	$$\rho_{X_{*}^{[d]}}((T_{1}^{n_{1}'\varepsilon_{1}}\ldots T_{d}^{n_{d}'\varepsilon_{d}}x_{0})_{\varepsilon \in \{0,1\}^{d}\setminus\{\vec{0}\}},\textbf{z}^{1})<\delta''.$$
Call $\textbf{z}^{2}=(T_{1}^{n_{1}'\varepsilon_{1}}\ldots T_{d}^{n_{d}'\varepsilon_{d}}x_{0})_{\varepsilon \in \{0,1\}^{d}\setminus\{\vec{0}\}}$ (\emph{i.e.}, $\rho_{X_{*}^{[d]}}(\textbf{z}^{2},\textbf{z}^{1})<\delta''$) and define $\textbf{z}^{3}=(T_{j}^{[d]})^{n_{j}-n_{j}'}\textbf{z}^{2}$.

Notice that the $j$-th lower face of $\textbf{z}^3$ coincides with the one of $\textbf{z}^2$ which is $\delta''+\delta $ close to the one of $\textbf{x}$ (here we are slightly abusing terminology, the $j$-th lower face should consider the coordinate $\vec{0}$, that we do not have in points of $\KK_{T_{1},\ldots,T_{d}}^{x_{0}}$, but let us use this word to avoid introducing more terminology). On the other hand, the $j$-th upper face of $(T_{j}^{[d]})^{-n_{j}'}\textbf{z}^{2}$ coincides with the $j$-lower face of $\textbf{z}^{2}$ (with $x_0$ in the $\vec{0}$ coordinate) and then it is $\delta''$ close to the $j$-th lower face of $\textbf{z}^1$ and thus to the $j$-th lower face of $\textbf{z}$. By the definition of $\delta''$, and using that the $j$-th upper face of $\textbf{z}$ is obtained by applying $(T_d^{[d]})^{n_d}$ to its $j$-th lower face we get that \[\textbf{z}^3=(T_{j}^{[d]})^{n_j}(T_{j}^{[d]})^{-n_{j}'}\textbf{z}^{2} \text{  is } 3\delta \text{ close to } \textbf{x}.\] 
	  
Now, remark that the point 	 $(T_{j}^{[d]})^{n_{j}'-n_j}\textbf{z}^3=\textbf{z}^{2}$ is $\delta''+\delta<2\delta$ close to $\textbf{y}$. As $\delta>0$ is arbitrary, we have shown that $(\textbf{x},\textbf{y})\in \QQ_{T_{j}^{[d]}}(\KK_{T_{1},\ldots,T_{d}}^{x_{0}})$ and therefore the  factor map $\pi:\KK_{T_{1},\ldots,T_{d}}^{x_{0}}/{\QQ_{T_{j}^{[d]}}(\KK_{T_{1},\ldots,T_{d}}^{x_{0}})} \to \KK_{T_{1},\ldots,T_{j-1},T_{j+1},\ldots,T_{d}}^{x_{0}}$ is an isomorphism.
\end{proof}

We use Theorem \ref{prop2} to prove the following theorem, which gives deeper information about the structure of the systems with the unique closing parallelepiped property, introducing independence with respect to further factors.

To state the theorem, let us introduce another definition. Let $(X,T_{1},\ldots,T_{d})$ be a $\Z^{d}$-system. We say that $x_0 \in X$ is a continuity point if $\KK_{T_{1},\ldots,T_{d}}^{x_{0}}$ coincides with the set of ${\bf x}\in X^{[d]}_{\ast}$ such that $(x_0,{\bf x})\in \QQ_{T_{1},\ldots,T_{d}}(X)$. The set of continuity points $x_0\in X$ is a dense $G_{\delta}$ set of points of $X$ \cite[Lemma 4.5]{glasner1994topological}. Note that if $x_0$ is a continuity point, then  any point in its orbit, \emph{i.e.}, $T_1^{n_1}\cdots T_d^{n_d}x_0$ with $n_1\ldots,n_d\in \Z$ is also a continuity point. For a continuity point $x_0$, we have that $\QQ_{T_j}(x_0)= \overline{\mathcal{O}(x_0,T_j)}$ for $j\in [d]$.

\begin{theorem} 
	\label{KRelInd}
	Let $(X,T_{1},\ldots,T_{d})$ be a minimal distal $\Z^{d}$-system and  $x_{0}\in X$ a continuity point. Suppose that $X$ has the unique closing parallelepiped property. Then, $\KK_{T_{1},\ldots,T_{d}}^{x_{0}}$ is a joining of its factors $\KK_{T_{1},\ldots,T_{j-1},T_{j+1},\ldots,T_{d}}^{x_{0}}$, $j\in[d]$, and it is relatively independent with respect to their maximal $\mathsf{Z}_{0}^{e_{j_{1}}}\wedge \mathsf{Z}_{0}^{e_{j_{2}}}$-factors, for all $j_{1},j_{2}\in [d]$ with $j_{1} < j_{2}$.
\end{theorem} 

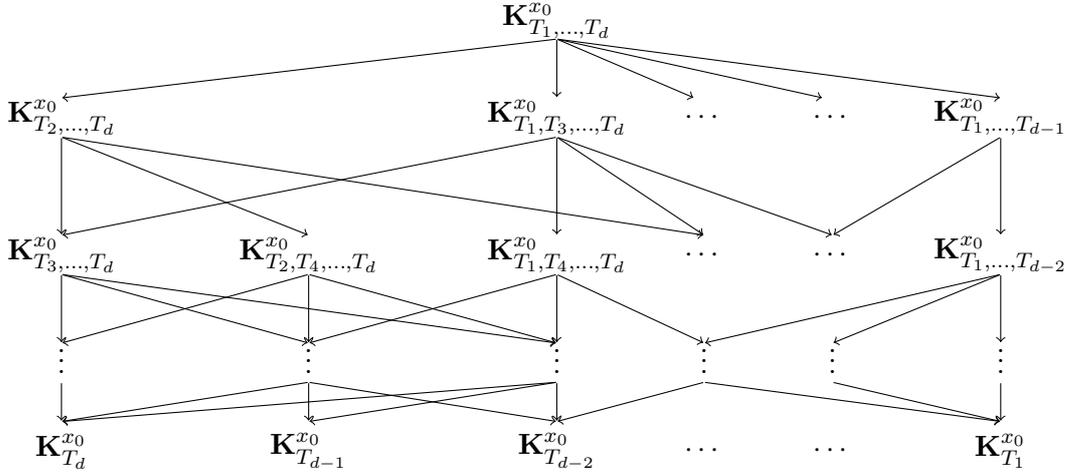
\begin{figure}[H]
\centering
\begin{tikzpicture}[scale=1.3]
\node(s) at (0.51,1.2) [scale=1] {\large$\begin{array}{c}\KK_{T_{1},\ldots,T_{d}}^{x_{0}}  \end{array}$};
	\node(s1) at (0.51,1) [scale=.1] {};
	
	\node(a) at (-4.5,0.2) [scale=1] {\large$\KK_{T_{2},\ldots,T_{d}}^{x_{0}}$};
	\node(a1) at (-4.5,0.4) [scale=.1] {};
	\node(a2) at (-4.5,0) [scale=.1] {};
	
	\node(b) at (0.51,0.2) [scale=1] {\large$\KK_{T_{1},T_{3},\ldots,T_{d}}^{x_{0}}$};
	\node(b1) at (0.51,0.4) [scale=.1] {};
	\node(b2) at (0.51,0) [scale=.1] {};
	
	\node(c) at (5,0.2) [scale=1] {\large$\KK_{T_{1},\ldots,T_{d-1}}^{x_{0}}$};
	\node(c1) at (5,0.4) [scale=.1] {};
	\node(c2) at (5,0) [scale=.1] {};
	
	\node(d) at (2,0.2) [scale=1] {\large$\cdots$};
	\node(d1) at (1.9,0.4) [scale=.1] {};
	
	\node(e) at (3.3,0.2) [scale=1] {\large$\cdots$};	
	\node(e1) at (3.2,0.4) [scale=.1] {};
	
	\node(f) at (-4.5,-1.2) [scale=1] {\large$\KK_{T_{3},\ldots,T_{d}}^{x_{0}}$};
	\node(f1) at (-4.5,-1) [scale=.1] {};
	\node(f2) at (-4.5,-1.4) [scale=.1] {};
	
	\node(g) at (-2,-1.2) [scale=1] {\large$\KK_{T_{2},T_{4},\ldots,T_{d}}^{x_{0}}$};
	\node(g1) at (-2,-1) [scale=.1] {};
	\node(g2) at (-2,-1.4) [scale=.1] {};
	
	\node(h) at (0.51,-1.2) [scale=1] {\large$\KK_{T_{1},T_{4},\ldots,T_{d}}^{x_{0}}$};
	\node(h1) at (0.51,-1) [scale=.1] {};
	\node(h2) at (0.51,-1.4) [scale=.1] {};
	
	\node(i) at (2,-1.2) [scale=1] {\large$\cdots$};
	\node(i1) at (2,-1) [scale=.1] {};
	
	\node(j) at (3.3,-1.2) [scale=1] {\large$\cdots$};
	\node(j1) at (3.3,-1) [scale=.1] {};
	
	\node(k) at (5,-1.2) [scale=1] {\large$\KK_{T_{1},\ldots,T_{d-2}}^{x_{0}}$};
	\node(k1) at (5,-1) [scale=.1] {};
	\node(k2) at (5,-1.4) [scale=.1] {};
	
	\node(l) at (-4.5,-2.2) [scale=1] {\large$\vdots$};
	\node(l1) at (-4.5,-2.1) [scale=.1] {};
	\node(l2) at (-4.5,-2.5) [scale=.1] {};
	
	\node(m) at (-2,-2.2) [scale=1] {\large$\vdots$};
	\node(m1) at (-2,-2.1) [scale=.1] {};
	\node(m2) at (-2,-2.5) [scale=.1] {};
	
	\node(n) at (0.51,-2.2) [scale=1] {\large$\vdots$};
	\node(n1) at (0.51,-2.1) [scale=.1] {};
	\node(n2) at (0.51,-2.5) [scale=.1] {};
	
	\node(o) at (2,-2.2) [scale=1] {\large$\vdots$};
	\node(o1) at (2,-2.1) [scale=.1] {};
	\node(o2) at (2,-2.5) [scale=.1] {};
	
	\node(p) at (3.3,-2.2) [scale=1] {\large$\vdots$};
	\node(p1) at (3.3,-2.1) [scale=.1] {};
	\node(p2) at (3.3,-2.5) [scale=.1] {};
	
	\node(q) at (5,-2.2) [scale=1] {\large$\vdots$};
	\node(q1) at (5,-2.1) [scale=.1] {};
	\node(q2) at (5,-2.5) [scale=.1] {};
	
	\node(r) at (-4.5,-3.2) [scale=1] {\large$\KK_{T_{d}}^{x_{0}}$};
	\node(r1) at (-4.5,-2.9) [scale=.1] {};
	
	\node(s) at (-2,-3.2) [scale=1] {\large$\KK_{T_{d-1}}^{x_{0}}$};
	\node(s11) at (-2,-2.9) [scale=.1] {};
	
	\node(t) at (0.51,-3.2) [scale=1] {\large$\KK_{T_{d-2}}^{x_{0}}$};
	\node(t1) at (0.51,-2.9) [scale=.1] {};
	
	\node(u) at (2,-3.2) [scale=1] {\large$\cdots$};
	\node(u1) at (2,-2.9) [scale=.1] {};
	
	\node(v) at (3.3,-3.2) [scale=1] {\large$\cdots$};
	\node(v1) at (3.3,-2.9) [scale=.1] {};
	
	\node(w) at (5,-3.2) [scale=1] {\large$\KK_{T_{1}}^{x_{0}}$};
	\node(w1) at (5,-2.9) [scale=.1] {};
	
	\path[thin,->] (s1) edge (a1);
	\path[thin,->] (s1) edge (b1);
	\path[thin,->] (s1) edge (c1);
	\path[thin,->] (s1) edge (d1);
	\path[thin,->] (s1) edge (e1);
	\path[thin,->] (a2) edge (f1);
	\path[thin,->] (a2) edge (g1);
	\path[thin,->] (a2) edge (i1);
	\path[thin,->] (b2) edge (f1);
	\path[thin,->] (b2) edge (h1);
	\path[thin,->] (b2) edge (i1);
	\path[thin,->] (b2) edge (j1);
	\path[thin,->] (c2) edge (j1);
	\path[thin,->] (c2) edge (k1);
	\path[thin,->] (f2) edge (l1);
	\path[thin,->] (f2) edge (m1);
	\path[thin,->] (f2) edge (n1);
	\path[thin,->] (g2) edge (l1);
	\path[thin,->] (g2) edge (m1);
	\path[thin,->] (g2) edge (n1);
	\path[thin,->] (h2) edge (m1);
	\path[thin,->] (h2) edge (n1);
	\path[thin,->] (h2) edge (o1);
	\path[thin,->] (k2) edge (o1);
	\path[thin,->] (k2) edge (p1);
	\path[thin,->] (k2) edge (q1);
	\path[thin,->] (l2) edge (r1);
	\path[thin,->] (m2) edge (r1);
	\path[thin,->] (m2) edge (s11);
	\path[thin,->] (m2) edge (t1);
	\path[thin,->] (n2) edge (r1);
	\path[thin,->] (n2) edge (s11);
	\path[thin,->] (n2) edge (t1);
	\path[thin,->] (o2) edge (t1);
	\path[thin,->] (o2) edge (w1);
	\path[thin,->] (p2) edge (w1);
	\path[thin,->] (q2) edge (w1);
	\end{tikzpicture}
	\caption{Decomposition of the system $(\KK_{T_{1},\ldots,T_{d}}^{x_{0}},T^{[d]}_{1},\ldots,T^{[d]}_{d})$ as the joining of the systems $\KK_{T_{1},\ldots,T_{j-1},T_{j+1},\ldots,T_{d}}^{x_{0}}$. Theorem \ref{KRelInd} asserts that the joining is relatively independent with respect to the factors in the ``second level'', \emph{i.e.}, $\KK_{T_{i_1},\ldots,T_{i_{d-2}}}^{x_{0}}$, $1\leq i_1<\cdots <i_{d-2}\leq d$. In the diagram further factors are also shown. A system in the $k$th-level is an action of $\Z^{d-k}$ ($k$ transformations are the identity, and  $\KK_{T_{1},\ldots,T_{d}}^{x_{0}}$ corresponds to the 0 level).}
\end{figure}

\begin{proof} 
	
	The fact that $\KK_{T_{1},\ldots,T_{d}}^{x_{0}}$ is a joining of its factors $\KK_{T_{1},\ldots,T_{j-1},T_{j+1},\ldots,T_{d}}^{x_{0}}$, $j\in [d]$, was already shown in Proposition  \ref{prop:softstructure} and Proposition \ref{cube2}, so we are left to show the relatively independence part. 
	
	Let $\textbf{x},\textbf{y}\in X_{*}^{[d]}$ be such that:
	\begin{enumerate}
		\item $\textbf{x}\in \KK_{T_{1},\ldots,T_{d}}^{x_{0}}$. \label{item:cond1}
		\item For every $j\in [d]$, $(y_{\varepsilon}\colon \varepsilon \in \{0,1\}^{d}\setminus\{\vec{0}\},\ \varepsilon_{j}=0)\in \KK_{T_{1},\ldots,T_{j-1},T_{j+1},\ldots,T_{d}}^{x_{0}}$.
		\item For every $j_{1},j_{2}\in [d]$, $j_{1}\neq j_{2}$, $$(x_{\varepsilon}\colon \varepsilon \in \{0,1\}^{d}\setminus \{\vec{0}\},  \varepsilon_{j_{1}}=0,\ \varepsilon_{j_{2}}=0)=(y_{\varepsilon}\colon \varepsilon \in \{0,1\}^{d}\setminus \{\vec{0}\},  \varepsilon_{j_{1}}=0,\ \varepsilon_{j_{2}}=0),$$
		\noindent {\em i.e.,} $\textbf{x}$ and $\textbf{y}$ coincide in the coordinates associated to the maximal $\mathsf{Z}_{0}^{e_{j_{1}}\wedge \ e_{j_{2}}}$-factor. 
	\end{enumerate}
	Roughly speaking, this condition says that given $\textbf{x}\in \KK_{T_{1},\ldots,T_{d}}^{x_{0}}$, one considers its projection into the systems of the ``second level'' and the relative independence means that one can freely change all the coordinates of $\textbf{x}$ of the form $[d]\setminus\{j\}$ with $j\in [d]$, as long as the resulting point (the point \textbf{y} in the statement)  has valid faces. We leave the reader to verify that these are the conditions we need to look at to check the relative independence.

	So, given (i), (ii) and (iii), we want to show that $\textbf{y}\in \KK_{T_{1},\ldots,T_{d}}^{x_{0}}$. By (iii), the only coordinates $\varepsilon\subseteq [d]$, $\varepsilon \neq \emptyset$, such that $x_{\varepsilon}\neq y_{\varepsilon}$ are the coordinates $\varepsilon=[d]$ and $\varepsilon=[d]\setminus \{j\}$ for $j\in [d]$. The program of the proof is to construct points $\textbf{x}^{0},\textbf{x}^{1},\ldots,\textbf{x}^{d}\in \KK_{T_{1},\ldots,T_{d}}^{x_{0}}$ with $\textbf{x}^{0}=\textbf{x}$ and such that for every $k\in [d]$
	$$x_{\varepsilon}^{k}=\left\{\begin{array}{ll} x_{\varepsilon} & \varepsilon\subseteq [d],\ \varepsilon\neq [d]\setminus\{ j\} \text{ for some }\ j \in [d],\\
	y_{\varepsilon} & \varepsilon =[d]\setminus\{ j \} \text{ for some }\ j \in \{1,\ldots,k\},\\
	a^k & \varepsilon=[d],
	\end{array}\right.$$ 
	
	\noindent for some $a^k\in X$. Notice that since $X$ has the unique closing parallelepiped property,  the point $a^k$ is unique. The point $\textbf{x}^{d}$ coincides with $\textbf{y}$ in all coordinates $\varepsilon$, $\varepsilon\neq [d]$, so the unique closing parallelepiped property implies that they are equal. This would conclude that $\textbf{y}\in \KK_{T_{1},\ldots,T_{d}}^{x_{0}}$. 
	To construct the points mentioned above, assume we have constructed the point $\textbf{x}^k\in\KK_{T_{1},\ldots,T_{d}}^{x_{0}} $ for $0\leq k<d$.
	Notice that $(x_{\varepsilon}^{k}\colon \varepsilon \in \{0,1\}^{d}\setminus \{\vec{0}\},  \varepsilon_{k+1}=0)$ and $(y_{\varepsilon}\colon \varepsilon \in \{0,1\}^{d}\setminus\{\vec{0}\}, \varepsilon_{k+1}=0)$ differ at most in the coordinate $[d]\setminus\{k+1\}$.

	\begin{claim} \label{cl:replacement} There exists $a\in X$ such that if we replace the coordinate $\varepsilon=[d]\setminus\{k+1\}$ of $\textbf{x}^{k}$ by the corresponding coordinate of $\textbf{y}$ and the coordinate $[d]$ of $\textbf{x}^{k}$ by $a$,  we obtain a point $\textbf{x}^{k+1}$ that belongs to $\KK_{T_{1},\ldots,T_{d}}^{x_{0}}$.
	\end{claim}
	
	    We suggest the reader to look at Figure \ref{fig:claim1}, while reading the proof of the claim. There, we have illustrated the proof in the case $d=3$. We show how to change the coordinate $[d]\setminus\{1\}$ of $\textbf{x}$ by the corresponding one of $\textbf{y}$ (paying the minor price of modifying the $[d]$ coordinate too).
	
	\begin{proof}[Proof of the \cref{cl:replacement}] We first remark that it is enough to show the statement of the claim for $k=0$ (\emph{i.e.}, for the first step). Indeed, by Lemma \ref{lem:digitpermutation}, after proving the case $k=0$ we can permute the transformations and assume that now the second transformation is the first one, apply again the case $k=0$ and repeat this argument as many times as needed. 
		So, we have to show that from $\textbf{x}^{0} = \textbf{x}$ we can change its $[d]\setminus \{1\}$ coordinate by the corresponding one of $\textbf{y}$ and obtain a point in $\KK_{T_{1},\ldots,T_{d}}^{x_{0}}$. In the process we also allow us to change the $[d]$ coordinate of $\textbf{x}$. 
		
		In this claim it will be convenient to work in $\QQ_{T_1,\ldots,T_d}(X)$ instead of in $\KK_{T_{1},\ldots,T_{d}}^{x_0}$ so  we identify $\textbf{x}, \textbf{y} \in \KK_{T_{1},\ldots,T_{d}}^{x_0}$ with $(x_0,\textbf{x}), (x_0,\textbf{y})\in \QQ_{T_{1},\ldots,T_{d}}(X)$ respectively. This identification is harmless because of the unique closing parallelepiped property and avoids introducing more notation. 
		
		Consider the system $(X,T_{2},\ldots,T_{d})$. Since the points $(x_{\varepsilon}\colon \varepsilon \in \{0,1\}^{d},  \varepsilon_1=0)$ and $(y_{\varepsilon}\colon \varepsilon \in \{0,1\}^{d}, \varepsilon_1=0)$ differ at most in one coordinate (the coordinate $[d]\setminus\{1\}$), by Theorem \ref{prop2} \eqref{xyyy} applied to the system $(X,T_2,\ldots,T_d)$ (see Remark \ref{Rem:gluingsubset}) we have that
		\begin{equation}
		(x_{[d]\setminus\{1\}},\ldots,x_{[d]\setminus\{1\}},y_{[d]\setminus\{1\}})\in \QQ_{T_{2},\ldots,T_{d}}(X).
		\label{eq1}
		\end{equation}
		
		Now, duplicating this point along the $1$-th face, we obtain the point $\textbf{v}^1 \in \QQ_{T_{1},\ldots,T_{d}}(X)$. This point is characterized by 
		$$v_{\varepsilon}^{1}=\left\{\begin{array}{ll}x_{[d]\setminus\{1\}} & \text{if}\ \varepsilon_1=0\ \wedge \exists j\neq 1,\ \varepsilon_{j}=0,\\
		y_{[d]\setminus\{1\}} & \text{if}\ \varepsilon = [d]\setminus\{1\},\\x_{[d]\setminus\{1\}} & \text{if}\ \varepsilon_1=1\ \wedge \exists j\neq 1,\ \varepsilon_{j}=0,\\ y_{[d]\setminus\{1\}} & \text{if}\ \varepsilon=[d].\end{array}\right.$$

\begin{claim}\label{cl:completingcube}
There exists $a\in X$ such that the point $\textbf{u}^1$ defined as 
				
$$u_{\varepsilon}^{1}=\left\{\begin{array}{ll}x_{[d]\setminus\{1\}} & \text{if}\ \varepsilon_1=0\ \wedge \exists j\neq 1,\ \varepsilon_{j}=0,\\
y_{[d]\setminus\{1\}} & \text{if}\ \varepsilon = [d]\setminus\{1\},\\ x_{[d]} & \text{if}\ \varepsilon_1=1\ \wedge \exists j\neq 1,\ \varepsilon_{j}=0,\\ a & \text{if}\ \varepsilon=[d],\end{array}\right.$$
belongs to $\QQ_{T_{1},\ldots,T_{d}}(X)$.
			\end{claim}
			\begin{proof}[Proof the \cref{cl:completingcube}]
				To prove the claim, notice that $(x_{[d]\setminus\{1\}},x_{[d]}) \in\QQ_{T_{1}}(X) $, because the point $\textbf{x}$ belongs to $\QQ_{T_{1},\ldots,T_{d}}(X)$. Now we analyze separately two cases, namely the case $\QQ_{T_{1}}(x_{[d]\setminus \{1\}})=\overline{\mathcal{O}(x_{[d]\setminus\{1\}},T_{1})}$ and the case ${\QQ_{T_{1}}(x_{[d]\setminus \{1\}})\supsetneq\overline{\mathcal{O}(x_{[d]\setminus\{1\}},T_{1})}}$. By Lemma 4.5 of \cite{glasner1994topological}, there is a $G_{\delta}$-dense set of $x\in X$ such that $\QQ_{T_{1}}(x)=\overline{\mathcal{O}(x,T_{1})}$. 
				
				\noindent Case 1: $\QQ_{T_{1}}(x_{[d]\setminus \{1\}})=\overline{\mathcal{O}(x_{[d]\setminus\{1\}},T_{1})}$. Since $(x_{[d]\setminus\{1\}},x_{[d]}) \in\QQ_{T_{1}}(X)$, by the assumption of this case we can find a sequence $(n_{i})_{i\in \N}\subseteq \Z$ such that
				$$T_{1}^{n_{i}}x_{[d]\setminus\{1\}}\to x_{[d]}.$$
				By compactness we can assume that $T_{1}^{n_{i}}y_{[d]\setminus\{1\}}\to a$ for some $a\in X$. Thus, applying the transformation $(T_1^{[d]})^{n_i}$ to the point $\textbf{v}^1$ and taking the limit, we obtain a point $\textbf{u}^{1}$ of the desired form that belongs to $\QQ_{T_{1},\ldots,T_{d}}(X)$.
				
				\noindent Case 2: $\QQ_{T_{1}}(x_{[d]\setminus \{1\}})\supsetneq\overline{\mathcal{O}(x_{[d]\setminus\{1\}},T_{1})}$. We define the following projection maps.
				
				\begin{itemize}
					\item Let $\phi_{1}$ be the projection from $\KK_{T_{1},\ldots,T_{d}}^{x_{0}}$ onto the coordinates $\varepsilon \in \{0,1\}^{d}$ where $\varepsilon_1=0$ and the coordinate $\varepsilon=[d]$. That is
					\[\phi_1(( z_{\varepsilon}: \varepsilon\in\{0,1\}^d)) =(z_{\varepsilon}: \varepsilon\in\{0,1\}^d, \varepsilon_1 =0 \vee \varepsilon =[d]) \]
					
					\item Let $\phi_{2}$ be the projection from $\KK_{T_{1},\ldots,T_{d}}^{x_{0}}$ onto the coordinates $\varepsilon \in \{0,1\}^{d}$ where $\varepsilon_1=0$, and set $\phi_{1,2}$ the map from $\phi_1(\KK_{T_{1},\ldots,T_{d}}^{x_{0}})$ onto $\phi_2(\KK_{T_{1},\ldots,T_{d}}^{x_{0}})$ such that $\phi_{1,2}\circ\phi_1=\phi_2$. That is, 
					\begin{align*} \phi_2(( z_{\varepsilon}: \varepsilon\in\{0,1\}^d)) &=(z_{\varepsilon}: \varepsilon\in\{0,1\}^d, \varepsilon_1 =0), \\ 
					\phi_{1,2}((z_{\varepsilon}: \varepsilon\in\{0,1\}^d, \varepsilon_1 =0 \vee \varepsilon =[d])) &=(z_{\varepsilon}: \varepsilon\in\{0,1\}^d, \varepsilon_1 =0).
					\end{align*}
					
					\item Let $\phi_{3}$ be the projection from $\KK_{T_{1},\ldots,T_{d}}^{x_{0}}$ onto the coordinates $\varepsilon \in \{0,1\}^{d}$ such that $\varepsilon_1=0$ and such that there exists $j\in [d]$, $j\neq 1$, such that $\varepsilon_{j}=0$. Set $\phi_{2,3}$ the map from $\phi_2(\KK_{T_{1},\ldots,T_{d}}^{x_{0}})$ onto $\phi_3(\KK_{T_{1},\ldots,T_{d}}^{x_{0}})$ such that $\phi_{2,3}\circ \phi_2=\phi_3$. 
					That is, 
					\begin{align*} \phi_3(( z_{\varepsilon}: \varepsilon\in\{0,1\}^d))& =(z_{\varepsilon}: \varepsilon\in\{0,1\}^d, \varepsilon_1 =0 \wedge \exists  j\neq 1 , \varepsilon_j=0), \\ 
					\phi_{2,3}((z_{\varepsilon}: \varepsilon\in\{0,1\}^d, \varepsilon_1 =0))& =(z_{\varepsilon}: \varepsilon\in\{0,1\}^d, \varepsilon_1 =0 \wedge \exists  j\neq 1 , \varepsilon_j=0).
					\end{align*}
					
				\end{itemize}
				Note that when we write a expression like  $(z_{\varepsilon}: \varepsilon\in\{0,1\}^d, \varepsilon_1 =0)$ we implicitly assume that this point is the restriction of a point $\textbf{z}=(z_{\varepsilon}: \varepsilon\in\{0,1\}^d)$. 
				By distality, all functions described above are open.  
				Let $\delta>0$. Because of the openness of $\phi_{1,2}$ and $\phi_{2,3}$, we can take $0<\delta'<\delta$ such that \begin{equation}
				B\left(\left(x_{\varepsilon}\colon \varepsilon\in \{0,1\}^{d},\ \varepsilon_1=0\right),\delta'\right)\subseteq \phi_{1,2}\left(B\left(\left(x_{\varepsilon}\colon \varepsilon \in \{0,1\}^{d},\ \varepsilon_1=0 \vee\ \varepsilon=[d]\right)\right),\delta\right),
				\label{eqphi2}
				\end{equation}
				and ${0<\delta''<\delta'}$ such that 
				\begin{equation}
				\begin{array}{l}
				B((x_{\varepsilon}\colon \varepsilon\in \{0,1\}^{d},\ \varepsilon_1=0 \wedge\ \exists  j\neq 1, \varepsilon_{j}=0),\delta'')\\
				\subseteq \phi_{2,3}(B((x_{\varepsilon}\colon \varepsilon\in \{0,1\}^{d},\ \varepsilon_1=0),\delta')).
				\end{array}
				\label{eqphi3}
				\end{equation}

Let $\textbf{a}=(a_{\varepsilon}\colon \varepsilon \in \{0,1\}^{d})$ be such that $\QQ_{T_{1}}(a_{[d]\setminus\{1\}})=\overline{\mathcal{O}(a_{[d]\setminus\{1\}},T_{1})}$ and such that is $\delta''$ close to $(x_{\varepsilon}\colon \varepsilon \in \{0,1\}^{d}).$ This point exists because $x_0$ is a continuity point (and we may take $\textbf{a}\in \KK_{T_{1},\ldots,T_{d}}^{x_{0}} $ in the face orbit of $(x_0,\ldots,x_0)$. See the remarks before the statement of Theorem  \ref{KRelInd})). 
				The point $(a_{\varepsilon}\colon \varepsilon \in \{0,1\}^{d},\varepsilon_1=0 \wedge \exists j\neq 1, \varepsilon_{j}=0)$ is $\delta''$ close to $(x_{\varepsilon}\colon \varepsilon \in \{0,1\}^{d},\varepsilon_1=0 \wedge \exists j\neq 1, \varepsilon_{j}=0)$ and then, by \eqref{eqphi3}, there exists a point $(b_{\varepsilon}\colon \varepsilon \in \{0,1\}^{d}, \varepsilon_1=0)$ such that
				$(b_{\varepsilon}\colon \varepsilon \in \{0,1\}^{d}, \varepsilon_1=0)$ is $\delta'$ close to $(y_{\varepsilon}\colon \varepsilon \in \{0,1\}^{d}, \varepsilon_1=0)$ and $(a_{\varepsilon}\colon \varepsilon \in \{0,1\}^{d}, \varepsilon_1=0 \wedge \exists j\neq 1, \varepsilon_{j}=0)=(b_{\varepsilon}\colon \varepsilon \in \{0,1\}^{d}, \varepsilon_1=0 \wedge \exists j\neq 1, \varepsilon_{j}=0)$.
				Therefore, the points $(a_{\varepsilon}\colon \varepsilon \in \{0,1\}^{d},\varepsilon_1=0)$ and $(b_{\varepsilon}\colon \varepsilon \in \{0,1\}^{d},\varepsilon_1=0) $ differ at most in one coordinate (the coordinate $[d]\setminus\{1\}$) and we can repeat the argument started  in $\eqref{eq1}$. 
				The point $(a_{\varepsilon}\colon \varepsilon \in \{0,1\}^{d}, \varepsilon_1=0)$ satisfies the assumption of Case 1 and thus we can construct $\textbf{w}\in \QQ_{T_{1},\ldots,T_{d}}(X)$ such that
				$$w_{\varepsilon}=\left\{\begin{array}{ll}a_{[d]\setminus\{1\}} & \text{if}\ \varepsilon_1=0\ \wedge \exists j\neq 1,\ \varepsilon_{j}=0,\\
				b_{[d]\setminus\{1\}} & \text{if}\ \varepsilon=[d]\setminus\{1\},\\ a_{[d]} & \text{if}\ \varepsilon_1=1\ \wedge \exists j\neq 1,\ \varepsilon_{j}=0,\\ u_{a} & \text{if}\ \varepsilon=[d],\end{array}\right.$$
				
				for some $u_{a}\in X$.
				
				Letting $\delta \to 0$, by compactness we can assume that $u_{a}$ converges to some $a\in X$. Thus we have that there exists ${\textbf{u}^{1}\in \QQ_{T_{1},\ldots,T_{d}}(X)}$ such that
				$$u_{\varepsilon}^{1}=\left\{\begin{array}{ll}x_{[d]\setminus\{1\}} & \text{if}\ \varepsilon_1=0\ \wedge \exists j\neq 1,\ \varepsilon_{j}=0,\\
				y_{[d]\setminus\{1\}} & \text{if}\ \varepsilon= [d]\setminus\{1\},\\ x_{[d]} & \text{if}\ \varepsilon_1=1\ \wedge \exists j\neq 1,\ \varepsilon_{j}=0,\\ a & \text{if}\ \varepsilon=[d].\end{array}\right.$$
				
This finishes the proof of \cref{cl:completingcube} in Case 2. 
\end{proof}
		
From $\textbf{u}^1$ we aim to construct points $\textbf{u}^2,\ldots,\textbf{u}^d$ in $\QQ_{T_{1},\ldots,T_{d}}(X)$ where at each step they increasingly coincide with  $\textbf{x}$, but having their $[d]\setminus\{1\}$ coordinate equal to $y_{[d]\setminus\{1\}}$ and their $[d]$ coordinate equal to $a$. We suggest the reader to keep in mind that the $[d]\setminus\{1\}$ and $[d]$ coordinates are special.
		
Consider the face $(u_{\varepsilon}^{1}\colon \varepsilon_2=0)$. Since $\textbf{x}\in \QQ_{T_{1},\ldots,T_{d}}(X)$, duplicating the restriction of $\textbf{x}$ to the coordinates 
		$ [d]\setminus\{1,2\},~[d]\setminus\{1\},~[d]\setminus\{2\}$ and $[d]$ we get the point $\textbf{v}^{2}\in \QQ_{T_{1},\ldots,T_{d}}(X)$ characterized by 
		$$v_{\varepsilon}^{2}=\left\{\begin{array}{ll}x_{[d]\setminus\{1,2\}} & \text{if}\ \varepsilon_1=0\ \wedge\ \varepsilon_2=0,\\ x_{[d]\setminus\{1\}} & \text{if}\ \varepsilon_1=0\ \wedge\ \varepsilon_2=1,\\ x_{[d]\setminus\{2\}} & \text{if}\ \varepsilon_1=1\ \wedge\ \varepsilon_2=0,\\ x_{[d]} & \text{if}\ \varepsilon_1=1\ \wedge\ \varepsilon_2=1.\\  \end{array}\right.$$
		
		Hence, $(v_{\varepsilon}^{2}\colon \varepsilon\in \{0,1\}^{d}, \varepsilon_2=1)=(u_{\varepsilon}^{1}\colon \varepsilon\in \{0,1\}^{d}, \varepsilon_2=0)$. By Lemma \ref{lem:insert}, we can glue the $2$-th lower face of $\textbf{v}$ with the $2$-th upper face of $\textbf{u}^1$ and obtain the point $\textbf{u}^{2}$ that belongs to  $\QQ_{T_{1},\ldots,T_{d}}(X)$. The point  $\textbf{u}^{2}$ is characterized by
$$u_{\varepsilon}^{2}=\left\{\begin{array}{ll}x_{[d]\setminus\{1,2\}} & \text{if}\ \varepsilon_1=0\ \wedge\ \varepsilon_2=0,\\ x_{[d]\setminus\{1\}} & \text{if}\ \varepsilon_1=0\ \wedge\ \varepsilon_2=1, \ \varepsilon \neq [d]\setminus\{1\} , \\ 
x_{[d]\setminus\{2\}} & \text{if}\ \varepsilon_1=1\ \wedge\ \varepsilon_2=0,\\ x_{[d]} & \text{if}\ \varepsilon_1=1\ \wedge\ \varepsilon_2=1,\\ y_{[d]\setminus\{1\}} & \varepsilon=[d]\setminus\{1\}, \\ a & \varepsilon=[d].  \end{array}\right.$$
		
		The value of $u_{\varepsilon}^{2}$ for $\varepsilon\neq [d]$ and $\varepsilon\neq [d]\setminus\{1\}$ can only take the values $x_{[d]\setminus \{1,2\}}$, $x_{[d]\setminus \{1\}}$, $x_{[d]\setminus \{2\}}$ or $x_{[d]}$. Furthermore, $\textbf{u}^{2}$ coincides with $\textbf{x}$ in the coordinates $[d]\setminus\{1,2\}$ and $[d]\setminus\{2\}$ while $u^2_{[d]\setminus\{1\}}=y_{ [d]\setminus\{1\}}$ and $u^2_{[d]}=a$. Note that $\textbf{u}^2$ does not depend on $\{3,\ldots,d \}$, meaning that $u^{2}_{\varepsilon}=x_{\varepsilon \cup \{3,\ldots, d\}}$, for any $\varepsilon$, except for $\varepsilon=[d]\setminus\{1\}$ and $\varepsilon=[d]$.

	 Now assume that we have constructed $\textbf{u}^{k}\in \QQ_{T_{1},\ldots,T_{d}}(X)$ for $2\leq k<d$ such that $\textbf{u}$ coincides with $\textbf{x}$ in all coordinates of the form $[d]\setminus A$ where $A\subseteq \{1,\ldots,k\}$, except for $A=\{1\}$ and $A=\emptyset$, where ${u}^{k}_{[d]\setminus \{1\}}=y_{ [d]\setminus\{1\}}$ and ${u}^{k}_{[d]}=a$. Assume also that $\textbf{u}^k$ does not depend on $\{k+1,\ldots,d \}$, meaning that for any $\varepsilon$ (except for $\varepsilon=[d]\setminus\{1\}$ and $\varepsilon=[d]$)
			\begin{equation} u^{k}_{\varepsilon}=x_{\varepsilon \cup \{k+1,\ldots, d\}}.  
			\label{equation:u^k}
			\end{equation}

			Consider the $(k+1)$-th lower face of $\textbf{u}^k$ (\emph{i.e.}, ${(u_{\varepsilon}^{k}\colon \varepsilon\in \{0,1\}^{d},\ \varepsilon_{k+1}=0)}$). Since ${\textbf{x}\in \QQ_{T_{1},\ldots,T_{d}}(X)}$, duplicating the restriction of $\textbf{x}$ to the coordinates $[d]\setminus A$, $A\subseteq \{1,\ldots,k+1 \}$  we get the point $\textbf{v}^{k+1}$ that belongs to  $\QQ_{T_{1},\ldots,T_{d}}(X)$. This point is characterized by
			\[v^{k+1}_{\varepsilon}=x_{\varepsilon \cup \{k+2,\dots,d\} } \text{ for every }\varepsilon \in \{0,1\}^d. \] 
Note that if  $\varepsilon_{k+1}=1$ (or equivalently $k+1\in \varepsilon)$ then,
\begin{equation}v^{k+1}_{\varepsilon}=x_{\varepsilon \cup \{k+2,\dots,d\} } =x_{\left(\varepsilon \setminus\{ k+1\} \right)\cup \{k+1,k+2,\dots,d\}} =u^k_{\varepsilon\setminus\{k+1\}}  \label{equation:v^{k+1}}, \end{equation}
where in the last equality we used \eqref{equation:u^k} with $\varepsilon \setminus\{ k+1\}$, which is different from $[d]\setminus\{1\}$ and $[d]$. 
			
From \eqref{equation:v^{k+1}} we deduce that  ${(v_{\varepsilon}^{k+1}\colon \varepsilon\in \{0,1\}^{d},\  \varepsilon_{k+1}=1)}={(u_{\varepsilon}^{k}\colon \varepsilon\in \{0,1\}^{d},\ \varepsilon_{k+1}=0)}$ . By \cref{pegado1}, we can glue the $(k+1)$-th lower face of $\textbf{v}^{k+1}$ with the $(k+1)$-th upper face of $\textbf{u}^k$ and obtain the point $\textbf{u}^{k+1}$ that belongs to $\QQ_{T_{1},\ldots,T_{d}}(X)$. The point $\textbf{u}^{k+1}$ is characterized by
			$$
			u_{\varepsilon}^{k+1} =  \left\{\begin{array}{ll}
			v_{\varepsilon}^{k+1} & \text{if}\ \varepsilon_{k+1}=0,\\
			u_{\varepsilon}^{k} & \text{if}\ \varepsilon_{k+1}=1.
			\end{array}\right.$$
			
			We check the point $\textbf{u}^{k+1}$ satisfies the conditions needed to continue the process, assuming  obviously that $k+1<d$. Let $A\subseteq \{1,\ldots,k+1\}$ with $A\neq \{1\}$ and $A\neq \emptyset$. If $k+1 \notin A$ then $k+1\in [d]\setminus A$ and $u^{k+1}_{[d]\setminus A}=u_{[d]\setminus A}^{k} =x_{([d]\setminus A)\cup \{k+1,\ldots,d\} }=x_{[d]\setminus A}$. On the other hand, if $k+1 \in A$ then $k+1\notin [d]\setminus A$ and  $u^{k+1}_{[d]\setminus A}=v_{[d]\setminus A}^{k+1}=x_{([d]\setminus A) \cup \{k+2,\dots,d\} } =x_{[d]\setminus A }$. We deduce that $\textbf{u}$ coincides with $\textbf{x}$ in all coordinates $[d]\setminus A$, with $A\neq \{1\}$ and $A\neq \emptyset$. It is also immediate to check that  $u^{k+1}_{[d]\setminus\{1\}}=u^{k}_{[d]\setminus\{1\}}=y_{[d]\setminus\{1\}}$ and $u^{k+1}_{[d]}=u^{k}_{[d]}=a$.
			We now check that $\textbf{u}^{k+1}$ does not depend on the coordinates $\{k+2,\ldots,d\}$. For $\varepsilon \neq [d]\setminus \{1\}$ and $\varepsilon\neq[d]$ we have 
			\[ u_{\varepsilon}^{k+1} =  \left\{\begin{array}{ll}
			v_{\varepsilon}^{k+1}=x_{\varepsilon\cup \{k+2,\ldots,d\}}  & \text{if}\ \varepsilon_{k+1}=0,\\
			u_{\varepsilon}^{k}=x_{\varepsilon\cup \{k+1,\ldots,d\}}=x_{\varepsilon\cup \{k+2,\ldots,d\}}    & \text{if}\ \varepsilon_{k+1}=1,
			\end{array}\right.\]
			and we conclude that $\textbf{u}^{k+1}$ does not depend on the coordinates $\{k+2,\ldots,d\}$.

			We proceed with the process until we construct $\textbf{u}^d\in \QQ_{T_{1},\ldots,T_{d}}(X)$. This point coincides with $\textbf{x}$ in all coordinates $[d]\setminus A$ for $A\subseteq \{1,\ldots,d \}=[d]$ except for $A=\{1\}$ and $A=\emptyset$, where $\textbf{u}^d_{[d]\setminus\{1\}}=y_{[d]\setminus\{1\}}$ and $\textbf{u}^d_{[d]}=a$. We therefore conclude that $\textbf{u}^d$ coincides with $\textbf{x}$ everywhere but in the coordinates $[d]\setminus\{1\}$ and $[d]$. In particular $\textbf{u}_{\vec{0}}=x_0$. Using that $x_0$ is a continuity point we get that $(u_{\varepsilon}:\varepsilon\neq \vec{0})$ belongs to $\KK_{T_{1},\ldots,T_{d}}^{x_0}$, completing the proof of \cref{cl:replacement}.  
			
\end{proof}
	
To finish the proof of \cref{KRelInd}, use \cref{cl:replacement} for $0\leq k <d$ to construct the points $\textbf{x}^{1},\ldots,\textbf{x}^{d}\in \KK_{T_{1},\ldots,T_{d}}^{x_{0}}$. As mentioned before, the point $\textbf{x}^d$ coincides everywhere with $\textbf{y}$, except for the $[d]$ coordinate. The unique closing parallelepiped property implies that $\textbf{y}=\textbf{x}^{d}\in \KK_{T_{1},\ldots,T_{d}}^{x_{0}}$. 
	
\end{proof}

{\footnotesize
	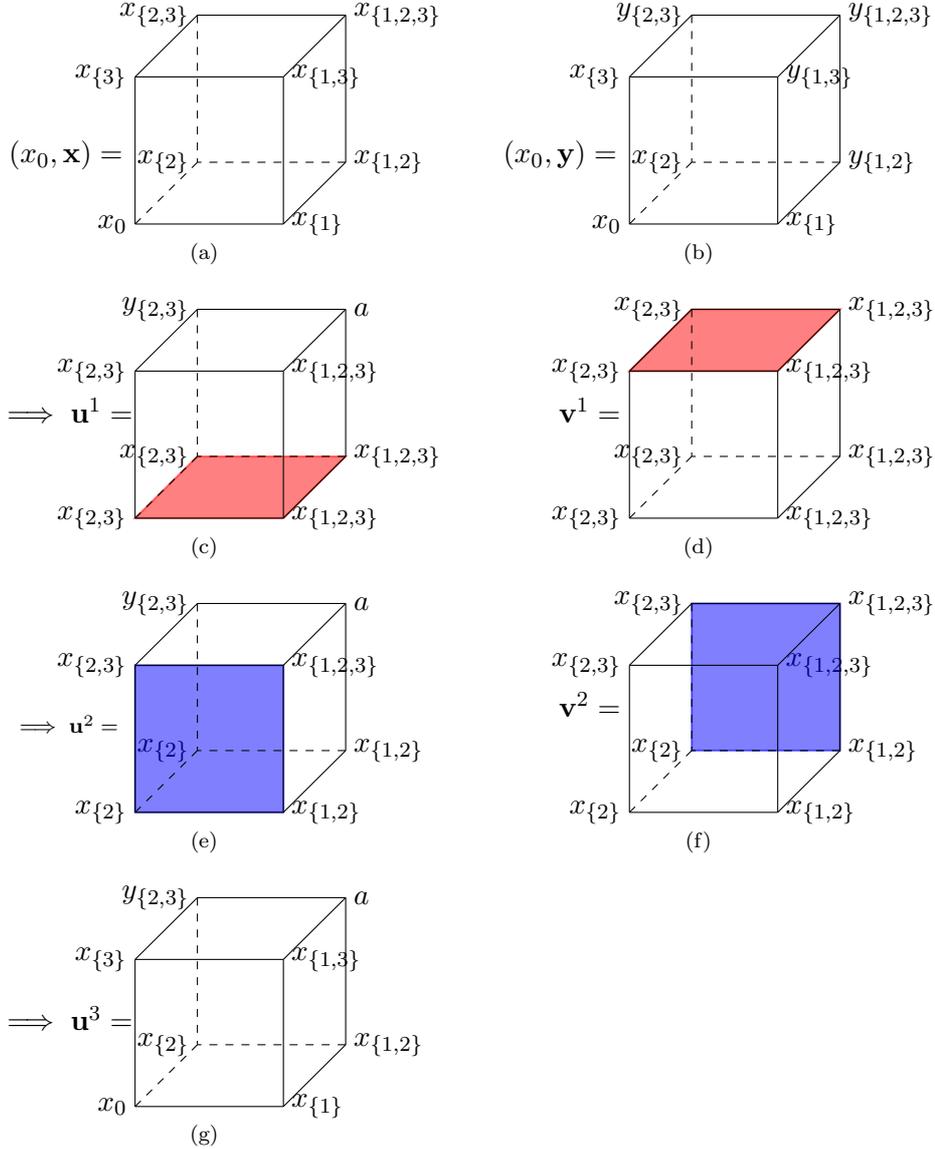
\begin{figure}[h]
		\centering
		\begin{tikzpicture}[scale=1.3]
		\node(s1) at (-0.7,0.7) [scale=1] {\large$(x_{0},\textbf{x})=$};
		\coordinate[scale=.5,label=left:\large$x_{0}$] (a1) at (0,0);
		\coordinate[scale=.5,label=left:\large$x_{\{3\}}$] (a2) at (0,1.5);
		\coordinate[scale=.5,label=right:\large$x_{\{1\}}$] (a3) at (1.5,0);
		\coordinate[scale=.5,label=right:\large$x_{\{1,3\}}$] (a4) at (1.5,1.5);
		\coordinate[scale=.5,label=right:\large$x_{\{1,2,3\}}$] (a5) at (2.13,2.13);
		\coordinate[scale=.5,label=left:\large$x_{\{2,3\}}$] (a6) at (0.63,2.13);
		\coordinate[scale=.5,label=right:\large$x_{\{1,2\}}$] (a7) at (2.13,0.63);
		\coordinate[scale=.5,label=left:\large$x_{\{2\}}$] (a8) at (0.63,0.63);
		\node(n1) at (0.7,-0.3) [scale=1] {(a)};
		
		\path[thin] (a1) edge (a2);
		\path[thin] (a1) edge (a3);
		\path[dashed] (a1) edge (a8);
		\path[thin] (a2) edge (a4);
		\path[thin] (a2) edge (a6);
		\path[thin] (a3) edge (a4);
		\path[thin] (a3) edge (a7);
		\path[thin] (a4) edge (a5);
		\path[thin] (a5) edge (a7);
		\path[thin] (a5) edge (a6);
		\path[dashed] (a6) edge (a8);
		\path[dashed] (a7) edge (a8);
		
		\node(s1) at (4.3,0.7) [scale=1] {\large$(x_{0},\textbf{y})=$};
		
		\coordinate[scale=.5,label=left:\large$x_{0}$] (c1) at (5,0);
		\coordinate[scale=.5,label=left:\large$x_{\{3\}}$] (c2) at (5,1.5);
		\coordinate[scale=.5,label=right:\large$x_{\{1\}}$] (c3) at (6.5,0);
		\coordinate[scale=.5,label=right:\large$y_{\{1,3\}}$] (c4) at (6.5,1.5);
		\coordinate[scale=.5,label=right:\large$y_{\{1,2,3\}}$] (c5) at (7.13,2.13);
		\coordinate[scale=.5,label=left:\large$y_{\{2,3\}}$] (c6) at (5.63,2.13);
		\coordinate[scale=.5,label=right:\large$y_{\{1,2\}}$] (c7) at (7.13,0.63);
		\coordinate[scale=.5,label=left:\large$x_{\{2\}}$] (c8) at (5.63,0.63);
		
		\node(n2) at (5.7,-0.3) [scale=1] {(b)};
		
		\path[thin] (c1) edge (c2);
		\path[thin] (c1) edge (c3);
		\path[dashed] (c1) edge (c8);
		\path[thin] (c2) edge (c4);
		\path[thin] (c2) edge (c6);
		\path[thin] (c3) edge (c4);
		\path[thin] (c3) edge (c7);
		\path[thin] (c4) edge (c5);
		\path[thin] (c5) edge (c7);
		\path[thin] (c5) edge (c6);
		\path[dashed] (c6) edge (c8);
		\path[dashed] (c7) edge (c8);

		\node(s6) at (4.6,-1.9) [scale=1] {\large$\textbf{v}^{1}=$};
		\coordinate[scale=.5,label=left:\large$x_{\{2,3\}}$] (d1) at (5,-3);
		\coordinate[scale=.5,label=left:\large$x_{\{2,3\}}$] (d2) at (5,-1.5);
		\coordinate[scale=.5,label=right:\large$x_{\{1,2,3\}}$] (d3) at (6.5,-3);
		\coordinate[scale=.5,label=right:\large$x_{\{1,2,3\}}$] (d4) at (6.5,-1.5);
		\coordinate[scale=.5,label=right:\large$x_{\{1,2,3\}}$] (d5) at (7.13,-0.87);
		\coordinate[scale=.5,label=left:\large$x_{\{2,3\}}$] (d6) at (5.63,-0.87);
		\coordinate[scale=.5,label=right:\large$x_{\{1,2,3\}}$] (d7) at (7.13,-2.37);
		\coordinate[scale=.5,label=left:\large$x_{\{2,3\}}$] (d8) at (5.63,-2.37);

		\filldraw [red,opacity=0.5,thick] (d5)--(d6)--(d2)--(d4)--(d5)--cycle;
		\node(n4) at (5.7,-3.3) [scale=1] {(d)};
		
		\path[thin] (d1) edge (d2);
		\path[thin] (d1) edge (d3);
		\path[dashed] (d1) edge (d8);
		\path[thin] (d2) edge (d4);
		\path[thin] (d2) edge (d6);
		\path[thin] (d3) edge (d4);
		\path[thin] (d3) edge (d7);
		\path[thin] (d4) edge (d5);
		\path[thin] (d5) edge (d7);
		\path[thin] (d5) edge (d6);
		\path[dashed] (d6) edge (d8);
		\path[dashed] (d7) edge (d8);
		
		\coordinate[scale=.5,label=left:\large$x_{\{2,3\}}$] (e1) at (0,-3);
		\coordinate[scale=.5,label=left:\large$x_{\{2,3\}}$] (e2) at (0,-1.5);
		\coordinate[scale=.5,label=right:\large$x_{\{1,2,3\}}$] (e3) at (1.5,-3);
		\coordinate[scale=.5,label=right:\large$x_{\{1,2,3\}}$] (e4) at (1.5,-1.5);
		\coordinate[scale=.5,label=right:\large$a$] (e5) at (2.13,-0.87);
		\coordinate[scale=.5,label=left:\large$y_{\{2,3\}}$] (e6) at (0.63,-0.87);
		\coordinate[scale=.5,label=right:\large$x_{\{1,2,3\}}$] (e7) at (2.13,-2.37);
		\coordinate[scale=.5,label=left:\large$x_{\{2,3\}}$] (e8) at (0.63,-2.37);
		
		\node(n5) at (0.7,-3.3) [scale=1] {(c)};
		
		\filldraw [red,opacity=0.5,thick] (e1)--(e3)--(e7)--(e8)--(e1)--cycle;
		
		\node(s9) at (-0.7,-1.9) [scale=1] {\large$\implies \textbf{u}^{1}=$};
		
		\node(s10) at (4.6,-4.9) [scale=1] {\large$\textbf{v}^{2}=$};
		
		\coordinate[scale=.5,label=left:\large$x_{\{2\}}$] (f1) at (5,-6);
		\coordinate[scale=.5,label=left:\large$x_{\{2,3\}}$] (f2) at (5,-4.5);
		\coordinate[scale=.5,label=right:\large$x_{\{1,2\}}$] (f3) at (6.5,-6);
		\coordinate[scale=.5,label=right:\large$x_{\{1,2,3\}}$] (f4) at (6.5,-4.5);
		\coordinate[scale=.5,label=right:\large$x_{\{1,2,3\}}$] (f5) at (7.13,-3.87);
		\coordinate[scale=.5,label=left:\large$x_{\{2,3\}}$] (f6) at (5.63,-3.87);
		\coordinate[scale=.5,label=right:\large$x_{\{1,2\}}$] (f7) at (7.13,-5.37);
		\coordinate[scale=.5,label=left:\large$x_{\{2\}}$] (f8) at (5.63,-5.37);
		
		\node(n6) at (5.7,-6.3) [scale=1] {(f)};
		
		\filldraw [blue,opacity=0.5,thick] (f6)--(f5)--(f7)--(f8)--(f6)--cycle;
		\path[thin] (e1) edge (e2);
		\path[thin] (e1) edge (e3);
		\path[dashed] (e1) edge (e8);
		\path[thin] (e2) edge (e4);
		\path[thin] (e2) edge (e6);
		\path[thin] (e3) edge (e4);
		\path[thin] (e3) edge (e7);
		\path[thin] (e4) edge (e5);
		\path[thin] (e5) edge (e7);
		\path[thin] (e5) edge (e6);
		\path[dashed] (e6) edge (e8);
		\path[dashed] (e7) edge (e8);
		
		\path[thin] (f1) edge (f2);
		\path[thin] (f1) edge (f3);
		\path[dashed] (f1) edge (f8);
		\path[thin] (f2) edge (f4);
		\path[thin] (f2) edge (f6);
		\path[thin] (f3) edge (f4);
		\path[thin] (f3) edge (f7);
		\path[thin] (f4) edge (f5);
		\path[thin] (f5) edge (f7);
		\path[thin] (f5) edge (f6);
		\path[dashed] (f6) edge (f8);
		\path[dashed] (f7) edge (f8);

		\node(s12) at (-0.7,-5.1) [scale=1] {$\implies \textbf{u}^{2}=$};
		
		\coordinate[scale=.5,label=left:\large$x_{\{2\}}$] (g1) at (0,-6);
		\coordinate[scale=.5,label=left:\large$x_{\{2,3\}}$] (g2) at (0,-4.5);
		\coordinate[scale=.5,label=right:\large$x_{\{1,2\}}$] (g3) at (1.5,-6);
		\coordinate[scale=.5,label=right:\large$x_{\{1,2,3\}}$] (g4) at (1.5,-4.5);
		\coordinate[scale=.5,label=right:\large$a$] (g5) at (2.13,-3.87);
		\coordinate[scale=.5,label=left:\large$y_{\{2,3\}}$] (g6) at (0.63,-3.87);
		\coordinate[scale=.5,label=right:\large$x_{\{1,2\}}$] (g7) at (2.13,-5.37);
		\coordinate[scale=.5,label=left:\large$x_{\{2\}}$] (g8) at (0.63,-5.37);
		
		\filldraw [blue,opacity=0.5,thick] (g1)--(g3)--(g4)--(g2)--(g1)--cycle;
		
		\node(s15) at (-0.7,-8.1) [scale=1] {\large$\implies \textbf{u}^{3}=$};
		
		\coordinate[scale=.5,label=left:\large$x_{0}$] (h1) at (0,-9);
		\coordinate[scale=.5,label=left:\large$x_{\{3\}}$] (h2) at (0,-7.5);
		\coordinate[scale=.5,label=right:\large$x_{\{1\}}$] (h3) at (1.5,-9);
		\coordinate[scale=.5,label=right:\large$x_{\{1,3\}}$] (h4) at (1.5,-7.5);
		\coordinate[scale=.5,label=right:\large$a$] (h5) at (2.13,-6.87);
		\coordinate[scale=.5,label=left:\large$y_{\{2,3\}}$] (h6) at (0.63,-6.87);
		\coordinate[scale=.5,label=right:\large$x_{\{1,2\}}$] (h7) at (2.13,-8.37);
		\coordinate[scale=.5,label=left:\large$x_{\{2\}}$] (h8) at (0.63,-8.37);
		
		\path[thin] (g1) edge (g2);
		\path[thin] (g1) edge (g3);
		\path[dashed] (g1) edge (g8);
		\path[thin] (g2) edge (g4);
		\path[thin] (g2) edge (g6);
		\path[thin] (g3) edge (g4);
		\path[thin] (g3) edge (g7);
		\path[thin] (g4) edge (g5);
		\path[thin] (g5) edge (g7);
		\path[thin] (g5) edge (g6);
		\path[dashed] (g6) edge (g8);
		\path[dashed] (g7) edge (g8);
		
		\path[thin] (h1) edge (h2);
		\path[thin] (h1) edge (h3);
		\path[dashed] (h1) edge (h8);
		\path[thin] (h2) edge (h4);
		\path[thin] (h2) edge (h6);
		\path[thin] (h3) edge (h4);
		\path[thin] (h3) edge (h7);
		\path[thin] (h4) edge (h5);
		\path[thin] (h5) edge (h7);
		\path[thin] (h5) edge (h6);
		\path[dashed] (h6) edge (h8);
		\path[dashed] (h7) edge (h8);
		
		\node(n7) at (0.7,-6.3) [scale=1] {(e)};
		\node(n8) at (0.7,-9.3) [scale=1] {(g)};
		
		\end{tikzpicture}
		\caption{ Illustration of the proof of \cref{cl:replacement} in Theorem \ref{KRelInd} for the case $d=3$. We change the $[d]\setminus\{1\}=\{2,3\}$ coordinate of $\textbf{x}$ by the corresponding one of $\textbf{y}$.}
		\label{fig:claim1}
	\end{figure}
}
	
\section{Proof of Theorem \ref{StructThm}}  \label{sec:ProofThm}
	
We have now all the ingredientes to prove the structure Theorem \ref{StructThm}. We restate it here for the reader's convenience. Interestingly, after the work of previous sections, all the $\Z^d$-systems that appear in the theorem can be described explicitly from the directional cube structures of $(X,T_{1},\ldots,T_{d})$.

\begin{theorem*}
Let $(X,T_{1},\ldots,T_{d})$ be a minimal distal $\Z^{d}$-system. The following statements are equivalent:
\begin{enumerate}[(1)]
\item $(X,T_{1},\ldots,T_{d})$ has the unique closing parallelepiped property, \emph{i.e.}, if $\textbf{x},\textbf{y}\in \QQ_{T_{1},\ldots,T_{d}}(X)$ have $2^{d}-1$ coordinates in common, then $\textbf{x}=\textbf{y}$.
\item $\mathcal{R}_{T_{1},\ldots,T_{d}}(X)=\Delta_{X}$.
\item The structure of $(X,T_{1},\ldots,T_{d})$ can be described as follows: (i) it is a factor of a minimal distal $\Z^{d}$-system $(Y,T_{1},\ldots,T_{d})$ which is a joining of $\Z^{d}$-systems $(Y_{1},T_{1},\ldots,T_{d}),\ldots,(Y_{d},T_{1},\ldots,T_{d})$, where for each $i\in \{1,\ldots,d\}$ the action of $T_i$ on $Y_i$ is the identity; (ii) for each $i,j\in \{1,\ldots,d\}$, $i< j$, there exists a $\Z^{d}$-system
$(Y_{i,j},T_{1},\ldots,T_{d})$ which is a common factor of $(Y_{i},T_{1},\ldots,T_{d})$ and $(Y_{j},T_{1},\ldots,T_{d})$ and where $T_i$ and $T_j$ act as the identity; and (iii) $Y$ is jointly relatively independent with respect to the systems $\left ( (Y_{i,j},T_{1},\ldots,T_{d}): \ i,j \in \{1,\ldots,d \}, \ i<j \right )$.

More precisely, $Y=\KK_{T_{1},\ldots,T_{d}}^{x_{0}}$ and $Y_j=\KK_{T_{1},\ldots,T_{j-1},T_{j+1},\ldots,T_{d}}^{x_{0}}$ with $x_0$ being a continuity point. Relative independence of $Y$ is with respect to their maximal $\mathsf{Z}_{0}^{e_{j_{1}}}\wedge \mathsf{Z}_{0}^{e_{j_{2}}}$-factors, for all $j_{1},j_{2}\in [d]$ with $j_{1} < j_{2}$.
\end{enumerate}
\end{theorem*}

\bigskip

\begin{proof}[Proof of Theorem \ref{StructThm}]
$(1) \implies (2)$. This follows from Proposition \ref{prop4}. 
		
$(2) \implies (1)$. Suppose that $(X,T_{1},\ldots,T_{d})$ does not verify the unique closing parallelepiped property, then there exist $x,y\in X$ with $x\neq y$ and $\textbf{a}_{*}\in X_{*}^{[d]}$ such that $(x,\textbf{a}_{*}),\ {(y,\textbf{a}_{*})\in \QQ_{T_{1},\ldots,T_{d}}(X)}$. By Proposition \ref{prop2}, we have that $(x,y)\in \mathcal{R}_{T_{1},\ldots,T_{d}}(X)$. Then, $x=y$, which is a contradiction.
		
$(1) \implies (3)$. This is a consequence of Theorem \ref{KRelInd}. 
		
$(3) \implies (1)$. We show that the system $(Y,T_{1},\ldots,T_{d})$ given by (3) verifies the unique closing parallelepiped property. 

Let $\textbf{y}\in \QQ_{T_{1},\ldots,T_{j}}(Y)$. 
By definition, for every $j\in [d]$ we have that $(y_{[d]},y_{[d]\setminus\{j\}}) \in 
\QQ_{T_j}(Y)$. But, $T_j$ acts as the identity on the $j$-th coordinate of points in $Y$, then
$((y_{[d]})_j,(y_{[d]\setminus\{j\}})_j) \in 
\QQ_{T_j}(Y_j)=\Delta_{Y_j}$. This implies that,  
$(y_{[d]})_j = (y_{[d]\setminus\{j\}})_j$. Hence, $y_{[d]}$ can be determined from previous coordinates of $\textbf{y}$, which proves the unique closing parallelepiped property for $(Y,T_1,\ldots,T_d)$. 

Since (1) is equivalent with (2) we have that $\rr_{T_{1},\ldots,T_{d}}(Y)=\Delta_{Y}$. By Theorem \ref{teo2}, $\rr_{T_{1},\ldots,T_{d}}(X)=\Delta_{X}$. This proves the unique closing parallelepiped property for $(X,T_1,\ldots,T_d)$. 	  
\end{proof}
	
The following two corollaries can be deduced from Theorem \ref{StructThm} and Theorem \ref{teo2}. The first one was implicitly proved inside the last proof. 
	
\begin{corollary}
Let $\pi\colon Y \to X$ be a factor map between  minimal distal $\Z^d$-systems $(Y,T_{1},\ldots,T_{d})$ and $(X,T_{1},\ldots,T_{d})$. If $(Y,T_{1},\ldots,T_{d})$ has the unique closing parallelepiped property then 
$(X,T_{1},\ldots,T_{d})$ has it too. 
\label{cor1}
\end{corollary}
	
\begin{corollary}
Let $(X,T_{1},\ldots,T_{d})$ be a minimal distal $\Z^d$-system. 
Then, $({X}/{\mathcal{R}_{T_{1},\ldots,T_{d}}(X)},T_{1},\ldots,T_{d})$ has the unique closing parallelepiped property. Moreover, this system is the maximal factor with this property, \emph{i.e.}, any other factor of $(X,T_{1},\ldots,T_{d})$ with the unique closing parallelepiped property factorizes through it.
\label{prop5}
\end{corollary}
	
\begin{proof}
Observe that if $(Z,T_{1},\ldots,T_{d})$ is a factor of $(X,T_{1},\ldots,T_{d})$ with the unique 
closing parallelepiped property, then by Theorem \ref{StructThm} ${\mathcal{R}_{T_{1},\ldots,T_{d}}(Z)=\Delta_{Z}}$. Now,  by Theorem \ref{teo2}, $\pi\times \pi(\mathcal{R}_{T_{1},\ldots,T_{d}}(X))=\rr_{T_{1},\ldots,T_{d}}(Z)$ $=\Delta_{Z}$. That is, there exists a factor map from 
$({X}/{\rr_{T_{1},\ldots,T_{d}}(X)},T_{1},\ldots,T_{d})$ to $(Z,T_{1},\ldots,T_{d})$. It remains to prove that $\rr_{T_{1},\ldots,T_{d}} (X/\rr_{T_{1},\ldots,T_{d}}(X))=\Delta_{{X}/{\rr_{T_{1},\ldots,T_{d}}(X)}}$ (\emph{i.e.}, the quotient system has the unique closing parallelepiped property). Let $\pi\colon X\to {X}/{\rr_{T_{1},\ldots,T_{d}}(X)}$ be the quotient map and take $(y_{1},y_{2})\in $  $\rr_{T_{1},\ldots,T_{d}}({X}/{\rr_{T_{1},\ldots,T_{d}}(X)})$. By Theorem \ref{teo2}, there exists $(x_{1},x_{2})\in \rr_{T_{1},\ldots,T_{d}}(X)$ with $\pi(x_{1})=y_{1}$ and $\pi(x_{2})$ $=y_{2}$. But $y_{1}=\pi(x_{1})=\pi(x_{2})=y_{2}$, so $\rr_{T_{1},\ldots,T_{d}}({X}/{\rr_{T_{1},\ldots,T_{d}}(X)})$ coincides with the diagonal relation of ${X}/{\rr_{T_{1},\ldots,T_{d}}(X)}$, which proves the corollary.
\end{proof}

\section{Recurrence in minimal distal $\Z^d$-systems with the unique closing parallelepiped property}  \label{sec:ReturnTime}

As an application of previous work, in this section we study sets of return times for minimal distal $\Z^d$-systems with the unique closing parallelepiped property. In particular, we get a characterization of minimal distal systems with this property using return time ideas.
	
\begin{definition}
Let $(X,T_{1},\ldots,T_{d})$ be a $\Z^d$-system.  Let $x\in X$ and $U$ be an open neighborhood of $x$. The \emph{set of return times} of $x$ to $U$ is defined as \[N_{T_{1},\ldots,T_{d}}(x,U)= \{(n_{1},\ldots,n_{d})\in \Z^{d}\colon\ T_{1}^{n_{1}}\cdots T_{d}^{n_{d}}x \in U\}.\]
		
A subset $B$ of $\Z^{d}$ is a set of return times for a $\Z^d$-system if there exists a $\Z^d$-system 
$(X,T_{1},\ldots,T_{d})$, $x\in X$ and an open neighborhood $U$ of $x$ such that $N_{T_{1},\ldots,T_{d}}(x,U)\subseteq B$.
\end{definition}
	
We are able to characterize sets of return times for minimal distal $\Z^{d}$-systems via the unique closing parallelepiped property. For this we consider the following definition.
	
\begin{definition}
Let $d\geq 2$ be an integer and  $B_{1},\ldots,B_{d}\subseteq \Z^{d-1}$. We define the $d$-joining of $B_{1},\ldots,B_{d}$ as the set 
$$B=\{(n_{1},\ldots,n_{d})\in \Z^{d}\colon \forall i \in [d],\ (n_{1},\ldots, n_{i-1},n_{i+1},\ldots,n_{d})\in B_{i}\}\subseteq \Z^{d}.$$
\end{definition}
We remark that the $2$-joining of $B_{1},B_{2}\subseteq \Z$ is the Cartesian product $B_{1}\times B_{2}$.  For $d\geq 3$ the set $B$ might be empty for general sets. For instance, let $B_1=\{(n_1,n_2) \in \Z^2: n_1-n_2 \text{ is even}\}$, $B_2=\{(n_1,n_2) \in \Z^2: n_1-n_2 \text{ is odd }\}$, $B_3=B_1$. Then, $(n_1,n_2,n_3)\in B$ if and only if $n_2-n_3$ is even, $n_1-n_3$ is odd and $n_1-n_2$ is even, which leads to a contradiction. 
However, if the building sets $B_1,\ldots,B_d$ are return times, then the set $B$ is always non-empty (it contains the vector of 0's). With the help of this notion of joining of sets we can obtain a nice relation with the unique closing parallelepiped property in the minimal distal case. 
	
\begin{theorem} \label{thm:returntimes}
Let $d\geq 2$ be an integer. A subset $B\subseteq \Z^{d}$ contains a set of return times for a minimal distal 
$\Z^d$-system with the unique closing parallelepiped property if and only if $B$ contains a $d$-joining of sets that are return times of  minimal distal $\Z^{d-1}$-systems.
\end{theorem}

\begin{proof}
Let $B$ be a subset of $\Z^d$ that contains a set of return times of a minimal distal 
$\Z^d$-system $(X,T_{1},\ldots,T_{d})$  with the unique closing parallelepiped property. 
Let $x\in X$ and $U$ an open neighborhood of $x$ such that $N_{T_{1},\ldots,T_{d}}(x,U)= \{(n_{1},\ldots,n_{d})\in \Z^{d}\colon\ T_{1}^{n_{1}}\cdots T_{d}^{n_{d}}x \in U\} \subseteq B$. 
We want to show that $B$ contains a $d$-joining of sets that are return times of  minimal distal $\Z^{d-1}$-systems.

Let $(Y,{T}_{1},\ldots,{T}_{d})$ be an extension of $(X,T_{1},\ldots,T_{d})$ as in Theorem \ref{StructThm}. Recall that $(Y,{T}_{1},\ldots,{T}_{d})$ is a joining of systems $(Y_j,{T}_{1},\ldots,{T}_{d})$, where for each $j\in [d]$ the action of $T_j$ on $Y_j$ is the identity. 

Let $(y_{1},\ldots,y_{d})\in Y$ be such that $\pi(y_{1},\ldots,y_{d})=x$ and let $\tilde{U}$ be a neighborhood of $(y_{1},\ldots,y_{d})$ in $Y$ such that $\pi(\tilde{U})\subseteq U$. We may assume that $\tilde{U}=(\tilde{U}_1\times\cdots\times \tilde{U}_d)\cap Y$, where for each $j\in[d]$ the set 
$\tilde{U}_j$ is an open neighborhood of $y_j$. We have,
\[{N_{{T}_{1},\ldots,{T}_{d}}((y_{1},\ldots,y_{d}),\tilde{U})=N_{T_{1},\ldots,T_{d}}((y_{1},\ldots,y_{d}),(\tilde{U}_1\times\cdots\times \tilde{U}_d)\cap Y)} =\bigcap_{j\in[d]} N_{T_{1},\ldots,T_{d}}(y_j,\tilde{U}_j). \]
Recall that on each $Y_{j}$ the action of $T_{j}$ is  the identity, so the action of $T_1,\ldots, T_d$ on $Y_{j}$ can be seen as a $\Z^{d-1}$ action and 
\begin{align*}
& (n_1,\ldots,n_d) \in N_{{T}_{1},\ldots,{T}_{d}}((y_{1},\ldots,y_{d}),\tilde{U}) \iff  \\
& (n_{1},\ldots,n_{j-1},n_{j+1},\ldots,n_{d}) \in N_{T_{1},\ldots,T_{j-1},T_{j+1},\ldots,T_{d}}(y_{j},\tilde{U}_j) \text{ for every } j\in[d]. \\
\end{align*}
Thus the sets $B_j=\{(n_{1},\ldots,n_{j-1},n_{j+1},\ldots,n_{d}) \in N_{T_{1},\ldots,T_{j-1},T_{j+1},\ldots,T_{d}}(y_{j},\tilde{U}_j) \}$ are return times of minimal $\Z^{d-1}$-systems whose $d$-joining coincides with $N_{{T}_{1},\ldots,{T}_{d}}((y_{1},\ldots,y_{d}),\tilde{U})$. 
Since $\pi(\tilde{U})\subseteq U$, we have that  ${N_{{T}_{1},\ldots,{T}_{d}}((y_{1},\ldots,y_{d}),\tilde{U})\subseteq N_{T_{1},\ldots,T_{d}}(x,U)}\subseteq B$, which serves to conclude. 

Conversely, assume that a set $B \subseteq \Z^d$ contains a $d$-joining of sets $B_j$, $j\in[d]$, where each $B_j$ is a set of return times of a minimal distal $\Z^{d-1}$-system $(Y_j,S_{j,1},\ldots,S_{j,d-1})$. We want to show that $B$ contains the set of return times of a minimal distal $\Z^d$-system with the unique closing parallelepiped property.  

For convenience, for each $j\in [d]$ we write $(Y_{j},T_{1},\ldots,T_{j-1},T_{j+1},T_{d})$ instead of 
$(Y_j,S_{j,1},\ldots,S_{j,d-1})$. By doing so, we stress the fact that we are viewing the $\Z^{d-1}$ action as a $\Z^{d}$ action where one of the transformations is the identity. We will see below that using the same set of transformations $T_1,\ldots,T_d$ for all systems will not be a notational problem since we will consider a product system.  
 
Let $y_{j}\in Y_{j}$ and let $U_{j}$ be an open neighborhood of $y_{j}$ such that
\begin{equation}\label{equation:Z^{d-1}-B_j}
N_{T_{1},\ldots,T_{j-1},T_{j+1},\ldots,T_{d}}(y_{j},U_{j})\subseteq B_{j}.
\end{equation}
We now construct a minimal distal $\Z^d$-system with the unique closing parallelepiped property and a set of return times for this system that is contained in $B$. Consider the product system $\prod\limits_{j=1}^{d}Y_{j}$ and the diagonal action of $T_1,\ldots, T_d$ on it.  For the  point $\textbf{y}=(y_1,\ldots,y_d)\in \prod\limits_{j=1}^{d} Y_j$ let $Y=\overline{\mathcal{O}(\textbf{y},T_{1},\ldots,T_{d})}$. The $\Z^d$-system $(Y,T_1,\ldots,T_d)$ is minimal and distal (see Theorem \ref{distal}), and contains the point $\textbf{y}=(y_{1},\ldots,y_{d})$. For any $j\in[d]$, the transformation $T_{j}$ acts trivially in the $j$-th coordinate of $Y$ and thus by Theorem \ref{StructThm} $(Y,T_1,\ldots,T_d)$ has the unique closing parallelepiped property. Now, consider the open neighborhood of $\textbf{y}$ given by  ${U=\left(\prod\limits_{j=1}^{d}U_{j}\right)\cap Y}$ and note that 
$(n_1,\ldots,n_d)\in N_{T_{1},\ldots,T_{d}}(\textbf{y},U)$ if and only if for every $j\in [d]$, $(n_{1},\ldots,n_{j-1},n_{j+1},\ldots,n_{d})\in N_{T_{1},\ldots,T_{j-1},T_{j+1},\ldots,T_{d}}(y_{j},U_{j})$.
That is, $N_{T_{1},\ldots,T_{d}}(\textbf{y},U)$ is the $d$-joining of the sets $N_{T_{1},\ldots,T_{j-1},T_{j+1},\ldots,T_{d}}(y_{j},U_{j})$. Using \eqref{equation:Z^{d-1}-B_j} we conclude that $N_{T_{1},\ldots,T_{d}}(\textbf{y},U)$ is contained in $B$. 
\end{proof}
	
We denote by $\mathcal{B}_{d}$ the family generated by sets of return times arising from minimal distal $\Z^d$-systems with the unique closing parallelepiped property and by $\mathcal{B}^{*}_{d}$ the (dual) family of subsets of $\Z^{d}$ which have nonempty intersection with every set in $\mathcal{B}_{d}$.

\begin{lemma} \label{SuperLifting}
Let $(X,T_{1},\ldots,T_{d})$ be a minimal distal $\Z^d$-system. Suppose that $(x,y)\in \mathcal{R}_{T_{1},\ldots,T_{d}}(X)$. Let $(Z,T_{1},\ldots,T_{d})$ be a minimal distal $\Z^d$-system with the unique closing parallelepiped property 
and let $(J,T_1,\ldots,T_d)$ be a joining between $(X,T_1,\ldots,T_d)$ and $(Z,T_1,\ldots,T_d)$. Then, for $z_0\in Z$ we have that $(x,z_0)\in J$ if and only if $(y,z_0)\in J$.
\end{lemma}

\begin{proof}
The proof is similar to the proof of Lemma 6.19 in \cite{donoso2014dynamical}, which is an adaptation of the corresponding statement in \cite{huang2016nil}. We provide it for completeness. Let $W=Z^{Z}$ and $T_{1}^{Z},\ldots, T_{d}^{Z}\colon W\rightarrow W$ be the corresponding commuting transformations. Let $\omega^{*}\in W$ be the point satisfying $\omega^*(z)=z$ for all $z\in Z$ and $Z_{\infty}=\overline{\mathcal{O}(\omega^{*},G^{Z})}$, where $G^{Z}$ is the group generated by $T_{1}^{Z},\ldots,T_{d}^{Z}$. Then, $Z_{\infty}$ is minimal and distal. So for any $\omega\in Z_{\infty}$ there exists $p\in E(Z,G)$ such that $\omega(z)=p\omega^{*}(z)=p(z)$ for any $z\in Z$. Since $(Z,T_{1},\ldots,T_{d})$ is minimal and distal, $E(Z,G)$ is a group and thus $p\colon Z\rightarrow Z$ is surjective. This implies that there exists $z_{\omega}\in Z$ such that $\omega(z_{\omega})=z_{0}$.

Take a minimal subsystem $(A,T_{1}\times T_{1}^{Z},\ldots,T_{d}\times T_{d}^{Z})$ of the product system $(X\times Z_{\infty},T_{1}\times T_{1}^{Z},\ldots,$ $T_{d}\times T_{d}^{Z})$. Let $\pi_{X}\colon A\to X$ be the natural coordinate projection. Then, $\pi_{X}$ is a factor map between two distal minimal systems. By Theorem \ref{teo2}, there exist $\omega^{1},\omega^{2}\in W$ such that $((x,\omega^{1}),(y,\omega^{2}))\in \mathcal{R}_{\hat{T}_{1},\ldots,\hat{T}_{d}}(A)$, where $\hat{T}_{j}=T_{j}\times T_{j}^{Z}$, for $j\in [d]$.

Let $z_{1}\in Z$ be such that $\omega^{1}(z_{1})=z_{0}$. Denote by $\pi\colon A\rightarrow X\times Z$, $\pi(u,\omega)=(u,\omega(z_{1}))$ for $(u,\omega)\in A$, $u\in X$ and $\omega\in W$. Consider the projection $B=\pi(A)$. Then, ${(B,T_{1}\times T_{1},\ldots,T_{d}\times T_{d})}$ is a minimal distal subsystem of $(X\times Z, T_{1}\times T_{1},\ldots,T_{d}\times T_{d})$ and since $\pi(x_0,\omega^{1})=(x,z_{0})\in B$ we have that $J$ contains $B$. Suppose that $\pi(x,\omega^{2})=(x,z_{2})$. Then, $((x,z_{0}),(y,z_{2}))\in \mathcal{R}_{T_{1}\times T_{1},\ldots,T_{d}\times T_{d}}(B)$ and we conclude that $(z_{0},z_{2})\in \mathcal{R}_{T_{1},\ldots,T_{d}}(Z)$. Since $\mathcal{R}_{T_{1},\ldots,T_{d}}(Z)=\Delta_{Z}$ we have that $z_{0}=z_{2}$ and thus $(y,z_0)\in B\subseteq J$.
\end{proof}

\begin{lemma} \label{RecRPST}
Let $(X,T_{1},\ldots,T_{d})$ be a minimal distal $\Z^d$-system. Then, for $x,y\in X$, $(x,y)\in \mathcal{R}_{T_{1},\ldots,T_{d}}(X)$ if and only if $N_{T_{1},\ldots,T_{d}}(x,U)\in\mathcal{B}^{*}_{d}$ for any open neighborhood $U$ of $y$.
\end{lemma}
\begin{proof}
The proof is similar to the one of Theorem 6.20 in \cite{donoso2014dynamical}.
 Suppose $N_{T_{1},\ldots,T_{d}}(x,U)\in\mathcal{B}^{*}_{T_{1},\ldots,T_{d}}$ for any open neighborhood $U$ of $y$. Since $(X,T_{1},\ldots,T_{d})$ is distal, $\mathcal{R}_{T_{1},\ldots,T_{d}}(X)$ is an equivalence relation. Let $\pi$ be the projection map $\pi\colon X\rightarrow Y= X/\mathcal{R}_{T_{1},\ldots,T_{d}}(X)$. By Corollary \ref{prop5} we have that $\mathcal{R}_{T_{1},\ldots,T_{d}}(Y)=\Delta_{Y}$. We also have that  the factor map $\pi$ is open and $\pi(U)$ is an open neighborhood of $\pi(y)$. In particular, $N_{T_{1},\ldots,T_{d}}(x,U)\subseteq N_{T_{1},\ldots,T_{d}}(\pi(x),\pi(U))$. Let $V$ be an open neighborhood of 
 $\pi(x)$. By hypothesis, we have that ${N_{T_{1},\ldots,T_{d}}(x,U)\cap N_{T_{1},\ldots,T_{d}}(\pi(x),\pi(U))\neq\emptyset}$, which implies that $N_{T_{1},\ldots,T_{d}}(\pi(x),\pi(U))\cap$ $ N_{T_{1},\ldots,T_{d}}(\pi(x),V)\neq \emptyset$. This implies that $\pi(U)\cap V\neq \emptyset$. But this holds for every $V$, so we have that $\pi(x)\in \overline{\pi(U)}=\pi(\overline{U})$. Finally, since this fact holds for every $U$ we conclude that $\pi(x)=\pi(y)$. This shows that $(x,y)\in \mathcal{R}_{T_{1},\ldots,T_{d}}(X)$ as desired.

Conversely, suppose that $(x,y)\in \mathcal{R}_{T_{1},\ldots,T_{d}}(X)$. Let $U$ be an open neighborhood of $y$ and $A$ be a $\mathcal{B}^{*}_d$ set. Then, there exists a minimal distal system $(Z,T_{1},\ldots,T_{d})$ with $\mathcal{R}_{T_{1},\ldots,T_{d}}(Z)=\Delta_Z$, an open set $V\subseteq Z$ and $z_0\in V$ such that $N_{T_{1},\ldots,T_{d}}(z_0,V)\subseteq A$. Let $J$ be the orbit closure of $(x,z_0)$ under $T_{j}\times T_{j}$ for $j\in [d]$. By distality we have that $(J,T_{1}\times T_{1},\ldots,T_{d}\times T_{d})$ is a minimal system and $(x,z_0)\in J$. By Theorem \ref{teo2} we have that $(y,z_0)\in J$ and particularly there exists a sequence $(\textbf{n}^i)_{i\in \N}\subseteq \Z^{d}$ such that ${(T_{1}^{n_{1}^{i}}\cdots T_{d}^{n_{d}^{i}}x,T_{1}^{n_{1}^{i}}\cdots T_{d}^{n_{d}^{i}}z_0)\to (y,z_0)}$. This implies that $N_{T_{1},\ldots,T_{d}}(x,U)\cap N_{T_{1},\ldots,T_{d}}(z_0,V)\neq \emptyset$ and the proof is achieved.
\end{proof}

We get the following characterization of the unique closing parallelepiped property for minimal distal $\Z^d$-systems.
	
\begin{corollary} \label{ReturnTProduct}
Let $(X,T_{1},\ldots,T_{d})$ be a minimal distal $\Z^d$-system. 
Then, $(X,T_{1},\ldots,T_{d})$ has the unique closing parallelepiped property if and only if for every $x\in X$ and every open neighborhood $U$ of $x$, $N_{T_{1},\ldots,T_{d}}(x,U)$ contains a $d$-joining built from $d$ sets of return times of minimal distal $\Z^{d-1}$-systems.
\end{corollary}
	
\begin{proof} The implication that $N_{T_{1},\ldots,T_{d}}(x,U)$ contains a $d$-joining of $d$ sets of return times for minimal distal $\Z^{d-1}$-systems whenever $(X,T_1,\ldots,T_d)$ has the unique closing parallelepiped property follows immediately from Theorem \ref{thm:returntimes}. Then 
we only need to prove the other implication. Let us suppose that there exists ${(x,y)\in \mathcal{R}_{T_{1},\ldots,T_{d}}(X)\setminus \Delta_X}$ and let $U,V$ be open neighborhoods of $x$ and $y$ respectively such that $U\cap V=\emptyset$.  By assumption $N_{T_{1},\ldots,T_{d}}(x,U)$ is a $\mathcal{B}_d$ set and by Lemma \ref{RecRPST} $N_{T_{1},\ldots,T_{d}}(x,V)$ has nonempty intersection with $N_{T_{1},\ldots,T_{d}}(x,U)$. This implies that $U\cap V\neq \emptyset$, a contradiction. We conclude that $\mathcal{R}_{T_{1},\ldots,T_{d}}(X)=\Delta_X$ and therefore $(X,T_{1},\ldots,T_{d})$ has the unique closing parallelepiped property.
\end{proof}

As an application we get the following criterion for a minimal distal system $(X,T)$ (so defined by a single transformation) for being topologically conjugate to an inverse limit of $d$-step nilsystems, \emph{i.e.},  to an inverse limit of systems that can be written as $X=G/\Gamma$, where $G$ is a $d$-step nilpotent Lie group and $\Gamma$ is a cocompact subgroup of $G$, and $T$ is a left translation by a fixed element in $G$. 
For an integer $d\geq 1$ let $\phi_{d}\colon \mathbb{Z}^d\to \mathbb{Z}$ be the morphism $(n_1,\ldots,n_d)\mapsto n_1+\cdots+n_d$.  

\begin{theorem}
Let $(X,T)$ be a minimal distal system. Then, $(X,T)$ is topologically conjugate to an inverse limit of 
$d$-step nilsystems if and only if for all $x\in X$ and $U$ an open set containing $x$ the set of return times $N_T(x,U)=\{n\in \mathbb{Z}: T^nx\in U\}$ contains the image under $\phi_{d+1}$ of a $(d+1)$-joining of $(d+1)$ sets of return times for distal $\Z^{d}$-systems.  
\end{theorem}
\begin{proof}
We regard $(X,T)$ as the $\Z^{d+1}$-system $(X,T,\ldots,T)$. Then, we have that $(X,T,\ldots,T)$ has the unique closing parallelepiped property if and only if for all $x\in X$ and all neighborhood  $U$ of $x$ the set $\{(n_1,\ldots,n_{d+1})\in \Z^{d+1} : T^{n_1}T^{n_2}\cdots T^{n_{d+1}}x \in U  \}\subseteq \Z^{d+1} $ contains a $(d+1)$-joining of $(d+1)$ sets of return times for $\Z^{d}$-distal systems. But, since all transformations are the same, 
$\{n\in \mathbb{Z}: T^nx\in U\}$ contains the image under $\phi_{d+1}$ of $\{(n_1,\ldots,n_{d+1})\in \Z^{d+1} : T^{n_1}T^{n_2}\cdots T^{n_{d+1}}x \in U  \}$.  
On the other hand, $(X,T,\ldots,T)$ has the unique closing parallelepiped property if and only if it is isomorphic to an inverse limit of $d$-step nilsystems \cite[Theorem 1.2]{host2010nilsequences}. This  finishes the proof. 
\end{proof}

\section{Examples of systems with the unique closing parallelepiped property}\label{sec:Examples}

\quad In this section we provide a family of examples of systems with the unique closing parallelepiped property.

\subsection{Affine transformations in the torus.}

Let $r\geq 1$ be an integer. Consider different affine transformations 
$T_{i}\colon\mathbb{T}^{r}\to \mathbb{T}^{r}$, $x\mapsto A_{i}x+\alpha_{i}$, where $A_{i}$ is an unipotent integer matrix (\emph{i.e.}, $(A_i-I)^{p}=0$ for some $p\in \N$) and $\alpha_i \in \mathbb{T}^{r}$ for every $i \in \{1,\ldots,d\}$. The $\Z^d$-system $(\mathbb{T}^r,T_1,\ldots, T_d)$ can be seen as a nilsystem as long as the matrices commute (we will not give all details on this fact but ideas can be found in  \cite{parry1969ergodic}). 

Let $G$ be the group of transformations of $\mathbb{T}^{r}$ generated by the matrices $A_{1},\ldots,A_{d}$ and the translations of $\T^{r}$. Then, every element $g\in G$ is a map 
$x \mapsto A(g) x+ \beta(g)$, where $A(g)=A_{1}^{m_{1}}\cdots A_{d}^{m_{d}}$ with $m_{1},\ldots,m_{d}\in \Z$ and $\beta(g)\in \mathbb{T}^{r}$. A simple computation shows that if $g_1,g_2\in G$ then the commutator $[g_1,g_2]$ is the map ${x\mapsto x + (A(g_1)-I)\beta(g_2) - (A(g_2)-I)\beta(g_1)}$ and thus it is a translation of $\mathbb{T}^r$. On the other hand, if $g\in G$ and $\beta \in \mathbb{T}^r$, then $[g,\beta]$ is the translation $x\mapsto x+(A(g)-I)\beta$. It follows that if $g_1,\ldots,g_k \in G$ then the iterated commutator $[\cdots[[g_1,g_2],g_3]\cdots g_k]$ belongs to $\mathbb{T}^r$ and is contained in the image of  $(A(g_3)-I)\cdots(A(g_k)-I)$. If $k$ is large enough, this product is trivial. So $G$ is a nilpotent Lie group. The torus $\T^{r}$ can be identified with ${G}/{\Gamma}$, where $\Gamma$ is the stabilizer of 0, which is the group generated by the matrices $A_{1},\ldots,A_{d}$. We refer to $(\T^{r},T_{1},\ldots,T_{d})$ as an \emph{affine nilsystem} with $d$ transformations. It is worth noting that the transformations $T_i$ and $T_j$ commute if and only if $(A_i-I)\alpha_j= (A_{j}-I)\alpha_i$ in $\mathbb{T}^r$. 

From results in \cite{leibman2005pointwise} one can deduce that, 

\begin{proposition}\label{Property:MinimalErgodic}
	Let $(\mathbb{T}^r,T_1,\ldots,T_d)$ be an affine nilsystem with $d$ transformations. Then, the properties  of transitivity, minimality, ergodicity and unique ergodicity under the action of $\langle T_1,\ldots,T_d \rangle$ are equivalent. 
\end{proposition}

We consider some conditions on the commuting transformations $T_{1},\ldots,T_{d}$ under which the system $(\mathbb{T}^r,T_1,\ldots,T_d)$ has the unique closing parallelepiped property. We start by presenting the examples whose elementary proofs are deliberately omitted. Interested readers can find them in \cite{cabezas2018}. For the sake of clarity we first consider the case $d=2$.

\begin{lemma}\label{Affine2}
	Let $(\mathbb{T}^r,T_1,T_2)$ be an affine nilsystem with 2 commuting transformations, where $T_{i}x=A_{i}x+\alpha_{i}$ for $i\in \{1,2\}$. Then, we have that for every $n,m\in \Z$,
	$$T_{1}^{n}T_{2}^{m}x=T_{1}^{n}x+T_{2}^{m}x-x$$
	
	\noindent if and only if the following conditions hold
	\begin{equation}
	(A_{1}-I)(A_{2}-I)=0,
	\label{MatCond1}
	\end{equation}
	\begin{equation}
	(A_{1}-I)\alpha_{2}=(A_{2}-I)\alpha_{1}=0.
	\label{MatCond2}
	\end{equation}
In particular, if conditions \eqref{MatCond1} and \eqref{MatCond2} are satisfied we have that
	$$\QQ_{T_{1},T_{2}}(X)=\overline{\{(x,T_{1}^{n}x,T_{2}^{m}x,T_{1}^{n}x+T_{2}^{m}x-x)\colon x \in \T^{r},\ n,m\in \Z\}},$$
	
	\noindent and thus $(\mathbb{T}^r,T_1,T_2)$ has the unique closing parallelepiped property.
\end{lemma}

\begin{example}
	Consider the following matrices,
	\smallskip
	
	{\footnotesize $$A_{1}=\left(\begin{array}{cccccc}
	1 & 0 & 0 & 1 & 0 & 2\\
	0 & 1 & 0 & 3 & 1 & 4\\
	0 & 0 & 1 & 6 & 3 & 6\\
	0 & 0 & 0 & 1 & 0 & 0\\
	0 & 0 & 0 & 0 & 1 & 0\\
	0 & 0 & 0 & 0 & 0 & 1
	\end{array}\right),\quad A_{2}=\left(\begin{array}{cccccc}
	1 & 0 & 0 & 1 & 1 & 2\\
	0 & 1 & 0 & 2 & 2 & 4\\
	0 & 0 & 1 & 1 & 2 & 3\\
	0 & 0 & 0 & 1 & 0 & 0\\
	0 & 0 & 0 & 0 & 1 & 0\\
	0 & 0 & 0 & 0 & 0 & 1
	\end{array}\right).$$}
	
	For the first matrix, the eigenspace (associated to the unique eigenvalue $1$) is given by, 
	{\footnotesize $$W_{1}(A_{1})=\left\langle\left\{\left(\begin{array}{cccccc}
	1 \\
	0 \\
	0 \\
	0 \\
	0 \\
	0 
	\end{array}\right),\left(\begin{array}{cccccc}
	0 \\
	1 \\
	0 \\
	0 \\
	0 \\
	0 
	\end{array}\right),\left(\begin{array}{cccccc}
	0 \\
	0 \\
	1 \\
	0 \\
	0 \\
	0 
	\end{array}\right),\left(\begin{array}{cccccc}
	0 \\
	0 \\
	0 \\
	-2 \\
	2 \\
	1 
	\end{array}\right)\right\}\right\rangle.$$}
	
	For the second one, the eigenspace is given by,

{\footnotesize	$$W_{1}(A_{2})=\left\langle\left\{\left(\begin{array}{cccccc}
	1 \\
	0 \\
	0 \\
	0 \\
	0 \\
	0 
	\end{array}\right),\left(\begin{array}{cccccc}
	0 \\
	1 \\
	0 \\
	0 \\
	0 \\
	0 
	\end{array}\right),\left(\begin{array}{cccccc}
	0 \\
	0 \\
	1 \\
	0 \\
	0 \\
	0 
	\end{array}\right),\left(\begin{array}{cccccc}
	0 \\
	0 \\
	0 \\
	-1 \\
	-1 \\
	1 
	\end{array}\right)\right\}\right\rangle.$$}

It is easy to see that $(A_{1}-I)(A_{2}-I)=0$ and we can choose $\alpha_{1}\in W_{1}(A_{2})$ and $\alpha_{2}\in W_{1}(A_{1})$ such that $(\T^{6},T_{1},T_{2})$ has the unique closing parallelepiped property.
\end{example}

Conditions \eqref{MatCond1} and \eqref{MatCond2} can be generalized as follows. Again, the simple but tedious proof of this result can be found in \cite{cabezas2018}.

\begin{lemma}\label{Affined}
	Let $(\mathbb{T}^r,T_1,\ldots,T_{d})$ be an affine nilsystem with $d$ commuting transformations such that $T_{i}x=A_{i}x+\alpha_{i}$ for $i\in \{1,\ldots,d\}$. Then, for every $(n_{1},\ldots,n_{d})\in \Z^{d}$,
	$$T_{1}^{n_{1}}\cdots T_{d}^{n_{d}}x = (-1)^{d}\sum\limits_{i=0}^{d-1}(-1)^{i+1}\sum\limits_{\substack{I\subseteq [d]\\ |I|=i}}\mathop{\bigcirc}\limits_{k\in I}T_{k}^{n_{k}}x$$
	\noindent if and only if the following conditions hold
	\begin{equation}
	\prod\limits_{i=1}^{d}(A_{i}-I)=0,
	\label{MatCond3}
	\end{equation}
	\begin{equation}
	\forall j\in \{1,\ldots,d\},\ \prod\limits_{\substack{i=1\\ i\neq j}}^{d}(A_{i}-I)\alpha_{j}=0.
	\label{MatCond4}
	\end{equation}
	In particular, if conditions \eqref{MatCond3} and \eqref{MatCond4} hold the system $(\mathbb{T}^r,T_1,\ldots,T_{d})$ has the unique closing parallelepiped property.	
\end{lemma}


\end{document}